\documentclass[10pt,a4paper]{amsart}
\usepackage{verbatim}
\usepackage{tikz}
\usetikzlibrary{patterns}
\usepackage[toc]{appendix}
\usepackage[T1]{fontenc}

\usepackage{graphicx}
\usepackage{enumerate}
\usepackage{amsmath,amsfonts,amssymb}
\usepackage{color}
\def\loc{\operatorname{loc}}
\usepackage{cite}
\usepackage{ latexsym }
\definecolor{citation}{rgb}{0.11,0.67,0.84}
\definecolor{formula}{rgb}{0.1,0.2,0.6}
\definecolor{url}{rgb}{0.11,0.67,0.84}
\usepackage{pgf,tikz}
\usepackage{mathrsfs}
\usepackage{fancyhdr}
\usepackage{dutchcal}

\newcommand{\medint}{-\kern -,375cm\int}

\newcommand{\medintinrigo}{-\kern -,315cm\int}
\makeatletter
\newcommand{\linethrough}{\mathpalette\@thickbar}
\newcommand{\@thickbar}[2]{{#1\mkern0mu\vbox{
    \sbox\z@{$#1#2\mkern-0.5mu$}%
    \dimen@=\dimexpr\ht\tw@-\ht\z@+2\p@\relax 
    \hrule\@height0.5\p@ 
    \vskip\dimen@
    \box\z@}}
}
\makeatother

\newcommand{\mathstrike}[1]{\ensuremath{\linethrough{#1}}}

\newcommand{\nra}[1]{\mathstrike{\lVert} #1 \rVert}
\newcommand{\snra}[1]{\mathstrike{[} #1 ]}

\newtheorem{theorem}{Theorem}[section]
\newtheorem{lemma}[theorem]{Lemma}
\newtheorem{proposition}[theorem]{Proposition}
\newtheorem{corollary}[theorem]{Corollary}
\newtheorem{definition}[theorem]{Definition}
\theoremstyle{definition}
\newtheorem{remark}[theorem]{Remark}

\numberwithin{equation}{section}

\usepackage{dsfont}

\newcommand{\reqnomode}{\tagsleft@false}

\usepackage{hyperref}

\newcommand\ccc{\textnormal{c}}

\def\dx{\,{\rm d}x}

\def\ds{\,{\rm d}s}
\def\dt{\,{\rm d}t}
\def\dy{\,{\rm d}y}

\def \d{\,{\rm d}}
\def \diver{\,{\rm div}}
\def\dist{\,{\rm dist}}

\def\supp{\,{\rm supp}}
\def\diam{\,{\rm diam}}

\def\kk{\,\kappa}

\allowdisplaybreaks
\makeatletter
\DeclareRobustCommand*{\bfseries}{%
  \not@math@alphabet\bfseries\mathbf
  \fontseries\bfdefault\selectfont
  \boldmath
}

\DeclareMathOperator*{\osc}{osc}

\makeatother

\newlength{\defbaselineskip}
\newcommand{\mint}{\mathop{\int\hskip -1,05em -\, \!\!\!}\nolimits}

\def \diver{\,{\rm div}}

\newcommand{\ti}[1]{\tilde{#1}}

\newcommand{\rrr}{\textnormal{\texttt{r}}}

\newcommand{\rr}{\varrho}
\newcommand{\snr}[1]{\lvert #1\rvert}
\newcommand{\nr}[1]{\lVert #1 \rVert}

\newcommand{\tx}[1]{\textnormal{\texttt{#1}}}

\def\loc{\operatorname{loc}}

\def\eqn#1$$#2$${\begin{equation}\label#1#2\end{equation}}

\newcommand{\data}{\textnormal{\texttt{data}}}

\def\supp{\,{\rm supp }}

\def\XXint#1#2#3{{\setbox0=\hbox{$#1{#2#3}{\int}$}
     \vcenter{\hbox{$#2#3$}}\kern-.5\wd0}}

\title{Intrinsic Schauder estimates at nearly linear growth}

\author[De Filippis]{Cristiana De Filippis}  \address{Cristiana De Filippis\\Dipartimento SMFI, Universit\`a di Parma, Viale delle Scienze 53/a, Campus, 43124 Parma, Italy} \email{\url{cristiana.defilippis@unipr.it}}
\author[De Filippis]{Filomena De Filippis}  \address{Filomena De Filippis\\Fachbereich Mathematik, Universität Salzburg, Hellbrunner Str. 34, 5020 Salzburg, Austria} \email{\url{filomena.defilippis@plus.ac.at}}
\author[H\"ast\"o]{Peter H\"ast\"o}  \address{Peter H\"ast\"o\\Department of Mathematics and Statistics, University of Helsinki, FI-00014, Helsinki, Finland} \email{\url{peter.hasto@helsinki.fi}}
\begin{document}

\subjclass[2020]{49N60, 35J60 \vspace{1mm}} 

\keywords{Regularity, $(p,q)$-growth, Nonuniform ellipticity, $\mu$-ellipticity\vspace{1mm}}

\thanks{{\it Acknowledgements.} C. De Filippis is supported by the European Research Council, through the ERC StG project NEW, nr.~101220121, and by the University of Parma through the action "Bando di Ateneo 2024 per la ricerca". This research was funded in whole or in part by the Austrian Science Fund (FWF) [10.55776/PAT1850524]. For open access purposes, the author has applied a CC BY public copyright license to any author accepted manuscript version arising from this submission.
\vspace{1mm}}

\maketitle

\begin{abstract}
We develop a nonlinear potential theoretic framework for Schauder estimates for vector-valued solutions of a broad class of nonautonomous variational problems at nearly linear growth. Our approach naturally embraces the variable exponent as well as the Double and Multi phase setting, yielding new regularity results in basic models and recovering optimal regularity recently established in specific cases.
 \end{abstract}
\setcounter{tocdepth}{1}
{\small \tableofcontents}
\section{Introduction}\label{intro}
\noindent Schauder estimates are fundamental tools in the Analysis of PDEs and in the Calculus of Variations, with applications to free boundary problems, evolutionary PDEs, and global regularity. The smoothing effect that (autonomous) elliptic operators have on solutions is the cornerstone of regularity theory. However, plugging in ingredients like forcing or transport terms, or space-depending coefficients, might inhibit the regularizing process. Think e.g.\ of the Laplace operator. By Weil's lemma, harmonic maps\footnote{Solutions to $\Delta u=0$.} are smooth --- in particular, the regularity of solutions self-improves. The presence of external ingredients is a clear obstruction to this mechanism, so it is natural to wonder how much of such regularity survives when coefficients are plugged in. Schauder theory answers this basic question. Specifically, given a bounded, elliptic matrix $\tx{A}\colon \Omega\to \mathbb{R}^{n\times n}$, $\tx{A}\approx \mathds{I}$, the natural guess for weak energy solutions to $\diver(\tx{A}(x)Du)=0$ is
\eqn{i.0}
$$
\tx{A}\in C^{0,\alpha}_{\loc}(\Omega,\mathbb{R}^{n\times n}) \ \Longrightarrow \ u\in C^{1,\alpha}_{\loc}(\Omega),\qquad \mbox{with} \ \ \alpha\in (0,1).
$$
Linear Schauder theory dates back to the end of the '20's of the previous century, with the classical results of Hopf, Caccioppoli, and Schauder, heavily relying on potential theory --- techniques later on streamlined by Campanato (via suitable function spaces), L. Simon (by means of blow up methods), and Trudinger (using convolution arguments). The key aspect emerging from these works is the perturbative nature of Schauder estimates for linear, uniformly elliptic equations. Analogous phenomena hold for nonautonomous, quasilinear elliptic PDEs\footnote{Weak solutions to quasilinear elliptic PDEs with H\"older continuous coefficients have H\"older continuous gradients, but implication \eqref{i.0} may fail for degenerate operators, see \cite{ura68}.} of the type $\diver(\tx{c}(x)\snr{Du}^{p-2}Du)=0$, with $1\lesssim \tx{c}(\cdot)\in C^{0,\alpha}_{\loc}(\Omega)$ and $1<p<\infty$. This is due to Manfredi \cite{man88}, after DiBenedetto \cite{dib83}, and Giaquinta \& Giusti \cite{gg82}, based on the fundamental contributions of Ural'tseva \cite{ura68} and Uhlenbeck \cite{uhl77}. The unifying feature of the equations and functionals treated in the aforementioned works lays in the uniform boundedness of the related ellipticity ratio. Specifically, for nonlinear elliptic PDEs of the type\footnote{Assume that field $\tx{B}$ is smooth and $z\mapsto \partial \tx{B}(\cdot,z)$ is symmetric for simplicity.} $\diver(\tx{B}(x,Du))=0$, the ellipticity ratio is defined as 
\eqn{er}
$$
\mathcal{R}(x,z):=\frac{\mbox{highest eigenvalue of }\partial \tx{B}(x,z)}{\mbox{lowest eigenvalue of }\partial \tx{B}(x,z)},
$$
and measures how the gradient variable affects the growth/ellipticity properties of field $\tx{B}$. The extension of \eqref{er} to functionals as in \eqref{fun} goes through the Euler-Lagrange equation, solved by minima. In fact, replacing $\tx{B}$ with the derivative in the gradient variable of a (suitably regular, strictly convex) integrand $\tx{f}$ in \eqref{er}, we are led to consider the pointwise ratio between the highest and the lowest eigenvalue of the Hessian matrix $\partial^{2}\tx{f}$. Equations or functionals characterized by a uniformly bounded ellipticity ratio are classified as uniformly elliptic, and Schauder theory is well-established, with three key aspects worth highlighting.
\begin{itemize}
    \item Schauder theory for uniformly elliptic problems always holds, in the sense that to H\"older continuous coefficients always correspond solutions with H\"older continuous gradient.
    \item The smoother the better: in nondegenerate, uniformly elliptic problems, smooth ingredients grant smooth solutions.
    \item Uniformly elliptic Schauder estimates are achieved via perturbation arguments.
\end{itemize}
The above paradigm dramatically fails in the nonuniformly elliptic setting. Nonuniformly elliptic equations or functionals feature unbounded ellipticity ratios, which blow up, in several significant cases, as a positive power of the gradient variable. This class of nonuniformly elliptic PDEs is extremely rich. The most celebrated model is the nonparametric area integral
$$
w\mapsto \int_{\Omega}\sqrt{1+\snr{Dw}^{2}}\dx,
$$
featuring quadratic ellipticity ratio, see Bombieri, De Giorgi \& Miranda \cite{bdm69}, Ladyzhenskaya \& Ural'tseva \cite{lu70}, Trudinger \cite{tru72}, and L. Simon \cite{sim76}, for relevant regularity theory. Anisotropic energies \cite{mar89,mar91,bb20},
$$
w\mapsto \int_{\Omega}\snr{Dw}^{p}+\sum_{i=1}^{n}\snr{\partial_{i}w}^{q_{i}}\dx,\qquad\quad  1<p\le q_{1}\le \cdots\le q_{n}<\infty,
$$
and slow-growing integrals \cite{fm00,bf01},
$$
w\mapsto \int_{\Omega}\snr{Dw}\log(1+\snr{Dw})\dx,
$$
feature power-type nonuniformity as well. The key idea in this setting is to link the growth/ellipticity features of the nonlinear vector field $\tx{B}$ to the rate of blow-up of the ellipticity ratio \eqref{er}. Indeed, by imposing that the lowest and the highest eigenvalues of $\partial \tx{B}$ behave as two different powers of the gradient variable:
\eqn{pq}
$$
\begin{array}{c}
\displaystyle
\snr{z}^{p-2}\mathds{I}\lesssim \partial\tx{B}(x,z)\lesssim \snr{z}^{q-2}\mathds{I} \ \ \mbox{in the sense of forms,}\\[8pt]\displaystyle
\mbox{for some} \ \ 1<p<q<\infty \ \ \mbox{and all} \ \ \snr{z}\ge 1,\ x\in \Omega,
\end{array}
$$
the behavior at infinity of the related ellipticity ratio can be controlled in terms of the difference $q-p$,
\eqn{rpq}
$$
\mathcal{R}(x,z)\approx_{\snr{z}\ge 1}\snr{z}^{q-p}.
$$
The growth rate of $\mathcal{R}$ can then be reduced by choosing $q-p$ sufficiently close to zero, so the asymptotic in \eqref{rpq} suggests that a moderate blow-up rate of $\mathcal{R}$ grants hope for regular solutions. Marcellini, Giaquinta and Hong proved that a restriction of the type $q-p\lesssim 1/n$ is both necessary \cite{gia87,mar91,hon92} and sufficient \cite{mar89,mar91} condition for regularity. Since these breakthroughs, autonomous nonuniformly elliptic theory flourished, see \cite{bm20,bs20,bb20,hs21,bs24,sch24} in the scalar setting, \cite{sch09,sch21,bdms22,def22,dkk24,gk24} in the vectorial one for an (incomplete) list of advances, and \cite{mar21,min24,dm25b} for general overviews. Most notably, the construction in \cite{hon92} provides a strongly convex, nonuniformly elliptic, functional with too large nonuniformity rate and unbounded scalar minima. This novel fact breaks the uniformly elliptic orthodoxy, as it points out that regularity should not always be expected for nonuniformly elliptic problems already in the autonomous, strongly convex setting, and suggests that Schauder theory might fail already in the most simple situations. In this respect, the first seminal results date back to the end of the '60's, due to Trudinger \cite{tru67}, Ladyzhenskaya \& Ural'tseva \cite{lu70}, L. Simon \cite{sim76}, take also \cite{iva84} as a general reference. In these foundational works, the use of strong solutions, the total differentiability of the equation, and the smoothness of ingredients was unavoidable, thus escaping the classical Schauder framework, which prescribes minimal regularity (H\"older continuity) of coefficients,\footnote{As several nonuniformly elliptic models feature $p$-Laplacian type degeneracy in zero, see e.g.~\cite{sch14,bs15}, the expected maximal regularity is gradient H\"older continuity of solutions \cite{ura68}, therefore it is useless to impose more regularity than H\"older continuity on coefficients.} and the validity of nonuniformly elliptic Schauder theory remained open. New impulses eventually came from Zhikov \cite{zhi86,jko94,zhi95}, who introduced novel, elementary models like the $p(x)$-Laplacian
\eqn{px.2}
$$
w\mapsto \int_{\Omega}\snr{Dw}^{p(x)}\dx,\qquad\quad  1<p(\cdot)\in L^{\infty}(\Omega),
$$
and the Double Phase energy
\eqn{dp.2}
$$
\begin{cases}
    \displaystyle
    \ w\mapsto \int_{\Omega}\snr{Dw}^{p}+a(x)\snr{Dw}^{q}\dx\vspace{1.5mm}\\
    \displaystyle
    \ 0\le a(\cdot)\in C^{0,\alpha}(\Omega), \ \ \alpha\in (0,1]\vspace{1.5mm}\\
    \displaystyle
    \ 1<p\le q<\infty,
\end{cases}
$$
in the setting of homogenization and to study the possible occurrence of Lavrentiev phenomenon. Functionals \eqref{px.2}--\eqref{dp.2} are uniformly elliptic as their ellipticity ratio \eqref{er} stays uniformly bounded. However, a very mild amount of nonuniformity emerges due to the (mis)behavior of coefficients, see \cite[Section 4.6]{dm21}, which disclose a new, sharp phenomenology for Schauder theory to hold. In fact, building on early 2-d examples of Zhikov \cite{zhi86}, Fonseca, Mal\'y \& Mingione \cite{fmm04} proved that as soon as 
\eqn{pqn}
$$
1<p<n<n+\alpha<q<\infty \ \Longrightarrow \ \frac{q}{p}>1+\frac{\alpha}{n},
$$
the functional \eqref{dp.2} admits minima whose set of essential discontinuity points has almost maximal Hausdorff dimension, while, subject again to \eqref{pqn}, Esposito, Leonetti \& Mingione \cite{elm04} exihibited a minimizer of \eqref{dp.2} with a one point singularity preventing its membership of $W^{1,q}_{\loc}$, thus implying the failure of Schauder theory. The examples in \cite{zhi86,fmm04,elm04} were eventually perfected by Balci, Diening \& Surnachev \cite{bds20,bds25}, who constructed minima of \eqref{dp.2} that cannot be better than $W^{1,p}$-regular, and the Schauder paradigm outlined below \eqref{er} starts crumbling. In fact, a closer inspection of \cite{zhi86,elm04,fmm04,bds20,bds25} reveals that the same construction works if in \eqref{dp.2} the coefficient $a$ is $C^{\alpha}$-regular with any $\alpha>0$, and the integrand is nondegenerate. This means that, despite the H\"older continuity of the ingredients, Schauder theory may not always hold now. In particular, even for nondegenerate integrands with smooth coefficients, minima need not be smooth, therefore two out of three distinctive features of classical Schauder theory dramatically fail, already in presence of very weak nonuniformity types as those in \eqref{px.2}--\eqref{dp.2} (keep in mind that both models are pointwise uniformly elliptic). In the case of \eqref{px.2}--\eqref{dp.2}, only the perturbative approach to gradient regularity survives, thanks to the pointiwise uniform ellipticity of the governing integrands, cf. \cite{am01,cm15}. Further extensions appear in \cite{bo20,rt20,bs21,buli,bos22,bar23,dl24,fsv24,bar25,adp25,bb90}, and a unifying (perturbative) approach to the maximal regularity of a general class of models sharing analogous weak nonuniformity as in \eqref{px.2}--\eqref{dp.2} can be found in \cite{ho22a,ho22b,ho23}, by the third author and Ok. However, owing to the lack of homogeneous reference estimates, perturbations are no longer feasible in the genuinely nonuniformly elliptic setting, i.e., when the ellipticity ratio \eqref{er} blows up. The longstanding\footnote{See the comments in \cite[page 1111 (footnote)]{dm23a}.} problem of establishing optimal Schauder theory in the nonuniformly elliptic setting was settled by the first author and Mingione \cite{dm23a,dm25,dm25b}, by designing novel nonlinear potential theoretic techniques that allow bypassing the structural obstructions generated by polynomial nonuniformity and yield maximal regularity results within the sharp nonuniformity range
\eqn{qpqp}
$$
\frac{q}{p}<1+\frac{\alpha}{n},
$$
cf. \eqref{pqn}. The approach of \cite{dm23a,dm25} can be further expanded to push Schauder theory as close as possible to linear growth. The regularity of minima of nonautonomous variational integrals at linear growth is notoriously very delicate. In fact, Giaquinta, Modica \& Souček \cite{gms79} constructed (remarkably, one-dimensional) examples showing that nonautonomous, area-type functionals with almost $C^{2}$-regular coefficients might admit minimizers with jump discontinuities, which therefore belong to $BV\setminus W^{1,1}$. On the other hand, $C^{2}$-regular coefficients guarantee (at least) $C^{1}$-regular minima, cf. \cite{lu70}. Area-type integrals are specific instances of $\mu$-elliptic problems, a degenerate form of nonuniform ellipticity, typical of functionals at linear or nearly linear growth. Being the limiting configurations between linear and power growth, the latter class of models is rather common in materials science: the theories of Prandtl-Eyring fluids and of plastic materials with logarithmic hardening are prominent instances of applications cf. Frehse \& Seregin \cite{fs99}. We further refer to Fuchs \& Mingione \cite{fm00}, Bildhauer \& Fuchs \cite{bf01}, Schmidt \cite{sch14}, Beck \& Schmidt \cite{bs13,bs15}, Beck \& Gmeineder \& Sch\"affner \cite{bgs25}, Gmeineder \cite{gme20,gme21}, and Gmeineder \& Kristensen \cite{gk19a,gk19b,gk24} on deep regularity results for minima of general $\mu$-elliptic functionals, and \cite{def25} for an overview. Schauder-type results can be achieved for problems at nearly linear growth. In fact, the first author and Mingione \cite{dm23b} provided the first set of Schauder estimates for functionals at nearly linear growth as
$$
w\mapsto \int_{\Omega}\tx{c}(x)\snr{Dw}\log(1+\snr{Dw})\dx,\qquad \quad 1\lesssim \tx{c}(\cdot)\in C^{0,\beta}(\Omega),\ \ \beta\in (0,1],
$$
and Log-Double Phase models
\eqn{dp}
$$
\begin{cases}
    \displaystyle
    \ w\mapsto\int_{\Omega}\snr{Dw}\log(1+\snr{Dw})+a(x)\snr{Dw}^{q}\dx\vspace{1.5mm}\\ 
    \displaystyle
    \ 0\le a(\cdot)\in C^{0,\alpha}(\Omega),\ \ \alpha\in (0,1],\qquad \quad 1<q<\infty,
\end{cases}
$$
subject to
\eqn{aq1}
$$
q<1+\frac{\alpha}{n},
$$
coherently with \eqref{qpqp},\footnote{Formally, take $p=1$ there.} see also \cite{dp24} for the Log-Multi Phase case, \eqref{mp} below. Later on, the first two authors and Piccinini \cite{ddp24} obtained sharp regularity for bounded minima of \eqref{dp} and highlighted via fractal counterexamples that nearly linear growing functionals may be the limiting configurations in which $\mu$-ellipticity, convex anisotropies and H\"older continuous coefficients can coexist, see \cite[Theorem 4 and Section 1.1]{ddp24} and \cite[Section 5.2]{def25}. This motivates our analysis. In fact, aim of this paper is to provide a comprehensive Schauder theory for vector-valued solutions to a large class of nonautonomous variational problems at nearly linear growth of the type
\eqn{fun}
$$
W^{1,1}_{\loc}(\Omega,\mathbb{R}^{N})\ni w\mapsto \mathcal{F}(w;\Omega):=\int_{\Omega}\tx{f}(x,Dw)\dx.
$$
Here, $\Omega\subset \mathbb{R}^{n}$ is an open subset, with $n\ge 2$, $N\ge 1$. We assume the radial (Uhlenbeck \cite{uhl77}) structure condition $\tx{f}(x,z):=A(x,\snr{z})$, cf. Section \ref{sa}, which is fundamental to achieve full regularity in the multidimensional setting, \cite{sy02}. The notion of (local) minimizer we shall use is the standard one.
\begin{definition}
 A function $u\in W^{1,1}_{\loc}(\Omega,\mathbb{R}^{N})$ is a (local) minimizer of functional $\mathcal{F}$ in \eqref{fun} if for every ball $B\Subset \Omega$, $\tx{f}(\cdot,Du)\in L^{1}(B)$ and $\mathcal{F}(u;B)\le \mathcal{F}(w;B)$ for all $w\in u+W^{1,1}_{0}(B,\mathbb{R}^{N})$.  
\end{definition}
\noindent We indeed develop a nonlinear potential theoretic framework to attack the regularity of a general class of anisotropic variational integrals at nearly linear growth in the vectorial setting. Specifically, being tailored on the Double Phase structure of \eqref{dp}, the methodology in \cite{dm23b,dp24} fails to cover very natural models, for which we deliver Schauder theory. For instance, we are able to treat variable exponent functionals like
\eqn{px}
$$
w\mapsto \int_{\Omega}\left(\snr{Dw}\log(1+\snr{Dw})\right)^{p(x)}\dx,
$$
where the (H\"older continuous) exponent $p$ is now allowed to attain one, a scenario previously forbidden in the literature, cf. \cite{dhhr11}. Our result in this respect reads as follows.
\begin{theorem}\label{mt.1}
 Let $u\in W^{1,1}_{\loc}(\Omega,\mathbb{R}^{N})$ be a local minimizer of \eqref{px}, with exponent 
 \eqn{px.11}
 $$
 1\le p(\cdot)\in C^{0,\alpha}(\Omega), \qquad \alpha\in (0,1].
 $$
 Then $Du$ is locally H\"older continuous.
\end{theorem}
\noindent We furthermore offer the vectorial parallel of the results in \cite{dm23b}, which naturally extends to functionals with multiple phases \cite{dp24}.
\begin{theorem}\label{mt.1.5}
Let $u\in W^{1,1}_{\loc}(\Omega,\mathbb{R}^{N})$ be a local minimizer of \eqref{dp}, subject to \eqref{aq1}. Then $Du$ is locally H\"older continuous.
\end{theorem}
\noindent Another application of our results covers integral
\eqn{dppx}
$$
\begin{cases}
\displaystyle
\  w\mapsto \int_{\Omega}\left(\snr{Dw}\log(1+\snr{Dw})\right)^{p(x)}+a(x)\snr{Dw}^{q(x)}\dx\vspace{1.5mm}\\
 \displaystyle
\ 0\le a(\cdot)\in C^{0,\alpha}, \ \ \alpha\in (0,1]\vspace{1.5mm}\\
\displaystyle
\ 1\le p(\cdot)\le q(\cdot)\in L^{\infty}(\Omega),\ \ p,q\in C^{0,\sigma}(\Omega),\ \ \sigma\in (0,1].
\end{cases}
$$
We indeed have the following theorem.
\begin{theorem}\label{mt.2}
In \eqref{dppx}, assume that 
\eqn{pxqx}
$$
1<\inf_{x\in \Omega}q(x)\qquad \mbox{and}\qquad \nr{q}_{L^{\infty}(\Omega)}<1+\frac{\min\{\alpha,\sigma\}}{n},
$$
and let $u\in W^{1,1}_{\loc}(\Omega,\mathbb{R}^{N})$ be a local minimizer of \eqref{dppx}. Then $Du$ is locally H\"older continuous.
\end{theorem}
\noindent Let us point out that the outcome of Theorem \ref{mt.2} aligns with \cite{rt20,bb90}, where the uniformly elliptic counterpart of \eqref{dppx} is studied. Moreover, notice that Theorems \ref{mt.1} and \ref{mt.2} are new already in the scalar case, see Section \ref{mode} for more models we can deal with. Overall, Theorems \ref{mt.1}--\ref{mt.2} are specific instances of a general result, which goes beyond the examples listed above and grants optimal, intrinsic Schauder estimates.
\begin{theorem}\label{mt}
Under assumptions \eqref{a.1}--\eqref{a.4}, \eqref{a.5.1}--\eqref{a.3} and \eqref{a.5.2}--\eqref{a.5.2s}, let $u\in W^{1,1}_{\loc}(\Omega,\mathbb{R}^{N})$ be a local minimizer of functional $\mathcal{F}$ in \eqref{fun}. There exists $\mu_{\textnormal{max}}\equiv \mu_{\textnormal{max}}(n,\vartheta_{*},\alpha,\gamma)>1$ such that if $1\le \mu<\mu_{\textnormal{max}}$ in \eqref{a.5}--\eqref{a.5.x}, then $Du$ is locally H\"older continuous. In particular, whenever $B_{r}\Subset \Omega$ is a ball with radius $r\in (0,1)$, Lipschitz estimate
$$
\nr{Du}_{L^{\infty}(B_{r/4})}\le \frac{c}{r^{\tx{d}}}\left(\int_{B_{r}}\tx{f}(x,Du)\dx+1\right)^{\tx{d}},
$$
holds for $c\equiv c(\texttt{data})$, and $\tx{d}\equiv \tx{d}(n,\mu,\gamma,\vartheta_{*})$.
\end{theorem}
\noindent We refer to Section \ref{sa} for a detailed description of the (minimal) set of assumptions in force. From a technical standpoint, Theorem \ref{mt} relies on the development of a general nonlinear potential-theoretic machinery yielding intrinsic Lipschitz estimates that fully preserve the structural information of the underlying integrand. Unlike previous approaches \cite{dm23b,dp24}, which crucially relied on the splitting and power-type structure of integrands of the form \eqref{dp}, our method does not depend on rigid assumptions. Instead, it provides a unified treatment of a broad class of variational integrals characterized by a few common structural properties. These include nearly linear growth below, $\mu$-ellipticity, and a quantified control on the oscillation of the coefficients. The latter condition plays a crucial role, as it ensures the absence of the Lavrentiev phenomenon and, consequently, allows for the construction of suitable approximation schemes for the problem. We conclude with an outline of the content of the paper.
\subsubsection*{Outline of the paper} In Section \ref{presec} we describe our notation and collect some auxiliary results that will be helpful at various stages of the paper. In Section \ref{mode} we list some examples to which our results apply, and possible generalizations. Section \ref{dgc} is devoted to the construction of suitable power-type, uniformly elliptic approximating integrands retaining in a sharp, quantitative fashion all the structural information of the original one in \eqref{fun}. In Section \ref{fi} we develop the basic Lipschitz regularity for certain auxiliary frozen problems, that will be crucial in the proof of our main result. Section \ref{lip} is the core of the paper as it contains the intrinsic Lipschitz bounds, a fundamental step for our main result. Section \ref{vs} completes the proof of our vectorial Schauder estimates, of course in an a priori form, and finally Section \ref{asas} provides the approximation scheme which ultimately leads to the proof of Theorem \ref{mt}.

\section{Preliminaries}\label{presec}
\noindent In this section we display our notation, collect some well-known functional analytic tools that will be useful throughout the paper, and list the main structural assumptions ruling integral \eqref{fun}.
\subsection{Notation} In this paper, $\Omega\subset \mathbb{R}^{n}$, $n\ge 2$, denotes an open, bounded domain with Lipschitz regular boundary. We denote by $c$ a general constant larger than $1$. Diverse occurrences from line to line will be still indicated by $c$. Special occurrences will be denoted by $c_*,  \tilde c$ or the like. Relevant dependencies on parameters will be as usual emphasized by putting them in parentheses. Sometimes we shall use symbols "$\gtrsim$", "$\lesssim$" with subscripts, to indicate that a certain inequality holds up to constants whose dependencies are marked in the subscript. We denote by $ B_r(x_0):= \{x \in \mathbb{R}^n  :   |x-x_0|< r\}$ the open ball with center $x_0$ and radius $r>0$; we omit the center when it is not necessary or irrelevant, i.e., $B \equiv B_r \equiv B_r(x_0)$; this especially happens when various balls in the same context share the same center. With $B$ being a given ball with radius $r$ and $\theta$ being a positive number, we denote by $\theta B$ the concentric ball with radius $\theta r$ and, analogously, $B/\theta \equiv (1/\theta)B$. We further denote $\ell_{s}(t):=s+t$ for all $s,t\in [0,\infty)$. Whenever $\ti{\Omega} \subset \mathbb{R}^{n}$ is a measurable subset with bounded positive measure $0<|\ti{\Omega}|<\infty$, and $f \colon \ti{\Omega} \to \mathbb{R}^{k}$, $\mathbb{N}\ni k\geq 1$, is a measurable map, we use
$$
(f)_{\ti{\Omega}}=\mint_{\ti{\Omega}}f(x)\dx:= \snr{\ti{\Omega}}^{-1}\int_{\ti{\Omega}}  f(x) \dx
$$
to indicate the integral average. If $f\in L^{p}(\ti{\Omega},\mathbb{R}^{k})$, for some $1\le p<\infty$, we shorten its averaged norm as
$$
\nra{f}_{L^{p}(\ti{\Omega})}:=\left(\mint_{\ti{\Omega}}\snr{f}^{p}\dx\right)^{\frac{1}{p}},
$$
while if $f\in W^{s,p}(\ti{\Omega})$ with $1\le p<\infty$ and $s\in (0,1)$, its averaged Sobolev–Slobodecki\v{i} seminorm will be denoted by 
$$
\snra{f}_{s,p;\ti{\Omega}}:=\left(\mint_{\ti{\Omega}}\int_{\ti{\Omega}}\frac{\snr{f(x)-f(y)}^{p}}{\snr{x-y}^{n+sp}}\dx\dy\right)^{\frac{1}{p}}.
$$
Moreover, given any open set $\ti{\Omega}\Subset \Omega$, to simplify the notation we collect the main parameters related to the problems under investigation in the shorthands
$$
\begin{cases}
\displaystyle
\ \data_{0}:=\left(n,N,A,\tx{g},\mu,\gamma,\vartheta\right),\qquad \quad \data:=(\data_{0},\alpha,\beta,\vartheta_{*}),\vspace{1.5mm}\\
\ \tx{l}(\ti{\Omega},\Omega):=(\dist(\ti{\Omega},\partial\Omega),\diam(\ti{\Omega}),\diam(\Omega)),
\end{cases}
$$
we refer to Section \ref{sa} for an outline of the various quantities appearing above. We conclude by introducing some notation that will be helpful when specializing the forthcoming estimates to the $2$-d supercritical setting. Specifically, we set
\begin{flalign}\label{1111}
\begin{array}{c}
\displaystyle
\mathds{1}_{1}:=1,\qquad\qquad\quad \mathds{1}_{2}:=1\\[10pt]\displaystyle
\mathds{1}_{3}:=\begin{cases}
\displaystyle
\ 1\quad &\mbox{if} \ \ n\ge 3 \quad \mbox{or} \quad n=2, \ \ 0<\alpha<2/3\vspace{0.5mm}\\ \displaystyle
\ 0 \quad &\mbox{if} \ \ n=2 \ \ \mbox{and} \ \ \alpha\ge 2/3,
\end{cases}\qquad\qquad \quad \ti{\mathds{1}}:=1-\mathds{1}_{3}.
\end{array}
\end{flalign}
\subsection{Nonlinear potentials} A key role in this paper is played by a general class of nonlinear potentials, first introduced by Havin \& Maz'ya \cite{HM}. Recently, nonlinear potentials became crucial tools in the regularity theory of nonuniformly elliptic problems \cite{bm20,dm21,def22,dm23a,dm23b,bs24,ddp24,dp24,dm25} - specifically, we refer to \cite[Section 4]{dm23a} for the potential theoretic technical toolbox needed here, and to \cite{km14,min24} for an overview. For a ball $B_{r}(x_{0})\subset \mathbb{R}^{n}$, parameters $\sigma>0$, $\vartheta\geq 0$, and a function $f\in L^{1}(B_{r}(x_{0}))$, we introduce the nonlinear Havin-Maz'ya-Wolff type potential ${\bf P}_{\sigma}^{\vartheta}(f;\cdot)$, i.e.:
\eqn{defi-P} 
$$
{\bf P}_{\sigma}^{\vartheta}(f;x_0,r) := \int_0^r \varrho^{\sigma} \left(  \mint_{B_{\varrho}(x_0)} \snr{f} \dx \right)^{\vartheta} \frac{\d\varrho}{\varrho} \,.
$$
The mapping properties among function spaces of ${\bf P}_{\sigma}^{\vartheta}(f;\cdot)$ needed here are contained in the next lemma, cf. \cite[Lemma 4.1]{dm23a}.
\begin{lemma}\label{crit} 
Let $B_{\tau}\Subset B_{\tau+r}\subset \mathbb{R}^{n}$ be two concentric balls with $\tau, r\leq 1$, $f\in L^{1}(B_{\tau+r})$ and let $\sigma,\vartheta>0$ be such that $ n\vartheta>\sigma$. Then
\eqn{stimazza}
$$
\nr{{\bf P}_{\sigma}^{\vartheta}(f;\cdot,r)}_{L^{\infty}(B_{\tau})} \lesssim_{n,\vartheta,\sigma,m} \|f\|_{L^{m}(B_{\tau+r})}^{\vartheta} $$
holds whenever $m > n\vartheta/\sigma>1$. 
\end{lemma}
\noindent Finally, we record a nonlinear potential theoretic iteration à la De Giorgi, whose basic prototype can be found in \cite{km94} and \cite{min11} --- we shall record it in the form of a quantified reverse H\"older inequality, first appeared in \cite{dm23a}, see also \cite{bm20}.
\begin{lemma}\label{revlem}
Let $B_{r_{0}}(x_{0})\subset \mathbb{R}^{n}$ be a ball and $\mathbb{N}\ni k\ge 1$ an integer. For $i\in \{1,\cdots,k\}$, assume that functions $w\in L^{2}(B_{r_{0}}(x_{0}))$, $f_{i} \in L^1(B_{2r_0}(x_{0}))$, and constants $\chi >1$, $\sigma_{i}, \vartheta_{i},\ti{c},M_{0}>0$ and $\kappa_0, M_{i}\geq 0$ satisfy 
\begin{flalign}
\left(\mint_{B_{\rr/2}(x_{0})}(w-\kk)_{+}^{2\chi}  \dx\right)^{\frac1{2\chi}}  &\le \ti{c}M_{0}\left(\mint_{B_{\rr}(x_{0})}(w-\kk)_{+}^{2}  \dx\right)^{\frac{1}{2}}+\ti{c} \sum_{i=1}^{k}M_{i}\rr^{\sigma_{i}}\left(\mint_{B_{\rr}(x_{0})}\snr{f_{i}}  \dx\right)^{\vartheta_{i}},
 \label{revva}
\end{flalign}
for all $\kk\ge \kk_{0}$, and for every concentric ball $B_{\rr}(x_{0})\subseteq B_{r_{0}}(x_{0})$. If $x_{0}$ is a Lebesgue point of $w$ in the sense that 
$$
w(x_0) = \lim_{r\to 0} (w)_{B_{r}(x_0)}\,,
$$ then
\eqn{siapplica}
$$
 w(x_{0})  \le\kk_{0}+cM_{0}^{\frac{\chi}{\chi-1}}\left(\mint_{B_{r_{0}}(x_{0})}(w-\kk_{0})_{+}^{2}  \dx\right)^{1/2}
+cM_{0}^{\frac{1}{\chi-1}} \sum_{i=1}^{k}M_{i}\mathbf{P}^{\vartheta_{i}}_{\sigma_{i}}(f_{i};x_{0},2r_{0})
$$
holds with $c\equiv c(n,\chi,\sigma,\vartheta,\ti{c},k)$.  
\end{lemma}
\subsection{Embeddings of fractional Sobolev spaces} For $w \colon \Omega \to \mathbb{R}^{k}$, $k\ge 1$, $\texttt{t}>0$ and $h \in \mathbb{R}^n$, we set $\Omega_{\texttt{t}\snr{h}}:=\left\{x\in \Omega\colon \dist(x,\partial \Omega)>\texttt{t}\snr{h}\right\}$, and introduce the finite difference operator $\tau_{h}\colon L^{1}(\Omega;\mathbb{R}^{k})\to L^{1}(\Omega_{|h|};\mathbb{R}^{k})$, defined as $\tau_{h}w(x):=w(x+h)-w(x)$. Given another map $v\colon \Omega\to \mathbb{R}^{k}$, the discrete Leibniz rule reads as
\eqn{prod}
$$
\tau_{h}(vw)(x)=w(x+h)\tau_{h}v(x)+v(x)\tau_{h}w(x)\,.
$$
Moreover, if $B_{\rr}\Subset B_{r}$ are concentric balls and $w\in W^{1,p}(B_r;\mathbb{R}^{k})$, $p\ge 1$ and $\snr{h}\leq r-\rr$, then
\eqn{gh}
$$
\nr{\tau_{h}w}_{L^{p}(B_{\rr})}\le \snr{h}\nr{Dw}_{L^{p}(B_{r})}\,.
$$
Two function spaces that will play a key role in this paper are Sobolev–Slobodeckij and Nikol'skii spaces.
\begin{definition}\label{fra1def}
Let  $p \in [1, \infty)$, $s \in (0,1)$.
\begin{itemize}
\item With $\Omega \subset \mathbb{R}^n$ open, $w\colon \Omega\to \mathbb{R}^{k}$ belongs to the Sobolev-Slobodeckij space $W^{s,p}(\Omega;\mathbb{R}^k )$ iff
\begin{flalign}
\notag
\| w \|_{W^{s,p}(\Omega)} & := \|w\|_{L^{p}(\Omega)}+ \left(\int_{\Omega} \int_{\Omega}  
\frac{|w(x)
- w(y) |^{p}}{|x-y|^{n+s p}} \dx \dy \right)^{1/p}\\
&=: \|w\|_{L^{p}(\Omega)} + [w]_{s,p;\Omega} < \infty\,.\label{gaglia}
\end{flalign}
\item $w\colon \mathbb{R}^n\to \mathbb{R}^{k}$ belongs to the Nikol'skii space $N^{s,p}(\mathbb{R}^n;\mathbb{R}^k)$ iff 
$$\| w \|_{N^{s,p}(\mathbb{R}^n;\mathbb{R}^k )} :=\|w\|_{L^{p}(\mathbb{R}^n)} + \left(\sup_{|h|\not=0}\, \int_{\mathbb{R}^n} \left|\frac{\tau_{h}w}{\snr{h}^{s}}\right|^{p}
 \dx  \right)^{1/p}<\infty\,.$$
\end{itemize}
\end{definition}
\noindent Next, a combination of the embedding of Nikol'skii spaces into fractional Sobolev spaces, and the Sobolev-Morrey embedding of fractional Sobolev spaces, \cite[Lemma 2.16]{bdlmbs24}.
\begin{lemma}\label{fraim}
Let $w\in L^{2}(B_{1}(0))$ be a function and assume that, for some $s \in (0,1)$, $\mathcal{S}\ge 0$ and $0<\tx{d}<1/2$ there holds
\eqn{cru1}
$$
\nr{\tau_{h}w}_{L^{2}(B_{1/2}(0))}\le \mathcal{S}\snr{h}^{s} \qquad \mbox{for every} \ \ h\in \mathbb{R}^{n} \ \ \mbox{with} \ \ 0<\snr{h}\le \tx{d}.
$$
Then,
\eqn{cru2}
$$
\nr{w}_{L^{\frac{2n}{n-2\sigma}}(B_{1/2}(0))}+\nr{w}_{W^{\sigma,2}(B_{1/2}(0))}\le c\tx{d}^{s-\sigma}\mathcal{S}
+ c\tx{d}^{-\sigma}\nr{w}_{L^{2}(B_{1/2}(0))},
$$
for all $\sigma\in(0,s)$, where $c\equiv c(n,s,\sigma)$. 
\end{lemma}

\subsection{Tools for degenerate/singular problems} Here we collect some basic tools of common use when dealing with singular or degenerate problems. More precisely, for $s\in [0,1]$ and $0<p<\infty$, we introduce the field $V_{s,p}\colon \mathbb{R}^{N\times n}\to \mathbb{R}^{N\times n}$ defined as $V_{s,p}(z):=(s^{2}+\snr{z}^{2})^{(p-2)/4}z$. It is well-known that $V_{s,p}$ satisfies
\eqn{Vm}
$$
 \snr{V_{s,p}(z_{1})-V_{s,p}(z_{2})}\approx_{n,p}(s^{2}+\snr{z_{1}}^{2}+\snr{z_{2}}^{2})^{\frac{p-2}{4}}\snr{z_{1}-z_{2}}\qquad \mbox{for all} \ \ z_{1},z_{2}\in \mathbb{R}^{N\times n},
$$
cf. \cite{ham92}. Next, the technical equivalence
\eqn{l60}
$$
 (\snr{z_{1}}^{2}+\snr{z_{2}}^{2}+\omega^{2})^{-t/2} \approx_{n,t} \int_{0}^{1}(\snr{z_{2}+\tau (z_{1}-z_{2})}^{2}+\omega^{2})^{-t/2} \d\tau\,, 
$$
holds for all $z_{1},z_{2}\in \mathbb{R}^{N\times n}$, and any $t<1$. The following is the multidimensional counterpart of the integration-by-parts trick from \cite[Section 2.2]{dm25}, after \cite[Section 8.2]{giu03}. 
\begin{lemma}
Let $\rr,\mathcal{h}_{0}>0$ be numbers, $h\in \mathbb{R}^{n}$ be a vector such that $\snr{h}\in (0,\mathcal{h}_{0}/4)$, $B_{\rr}(x_{0})\subset \mathbb{R}^{n}$ be a ball, $V\in L^{\infty}(B_{\rr+\mathcal{h}_{0}}(x_{0}),\mathbb{R}^{N\times n})$ and $W\in W^{1,\infty}_{0}(B_{\rr}(x_{0}),\mathbb{R}^{N\times n})$ be functions, and $H\in C(B_{\rr+\mathcal{h}_{0}}(x_{0})\times \mathbb{R}^{N\times n},\mathbb{R}^{N\times n})$ be a continuous vector field which is bounded on $B_{\rr+\mathcal{h}_{0}}(x_{0})\times \ti{\Omega}$ for every bounded subset $\ti{\Omega}\subset \mathbb{R}^{n}$. Then
\eqn{af}
$$
\int_{B_{\rr}(x_{0})}\langle \tau_{h}H(\cdot,V(\cdot)),W\rangle\dx=-\snr{h}\int_{B_{\rr}(x_{0})}\int_{0}^{1}\langle H(x+\beta h,V(x+\beta h)),\partial_{h/\snr{h}}W\rangle\d\beta\dx.
$$
\end{lemma}
\noindent We close this section with the "simple, but fundamental" iteration lemma from \cite{gg82}.
\begin{lemma}\label{iterlem}
Let $h\colon [t,s]\to \mathbb{R}$ be a non-negative and bounded function, and let $a,b, m$ be non-negative numbers. Assume that the inequality 
$ 
h(\tau_2)\le  (1/2) h(\tau_1)+(\tau_1-\tau_2)^{-m}a+b,
$
holds whenever $t\le \tau_2<\tau_1\le s$. Then $
h(t)\le c(m)[a(s-t)^{-m}+b]
$, holds too. 
\end{lemma}

\subsection{Structural assumptions}\label{sa}
The high degree of generality we aim to achieve in this paper requires a careful description of the abstract integrand in \eqref{fun} that incorporates all the relevant features of the main models listed in Section \ref{mode} below. For this reason, we split the remainder of this section in several paragraph each devoted to the description of a key aspect of functional $\mathcal{F}$.
\subsection*{The main integrand: regularity and basic properties} We assume that $\tx{f}\colon \Omega\times \mathbb{R}^{N\times n}\to \mathbb{R}$ has radial (Uhlenbeck \cite{uhl77}) structure, i.e., there exists a function $A\colon \Omega\times [0,\infty)\to [0,\infty)$ such that
\eqn{a.1}
$$
\begin{cases}
\displaystyle
\ \tx{f}(x,z)=A(x,\snr{z})\ \  &\mbox{for all} \ \ (x,z)\in \Omega\times \mathbb{R}^{N\times n}\vspace{0.5mm}\\ \displaystyle
\ t\mapsto A(x,t)\in C^{2}_{\loc}(0,\infty)\cap C^{1}_{\loc}[0,\infty) \ \ &\mbox{for all} \ \ x\in \Omega \vspace{0.5mm}\\ \displaystyle
\ x\mapsto A'(x,t) \ \  & \mbox{continuous for all}\ \ t\ge 0.
\end{cases}
$$
In general,\footnote{Here $A'(x,t)=\frac{\d}{\dt}A(x,t)$.} $A$, $A'$ are continuous on $\Omega\times [0,\infty)$, while $A''$ is Carat\'eodory-regular on the same set. We shall further suppose that the basic monotonicity properties
\eqn{a.2}
$$
0<\inf_{x\in \Omega}A(x,1)\le \sup_{x\in \Omega}A(x,1)<\infty,
$$
and
\eqn{a.2.x}
$$
\begin{cases}
    \displaystyle
    \ t\mapsto \frac{A(\cdot,t)}{t} \ \ \mbox{almost increasing for all} \ \ t\in (0,\infty)\vspace{1.5mm}\\ 
    \displaystyle
    \ t\mapsto \frac{A(\cdot,t)}{t^{\gamma}}, \ \ \mbox{almost decreasing for all} \ \ t\in  [1,\infty),
\end{cases}
$$
hold\footnote{Condition $\eqref{a.2.x}_{2}$ implies that $A$ satisfies the $\Delta_{2}$-condition, that is $A(x,\tx{d}t)\le c(A,\gamma,\tx{d})(A(x,t)+1)$ for all constants $\tx{d}\in [0,\infty)$, see \cite[Lemma 2.2.6]{hh19}.} uniformly in $x\in \Omega$ for some $1<\gamma<\infty$. Notice that \eqref{a.2.x}$_{1}$ implies that $A(x,0)=0$ for all $x\in \Omega$. For $(x,t)\in \Omega\times (0,\infty)$, we set for simplicity $\mathcal{a}(x,t):=A'(x,t)t^{-1}$.

\subsection*{Closeness to linear growth}
To measure how the growth/ellipticity of $\tx{f}$ is close to linear, we introduce a function $\tx{g}\colon [0,\infty)\to [0,\infty)$ 
\eqn{a.4}
$$
\begin{cases}
\ 0\le\tx{g}(\cdot)\in C[0,\infty),\qquad t\tx{g}(t) \ \ \mbox{is convex}, \vspace{1mm}\\
\ \tx{g} \ \ \mbox{is nondecreasing, unbounded, and concave,}\vspace{1mm}\\
\ \tx{g}(t)\lesssim_{\tx{g},\omega}\ell_{1}(t)^{\omega}\qquad \mbox{for all} \ \ \omega>0, \ \ t\ge 0\vspace{1.5mm}\\
\displaystyle
\ t\tx{g}(t)\le \tx{c}A(x,t)+\tx{c},
\end{cases}
$$
for all $(x,t)\in \Omega\times [0,\infty)$. Notice that, letting $b(t):=t\tx{g}(t)$, $\eqref{a.4}_{1,2}$ implies that
\eqn{binf}
$$
\lim_{t\to \infty}\frac{b(t)}{t}=\infty,
$$
which will be useful to gain compactness in the $W^{1,1}$-setting.

\subsection*{Measuring growth/ellipticity and oscillation}
The growth/ellipticity features of $\tx{f}$ will be described by two positive continuous functions $\lambda,\Lambda\colon \Omega\times (0,\infty)\to [0,\infty)$ such that
\eqn{a.5.1}
$$
\tx{c}^{-1}\lambda(x,\snr{z})\snr{\xi}^{2}\le \langle\partial^{2}\tx{f}(x,z)\xi,\xi\rangle\qquad \mbox{and}\qquad \snr{\partial^{2}\tx{f}(x,z)}\le \tx{c}\Lambda(x,\snr{z}),
$$
for all $z\in \mathbb{R}^{N\times n}\setminus \{0\}$, $\xi\in \mathbb{R}^{N\times n}$, and $x\in \Omega$. The functions $\lambda$ and $\Lambda$ are the lowest and the highest eigenvalues of $\partial^{2}\tx{f}$ respectively, and satisfy the minimal set of assumptions
\eqn{a.5}
$$
\begin{cases}
    \displaystyle
 \ t\mapsto t^{\mu} \lambda(\cdot,t) \ \ &\mbox{is almost increasing for all} \ \ t\in (0,\infty)\vspace{1.5mm}\\
 \displaystyle
 \ t\mapsto \max\{t^{-\vartheta},1\}\lambda(\cdot,t)\ \ &\mbox{is almost decreasing for all} \ \ t\in  (0,\infty),
\end{cases}
$$
for some $\vartheta\in [0,1)$, and
\eqn{a.5.x}
$$
\inf_{x\in \Omega}\lambda(x,1)\in (0,\infty)\qquad \mbox{and}\qquad 
\frac{\Lambda(x,t)}{\lambda (x,t)}\le \tx{c}+\tx{c}\tx{g}(t)\ell_{1}(t)^{\mu-1}=:\tx{c}\rrr_{*}(t),
$$
for all $x\in \Omega$, $t\in (0,\infty)$, where $1\le \mu<2$ is a number, and $\tx{g}$ is as in \eqref{a.4}. Finally, let $B\subset \Omega$ be a ball and 
\eqn{atat}
$$
\alpha\in (0,1],\qquad \quad \vartheta_{*}\in \left(0,\alpha/n\right),
$$
be numbers. We assume that
\eqn{a.3}
$$
\snr{\lambda(x_{1},t)-\lambda(x_{2},t)}t+\snr{\mathcal{a}(x_{1},t)-\mathcal{a}(x_{2},t)}t\le \tx{c}\snr{B}^{\frac{\alpha}{n}}+\tx{c}\snr{B}^{\frac{\alpha}{n}}\left(\inf_{x\in B}A(x,t)\right)^{\vartheta_{*}},
$$
for all $x_{1},x_{2}\in B$, $t\in (0,\infty)$. 
\begin{remark}\label{rmp}
Let us highlight a few relevant facts.
\begin{itemize}
\item Oscillation bound \eqref{a.3} implies that
    \eqn{a.3.1}
$$
A(x_{1},t)\le \tx{c}A(x_{2},t)+\tx{c}\qquad \text{for all} \ \ t\in [0,\snr{B}^{-1/n}], \ \ x_{1},x_{2}\in B,
$$
which is the basic, optimal condition ensuring the absence of Lavrentiev phenomenon for $\mathcal{F}$, \cite{bds20,buli,ddp24,dp24}. Given any ball $B\subset\Omega$, $x_{1},x_{2}\in B$, $t\in (0,\snr{B}^{-1/n}]$, 
\begin{eqnarray*}
\snr{A(x_{1},t)-A(x_{2},t)}&\le&\int_{0}^{t}\snr{\mathcal{a}(x_{1},s)-\mathcal{a}(x_{2},s)}s\ds\nonumber \\
&\stackrel{\eqref{a.3}}{\le}&\tx{c}\snr{B}^{\frac{\alpha}{n}}t+\tx{c}\snr{B}^{\frac{\alpha}{n}}\int_{0}^{t}\left(\inf_{x\in B}A(x,s)\right)^{\vartheta_{*}}\ds\nonumber \\
&\le&\tx{c}\snr{B}^{\frac{\alpha}{n}}t^{\vartheta_{*}}\left(t^{1-\vartheta_{*}}+\left(\inf_{x\in B}A(x,t)\right)^{\vartheta_{*}}t^{1-\vartheta_{*}}\right)\nonumber \\
&\stackrel{\eqref{a.4}_{4}}{\le}&c\snr{B}^{\frac{\alpha}{n}}t^{\vartheta_{*}}\left(1+A(x_{2},t)\right)\nonumber \\
&\le& c\snr{B}^{\frac{\alpha-\theta_{*}}{n}}\left(1+A(x_{2},t)\right)\le c\left(1+A(x_{2},t)\right),
\end{eqnarray*}
for $c\equiv c(A,\tx{g},\alpha,\diam(B))$, which\footnote{If $t=0$, $A(x_{1},0)=A(x_{2},0)=0$ and there is nothing to prove.} is \eqref{a.3.1}.
\item To keep the technicalities at a reasonable level, in \eqref{a.3} we accounted only for a single modulus of continuity governing the oscillation of $A$. To handle without unnecessary restrictions multi phase integrals like those in \cite{dp24}, where multiple, say $\tx{k}\ge 2$, moduli of continuity appear, we just need to replace \eqref{a.3} with the more general
  \begin{flalign*}
  &\snr{\lambda(x_{1},t)-\lambda(x_{2},t)}t+\snr{\mathcal{a}(x_{1},t)-\mathcal{a}(x_{2},t)}t\nonumber \\
  &\qquad \qquad \quad \le \tx{c}\sum_{i=1}^{\tx{k}}\snr{B}^{\frac{\alpha_{i}}{n}}+\tx{c}\sum_{i=1}^{\tx{k}}\snr{B}^{\frac{\alpha_{i}}{n}}\left(\inf_{x\in B}A(x,t)\right)^{\vartheta_{*;i}},
  \end{flalign*}
with $\alpha_{i}\in (0,1]$ and $\vartheta_{*;i}\in (0,\alpha_{i}/n)$.
\end{itemize}
\end{remark}
\subsection*{Continuity of second derivatives near the origin} To gain full regularity in the vectorial setting, we need to control the oscillations of second derivatives. In fact, we prescribe that there exists $\beta\in (0,1]$, such that for every $M>0$ there is a constant $\tx{c}_{M}\equiv \tx{c}_{M}(M,A)$ satisfying
\eqn{a.5.2}
$$
\snr{A''(x,t+\tau)-A''(x,t)}\le \tx{c}_{M}A''(x,t)\left(\frac{\snr{\tau}}{t}\right)^{\beta},
$$
for all $x\in \Omega$, $t\in (0,M)$, $\tau\in \mathbb{R}$ such that $0<\snr{\tau}< t/2$.

\subsection*{A slight enforcement in $2$-d supercritical cases} To beat the criticality of the embedding of nonlinear potentials in low dimension, if $n=2$ and $\alpha\ge 2/3$ we reinforce \eqref{a.3} by requiring $\vartheta_{*}\in (0,\gamma-1]$ in \eqref{a.3} and also
\eqn{a.5.2s}
$$
\snr{\mathcal{a}(x_{1},t)-\mathcal{a}(x_{2},t)}t\le \tx{c}\snr{B}^{\frac{\alpha}{n}}\left(1+t^{\gamma-1}\right),\qquad \quad \gamma<1+\frac{\alpha}{n}.
$$
Notice that such a limitation is coherent with the usual ones imposed in the unbalanced polynomial growth setting \cite{dm23b,dp24,fps24}. 
\begin{remark}
The constant $\tx{c}\ge 1$ appearing in the structural description of integrand $\tx{f}$ in \eqref{fun} depends on the intrinsic properties of $\tx{f}$. To simplify notation, we will absorb any such dependence on $\tx{c}$, or on similar structural quantities associated with $\tx{f}$, into a generic dependence on $A$. For example, a dependence on $\inf_{x\in \Omega}A(x,1)$, see \eqref{a.2}, or on analogous bounds, will simply be recorded as a dependence on $A$.
\end{remark}
\section{Models}\label{mode} 
\noindent In this section we briefly discuss the models displayed in Section \ref{intro}, and show that they all satisfy the structural conditions in Section \ref{sa}.
\subsubsection*{\textbf{Variable exponent}} We focus on $p(x)$-type functionals, specifically on the case in which $\inf_{x\in \Omega}p(x)=1$. In fact, if $\inf_{x\in \Omega}p(x)>1$, the H\"older continuity of the gradient of minima is well-known, \cite{am01}. Since our results are local in nature, we can work on balls $B(\equiv B_{r}(x_{0})\Subset \Omega)$ with radius $r\in (0,r_{*}]$ where $r_{*}\in (0,1)$ is a threshold parameter to be determined in a few lines, such that $\min_{x\in \bar{B}}p(x)=1$. Let us consider integral \eqref{px} with \eqref{px.11},
and verify that the assumptions listed in Section \ref{sa} are satisfied.
\subsection{Verification of \eqref{a.1}--\eqref{a.4}}\label{c.1}
For $(x,t)\in \Omega\times [0,\infty)$, set $\tx{P}(x,t):=(t\log(1+t))^{p(x)}$. The very definition and the strict convexity of integrand $\tx{P}$, yield that \eqref{a.1}--\eqref{a.2} are satisfied. Moreover, $\eqref{a.2.x}_{1}$ holds as $p(x)\ge 1$ for all $x\in \Omega$ and $\eqref{a.2.x}_{2}$ is verified for any $\gamma>\nr{p}_{L^{\infty}(\Omega)}$. Condition \eqref{a.4}$_{4}$ holds with $\tx{g}(t):=\log(1+t)$ up to constants depending on $\nr{p}_{L^{\infty}(\Omega)}$, while \eqref{a.4}$_{1,2,3}$ follow from the basic properties of logarithms. 
\subsection{Verification of \eqref{a.5.1}--\eqref{a.5.x} and \eqref{a.5.2}}\label{c.2} Fix any $\vartheta_{*}\in (0,\alpha/n)$, restrict $r_{*}\in (0,1)$ so much that
\eqn{r*}
$$
[p]_{0,\alpha;\Omega}r_{*}^{\alpha}\le \frac{\vartheta_{*}}{4} \ \Longrightarrow \ \nr{p}_{L^{\infty}(\Omega)}\le 1+\frac{\vartheta_{*}}{4},
$$
and, for all $(x,t)\in \Omega\times (0,\infty)$, record
$$
\begin{cases}
    \displaystyle
    \ \frac{\tx{P}'(x,t)}{t}=p(x)\left(t\log(1+t)\right)^{p(x)-1}\left(\frac{\log(1+t)}{t}+\frac{1}{1+t}\right)\vspace{1.5mm}\\
    \displaystyle
    \ \tx{P}''(x,t)=\frac{(p(x)-1)}{\log(1+t)}\left(\log(1+t)+\frac{t}{1+t}\right)\left(\frac{\tx{P}'(x,t)}{t}\right)+p(x)\left(t\log(1+t)\right)^{p(x)-1}\left(\frac{2+t}{(1+t)^{2}}\right).
\end{cases}
$$
Next, set
$$
\begin{cases}
\displaystyle
    \ \lambda(x,t):=\frac{(t\log(1+t))^{p(x)-1}}{1+t}\vspace{1.5mm}\\
    \displaystyle
    \ \Lambda(x,t):=t^{p(x)-2}(\log(1+t))^{p(x)-2}\max\left\{1,(\log(1+t))^{2}\right\}.
\end{cases}
$$
A direct computation yields estimates for the eigenvalues
$$
\min\left\{\frac{\tx{P}'(x,t)}{t},\tx{P}''(x,t)\right\}\ge \lambda(x,t),\qquad \quad \max\left\{\frac{\tx{P}'(x,t)}{t},\tx{P}''(x,t)\right\}\lesssim_{\nr{p}_{L^{\infty}(\Omega)}} \Lambda(x,t),
$$
and \eqref{a.5.1} follows, up to constants depending on $(n,N,\nr{p}_{L^{\infty}(\Omega)})$. Moreover, $\eqref{a.5.x}_{1}$ is a consequence of the definition of $\lambda$, and
$$
\frac{\Lambda(x,t)}{\lambda(x,t)}\lesssim_{\nr{p}_{L^{\infty}(\Omega)}}1+\log(1+t),
$$
which implies $\eqref{a.5.x}_{2}$ for all $\mu\ge 1$. Estimate \eqref{a.5.2} is obtained by the mean value theorem after observing that $\snr{\tx{P}'''(t)}\lesssim_{\nr{p}_{L^{\infty}(\Omega)}} t^{-1}\tx{P}''(t)$ and exploiting the very definition of $\tx{P}''$. Finally, let us take care of \eqref{a.5}. A direct computation (take first derivatives), shows that $t\mapsto t^{\mu}\lambda(\cdot,t)$ is nondecreasing for all $t\in (0,\infty)$ and any $\mu\ge 1$. Concerning $\eqref{a.5}_{2}$, set $\vartheta:=2\left(\nr{p}_{L^{\infty}(\Omega)}-1\right)\in [0,1)$, 
cf.\ \eqref{r*}, and notice that if $t\in (0,1)$, then $(1+t)^{-1},\log(1+t)t^{-1}\approx 1$, thus $\max\{t^{-\vartheta},1\}\lambda(x,t)=t^{-\vartheta}\lambda(x,t)\approx_{\nr{p}_{L^{\infty}(\Omega)}} t^{-\vartheta+2(p(x)-1)}$, which is decreasing for all $(x,t)\in \Omega\times (0,1)$. On the other hand, if $t\ge 1$, then $\max\{t^{-\vartheta},1\}\lambda(x,t)=\lambda (x,t)$, which is decreasing for all $x\in \Omega$, therefore $\eqref{a.5}_{2}$ comes by combining the two previous observations. 
\subsection{Verification of \eqref{a.3} and \eqref{a.5.2s}}\label{c.3} Let $B_{\rr}\subseteq B$ be any ball, set $\mathcal{p}(x,t):=(t\log(1+t))^{p(x)-1}$ and notice that, given the explicit expressions of $\tx{P}'(x,t)t^{-1}$ and $\lambda(x,t)$, it is enough to estimate $\osc_{x\in B_{\rr}}\mathcal{p}(x,t)$. Note that \eqref{r*} implies that
\eqn{c.3.1}
$$
\mathcal{p}(x,t)\lesssim_{\nr{p}_{L^{\infty}(\Omega)}} \max\left\{1,\left(t\log(1+t)\right)\right\}^{\frac{\vartheta_{*}}{4}},
$$
and
\begin{eqnarray}\label{c.3.2}
\snr{\mathcal{p}(x_{1},t)-\mathcal{p}(x_{2},t)}&\le& [p]_{0,\alpha;\Omega}\snr{B_{\rr}}^{\frac{\alpha}{n}}\log\left(\max\left\{1,t\log(1+t)\right\}\right)\max\left\{1,\left(t\log(1+t)\right)\right\}^{\nr{p}_{L^{\infty}(\bar{B})}-1}\nonumber \\
&\le&c([p]_{0,\alpha;\Omega},\vartheta_{*})\snr{B_{\rr}}^{\frac{\alpha}{n}}\max\left\{1,\left(t\log(1+t)\right)\right\}^{\frac{\vartheta_{*}}{2}},
\end{eqnarray}
for all $x_{1},x_{2}\in B_{\rr}$, where we also used well-known properties of logarithms. We then obtain
\begin{flalign}\label{c.3.3}
&\snr{\lambda(x_{1},t)-\lambda(x_{2},t)}t+\snr{\tx{P}'(x_{1},t)-\tx{P}'(x_{2},t)}\nonumber \\
&\qquad \qquad \quad \le\snr{p(x_{1})-p(x_{2})}\nr{\mathcal{p}(\cdot,t)}_{L^{\infty}(\Omega)}\left(\log(1+t)+\frac{t}{1+t}\right)\nonumber \\
&\qquad\qquad \quad \quad +\nr{p}_{L^{\infty}(\Omega)}\left(\log(1+t)+\frac{t}{1+t}\right)\snr{\mathcal{p}(x_{1},t)-\mathcal{p}(x_{2},t)}\nonumber \\
&\qquad \qquad \quad \quad +\left(\frac{t}{1+t}\right)\snr{\mathcal{p}(x_{1},t)-\mathcal{p}(x_{2},t)}\nonumber \\
&\qquad \qquad \quad \le c\snr{B_{\rr}}^{\frac{\alpha}{n}}\left(1+\log(1+t)+\frac{t}{1+t}\right)\max\left\{1,\left(t\log(1+t)\right)\right\}^{\frac{\vartheta_{*}}{2}}\nonumber \\
&\qquad \qquad \quad \le c\snr{B_{\rr}}^{\frac{\alpha}{n}}+c\snr{B_{\rr}}^{\frac{\alpha}{n}}(t\log(1+t))^{\frac{3\vartheta_{*}}{4}}\le c\snr{B_{\rr}}^{\frac{\alpha}{n}}+c\snr{B_{\rr}}^{\frac{\alpha}{n}}\inf_{x\in B_{\rr}}\tx{P}(x,t)^{\vartheta_{*}},
\end{flalign}
and \eqref{a.3} holds with $c\equiv c(\nr{p}_{C^{0,\alpha}(\Omega)},\alpha,\vartheta_{*})$. To validate \eqref{a.5.2s}, we look back at the last-but-one line of the previous display, and use the slow growth of logarithms to bound
$$
\snr{\tx{P}'(x_{1},t)-\tx{P}'(x_{2},t)}\le c\snr{B_{\rr}}^{\frac{\alpha}{n}}+c\snr{B_{\rr}}^{\frac{\alpha}{n}}(t\log(1+t))^{\frac{3\vartheta_{*}}{4}}\le c\snr{B_{\rr}}^{\frac{\alpha}{n}}\left(1+t^{\vartheta_{*}}\right),
$$
for $c\equiv c(\nr{p}_{C^{0,\alpha}(\Omega)},\alpha,\vartheta_{*})$. Letting $\vartheta_{*}=\gamma-1$, any $1\le\nr{p}_{L^{\infty}(\Omega)}<\gamma<1+\alpha/2$, together with \eqref{r*} works.\\\\
\noindent The assumptions of Theorem \ref{mt} are then verified by merging the content of Sections \ref{c.1}--\ref{c.3}, and Theorem \ref{mt.1} follows.
\begin{remark}
One can check that analogous considerations hold if we replace the integrand $\tx{P}$ in \eqref{px}, by
    $$
    \tx{P}_{1}(x,t):=t^{p(x)}\log(1+t)\quad \mbox{or}\quad \tx{P}_{2}(x,t):=t(\log(1+t))^{p(x)},
    $$
    with exponent $p$ as in \eqref{px}.
\end{remark}
\subsubsection*{\textbf{Log-Double Phase}} Our setting enc also the Log-Double Phase energy, that is functional \eqref{dp}, subject to \eqref{aq1}.
Maximal regularity for minima of functional $\eqref{dp}_{1}$ has been obtained in \cite{dm23b,ddp24} in the scalar setting $N=1$, here we show that our approach grants the vectorial counterpart of \cite[Theorem 1.3]{dm23b}.\footnote{Theorem \ref{mt} does not cover \cite[Theorem 1]{ddp24}, which deals with bounded minima of integral $\eqref{dp}$ within the sharp maximal nonuniformity range $q<1+\alpha$, and requires a completely different strategy.} 
\subsection{Verification of \eqref{a.1}--\eqref{a.4}}\label{c.5} For $(x,t)\in \Omega \times [0,\infty)$, set $\tx{H}(x,t):=t\log(1+t)+a(x)t^{q}$, with coefficient $a$ and exponent $q$ as in \eqref{dp}--\eqref{aq1}. Verifying \eqref{a.1}--\eqref{a.2} is straightforward, given the definition of integrand $\tx{H}$. The monotonicity conditions in \eqref{a.2.x} hold for any $\gamma\ge q$. Moreover, the conditions in \eqref{a.4} are satisfied with $\tx{g}(t):=\log(1+t)$.
\subsection{Verification of \eqref{a.5.1}--\eqref{a.5.x} and \eqref{a.5.2}}\label{c.6} For $(x,t)\in \Omega\times (0,\infty)$, we have
$$
\begin{cases}
    \displaystyle
    \ \frac{\tx{H}'(x,t)}{t}=\frac{\log(1+t)}{t}+\frac{1}{1+t}+qa(x)t^{q-2}\vspace{1.5mm}\\
    \displaystyle
    \ \tx{H}''(x,t)=\frac{2+t}{(1+t)^{2}}+q(q-1)a(x)t^{q-2}.
\end{cases}
$$
We then define
$$
\begin{cases}
    \displaystyle
    \ \lambda(x,t):=\frac{1}{1+t}+a(x)t^{q-2}\vspace{1.5mm}\\
    \displaystyle
    \ \Lambda(x,t):=\frac{1+\log(1+t)}{1+t}+a(x)t^{q-2},
\end{cases}
$$
so that 
$$
\min\left\{\frac{\tx{H}'(x,t)}{t},\tx{H}''(x,t)\right\}\ge \lambda(x,t),\qquad \quad \max\left\{\frac{\tx{H}'(x,t)}{t},\tx{H}''(x,t)\right\}\lesssim_{q}\Lambda(x,t),
$$
thus \eqref{a.5.1} and \eqref{a.5.x} are satisfied up to constants depending on $(n,N,q)$, cf. \cite[Section 3]{dm23b}. Furthermore, \eqref{a.5} is true for $\mu=1$ and any $\vartheta\ge0$ - recall that, by \eqref{aq1}, $1<q<2$. Finally, the splitting structure of $\tx{H}$ guarantees the validity of the H\"older condition \eqref{a.5.2} of $\tx{H}''$. In fact, \eqref{a.5.2} can be verified by treating separately the (weighted) $q$-power term, for which \eqref{a.5.2} is standard \cite[Remark 2.3]{dsv09}, and the nearly linear growing one, that is subject to considerations analogous to those in Section \ref{c.2}.
\subsection{Verification of \eqref{a.3} and \eqref{a.5.2s}}\label{c.7} Fix any ball $B_{\rr}\subset \Omega$. By the splitting structure of $\tx{H}$, it is enough to control
$$
t^{q-1}\snr{a(x_{1})-a(x_{2})}\le [a]_{0,\alpha;\Omega}\snr{B_{\rr}}^{\frac{\alpha}{n}}t^{q-1}\stackrel{\eqref{aq1}}{\le}[a]_{0,\alpha;\Omega}\snr{B_{\rr}}^{\frac{\alpha}{n}}\left(\inf_{x\in B_{\rr}}\tx{H}(x,t)\right)^{\vartheta_{*}},
$$
with $\vartheta_{*}:=q-1$, admissible by \eqref{aq1}. Finally, \eqref{a.5.2s} holds by direct computation via \eqref{aq1}.
\subsubsection*{\textbf{Log-Multi Phase}} The prototypical model we have in mind is
\eqn{mp}
$$
\begin{cases}
\displaystyle
\ w\mapsto \int_{\Omega}\snr{Dw}\log(1+\snr{Dw})+\sum_{i=1}^{k}a_{i}(x)\snr{Dw}^{q_{i}}\dx\vspace{1.5mm}\\
\displaystyle
\ 0\le a_{i}(\cdot)\in C^{0,\alpha_{i}}(\Omega),  \ \ i\in \{1,\cdots,k\},
\end{cases}
$$
for exponents $q_{i}>1$, $\alpha_{i}\in (0,1]$ for all $i\in \{1,\cdots, k\}$, verifying
$$
q_{i}<1+\frac{\alpha_{i}}{n}\qquad \mbox{for all} \ \ i\in \{1,\cdots,k\}.
$$
Recalling the third bullet of Remark \ref{rmp}, the verification of the assumptions listed in Section \ref{sa} is analogous to the Log-Double Phase case (with obvious structural variations accounting for the multiple phases), thus Theorem \ref{mt} grants the validity of Theorem \ref{mt.1.5} (and of its Log-Multi Phase counterpart). In particular, Schauder theory for vector-valued minimizers of the Log-Multi Phase energy \eqref{mp} holds true, see \cite[Theorem 1.1 and Corollary 1.2]{dp24} for the scalar counterpart.
\begin{remark}
Similar considerations as in Sections \emph{\textbf{Log-Double Phase}} - \emph{\textbf{Log-Multi Phase}} show that the assumptions listed in Section \ref{sa} cover also the perturbed model $\ti{\tx{H}}(x,t):=t\log(1+t)+a(x)t^{q}\log(1+t)$, with obvious extensions to the multi-phase scenario.
\end{remark}
\subsubsection*{\textbf{Log-Double Phase with variable exponents}} This is integral \eqref{dppx}, under assumptions \eqref{pxqx}.
We focus on the case $1=\inf_{x\in \Omega}p(x)$, otherwise the gradient H\"older continuity of minima follows by more standard means, cf. \cite{rt20,bb90} and references therein. We can therefore proceed as in Section \ref{c.2}, set $\tx{J}(x,t):=(t\log(1+t))^{p(x)}+a(x)t^{q(x)}$, $q_{0}:=\inf_{x\in \Omega}q(x)$, fix any $\vartheta_{*}\in (\nr{q}_{L^{\infty}(\Omega)}-1,\min\{\alpha,\sigma\}/n)$ and work on balls $B(\equiv B_{r}(x_{0}))\Subset \Omega$ such that $\min_{x\in \bar{B}}p(x)=1$, $q_{0;B}:=\min_{x\in \bar{B}}q(x)>1$ and with radius $r\in (0,r_{*}]$ for some threshold $r_{*}\in (0,1)$ satisfying
\eqn{px.1}
$$
\max\{[p]_{0,\beta;\Omega},[q]_{0,\beta;\Omega}\}r_{*}^{\sigma}\le \frac{\vartheta_{*}}{8} \ \Longrightarrow \ 
\nr{p}_{L^{\infty}(\bar{B})}\le 1+\frac{\vartheta_{*}}{8}.
$$
Let us verify the validity of the conditions in Section \ref{sa}.
\subsection{Verification of \eqref{a.1}--\eqref{a.4}}\label{c.10} Assumptions \eqref{a.1}--\eqref{a.2.x}$_{1}$ immediately follow from the definition of integrand $\tx{J}$, $\eqref{a.2.x}_{2}$ is satisfied for any $\gamma>\nr{q}_{L^{\infty}(\bar{B})}$, and \eqref{a.4} holds with $\tx{g}(t):=\log(1+t)$.
\subsection{Verification of \eqref{a.5.1}--\eqref{a.5.x} and \eqref{a.5.2}} For $(x,t)\in \Omega\times (0,\infty)$, 
$$
\begin{cases}
    \displaystyle
    \ \frac{\tx{J}'(x,t)}{t}=p(x)\left(t\log(1+t)\right)^{p(x)-1}\left(\frac{\log(1+t)}{t}+\frac{1}{1+t}\right)+q(x)a(x)t^{q(x)-1}\vspace{1.5mm}\\
    \displaystyle
    \ \tx{J}''(x,t)=\left(\frac{\tx{J}'(x,t)}{t}\right)\left(\frac{(p(x)-1)}{\log(1+t)}\left(\log(1+t)+\frac{t}{1+t}\right)\right)\vspace{1.5mm} \\
    \displaystyle
\qquad\qquad\quad+p(x)\left(t\log(1+t)\right)^{p(x)-1}\left(\frac{2+t}{(1+t)^{2}}\right)+q(x)(q(x)-1)a(x)t^{q(x)-2}.
\end{cases}
$$
Letting
$$
\begin{cases}
    \displaystyle
    \ \lambda(x,t):=\frac{(t\log(1+t))^{p(x)-1}}{1+t}+a(x)t^{q(x)-2}\vspace{1.5mm}\\
    \displaystyle
    \ \Lambda(x,t):=t^{p(x)-2}(\log(1+t))^{p(x)-2}\max\left\{1,(\log(1+t))^{2}\right\}+a(x)t^{q(x)-2},
\end{cases}
$$
a direct computation gives
$$
\min\left\{\frac{\tx{J}'(x,t)}{t},\tx{J}''(x,t)\right\}\ge \lambda(x,t),\qquad \quad \max\left\{\frac{\tx{J}'(x,t)}{t},\tx{J}''(x,t)\right\}\lesssim_{q}\Lambda(x,t),
$$
and \eqref{a.5.1} follows, with constants depending on $(n,N,\nr{p}_{L^{\infty}(\Omega)},\nr{q}_{L^{\infty}(\Omega)},q_{0})$. The monotonicity properties in \eqref{a.5} hold with $\vartheta:=2(\nr{p}_{L^{\infty}(\Omega)}-1)$, by combining Sections \ref{c.2} and \ref{c.6}. Moreover, \eqref{a.5.x}$_{1}$ trivially follows from the definition of $\lambda$, and
$$
\frac{\Lambda(x,t)}{\lambda(x,t)}\lesssim_{\nr{p}_{L^{\infty}(\Omega)},\nr{q}_{L^{\infty}(\Omega)}}1+\log(1+t),
$$
which is $\eqref{a.5.x}_{2}$, valid for all $\mu\ge 1$. Finally, thanks to the splitting structure of integrand $\tx{J}$, \eqref{a.5.2} follows by merging again Sections \ref{c.2} and \ref{c.6}.
\subsection{Verification of \eqref{a.3} and \eqref{a.5.2s}} Let $\mathcal{p}$ be the auxiliary function introduced in Section \ref{c.3}, set $\mathcal{q}(x,t):=a(x)t^{q(x)-1}$ and let $B_{r}\subseteq B$. We then estimate, via \eqref{c.3.1} and \eqref{px.1},
\begin{eqnarray*}
\mathcal{p}(x,t)+\mathcal{q}(x,t)&\le&c\max\left\{1,\left(t\log(1+t)\right)\right\}^{\frac{\vartheta_{*}}{4}}+ca(x)\ell_{1}(t)^{\nr{q}_{L^{\infty}(\Omega)}-1}\nonumber \\
&\le&c\ell_{1}(t)^{\frac{\vartheta_{*}}{2}}+c\ell_{1}(t)^{\nr{q}_{L^{\infty}(\Omega)}-1},
\end{eqnarray*}
for $c\equiv c(n,\nr{p}_{L^{\infty}(\Omega)},\nr{q}_{L^{\infty}(\Omega)},\alpha)$, and
\begin{eqnarray*}
\snr{\mathcal{q}(x_{1},t)-\mathcal{q}(x_{2},t)}&\le &c\snr{q(x_{1})-q(x_{2})}\snr{\log(t)}t^{\nr{q}_{L^{\infty}(\Omega)}-1}\nonumber \\
&&+c\snr{a(x_{1})-a(x_{2})}\ell_{1}(t)^{\nr{q}_{L^{\infty}(\Omega)}-1}\le c\snr{B_{\rr}}^{\frac{\min\{\alpha,\sigma\}}{n}}\ell_{1}(t)^{\nr{q}_{L^{\infty}(\Omega)}-1+\epsilon},
\end{eqnarray*}
where $\epsilon\in (\nr{q}_{L^{\infty}(\Omega)}-1,\vartheta_{*})$ is any number and $c\equiv c(n,\nr{a}_{C^{0,\alpha}(\Omega)}, \nr{q}_{C^{0,\sigma}(\Omega)},\alpha)$. By the splitting structure of $\tx{J}$ and \eqref{c.3.3} we get
\begin{eqnarray*}
\snr{\lambda(x_{1},t)-\lambda(x_{2},t)}t+\snr{\tx{J}'(x_{1},t)-\tx{J}'(x_{2},t)}&\le&c\snr{B_{\rr}}^{\frac{\sigma}{n}}\left(1+(t\log(1+t))^{\frac{3\vartheta_{*}}{4}}\right)\nonumber \\
&&+c\snr{\mathcal{q}(x_{1},t)-\mathcal{q}(x_{1},t)}+c\snr{q(x_{1})-q(x_{2})}\ell_{1}(t)^{\nr{q}_{L^{\infty}(\Omega)}-1}\nonumber \\
&\le&c\snr{B_{\rr}}^{\frac{\sigma}{n}}+c\snr{B_{\rr}}^{\frac{\sigma}{n}}\left(\inf_{x\in B_{\rr}}\tx{J}(x,t)\right)^{\vartheta_{*}}\nonumber \\
&&+c\snr{B_{\rr}}^{\frac{\min\{\alpha,\sigma\}}{n}}\ell_{1}(t)^{\nr{q}_{L^{\infty}(\Omega)}-1+\epsilon}\nonumber \\
&\le&c\snr{B_{\rr}}^{\frac{\min\{\sigma,\alpha\}}{n}}+c\snr{B_{\rr}}^{\frac{\min\{\sigma,\alpha\}}{n}}\left(\inf_{x\in \bar{B}}\tx{J}(x,t)\right)^{\vartheta_{*}},
\end{eqnarray*}
with $c\equiv c(n,\nr{a}_{C^{0,\alpha}(\Omega)}, \nr{q}_{C^{0,\sigma}(\Omega)},\alpha)$, and \eqref{a.3} is verified. Concerning \eqref{a.5.2s}, we just need to restrict the size of $\gamma$ in Section \ref{c.10}: any $\gamma\in (\nr{q}_{L^{\infty}(\Omega)},1+\vartheta_{*})$ works.
\begin{remark} 
Let us quickly comment on possible extensions for the variable exponent Log-Double Phase model.
\begin{itemize}
\item 
    The content of Section \emph{\textbf{Log-Double Phase with variable exponents}} can be adapted to similar integrals like
$$
\begin{cases}
    \displaystyle
    \ w\mapsto \int_{\Omega}(\snr{Dw}\log(1+\snr{Dw}))^{p(x)}+a(x)\left(\sqrt{1+\snr{Dw}^{2}}\right)^{q(x)}\vspace{1.5mm}\\
    \displaystyle
    \ 0\le a(\cdot)\in C^{0,\alpha}(\Omega),\qquad 1\le p(\cdot)\le q(\cdot)\in C^{0,\beta}(\Omega),
\end{cases}
$$
with $\alpha,\beta\in (0,1]$, and
$$
\nr{q}_{L^{\infty}(\Omega)}<1+\frac{\alpha}{n},\qquad \quad\inf_{x\in \Omega}q(x)=\inf_{x\in \Omega}p(x)=1. 
$$
\item Needless to say, all the considerations made for integral \eqref{dppx} can be extended effortlessly to the variable exponent Multi-Phase case via the last bullet of Remark \ref{rmp}.
\end{itemize}
   
\end{remark}
\begin{remark}
Our results apply also when a general, $\mu$-elliptic integrand $\tx{L}\in C^{2}_{\loc}(0,\infty)\cap C^{1}_{\loc}[0,\infty)$ satisfying
    $$
    \begin{cases}
        \displaystyle
        \ t\tx{g}(t)-1\lesssim \tx{L}(t)\lesssim t\tx{g}(t)+1\vspace{1.5mm}\\
        \displaystyle
        \ \min\{t^{-1}\tx{L}'(t),\tx{L}''(t)\}\gtrsim (1+t)^{-\mu}\vspace{1.5mm}\\
        \displaystyle
        \ \max\{t^{-1}\tx{L}'(t),\tx{L}''(t)\}\lesssim \frac{\tx{g}(t)+1}{1+t}
    \end{cases}
    $$
    for all $t\in (0,\infty)$, some function $\tx{g}$ as in \eqref{a.4}$_{1,2,3}$ and $\mu>1$ arbitrarily close to one, replaces the basic nearly linear growing quantity $t\mapsto t\log(1+t)$ in all models listed above. This includes, for instance, the case of iterated logarithms:
    $$
\begin{cases}
\ \tx{L}(t):=t\bar{\tx{L}}_{i+1}(t)\quad & \mbox{for} \ \ i\ge 0\\
\ \bar{\tx{L}}_{i+1}(t):=\log(1+\bar{\tx{L}}_{i}(t)) \quad & \mbox{for} \ \ i\ge 0\\
\ \bar{\tx{L}}_{0}(t):=t,
\end{cases}
$$
see \cite{fm00,dm23b}. The only changes affect the eigenvalues $\lambda$, $\Lambda$ which now grow at infinity as $(1+t)^{-\mu}$ and $(1+\tx{g}(t))(1+t)^{-1}$ respectively.
\end{remark}
\section{Degeneracy/growth corrections}\label{dgc}
\noindent In this section we construct suitable regularized counterparts of the integrands introduced in Section \ref{sa}. Let $s\mapsto \sigma_{s}$ be a decreasing function with $\sigma_{0}=0$, $\delta,\varepsilon\in (0,1/4)$ be positive numbers, and, for $(x,t)\in \Omega\times [0,\infty)$, define 
\eqn{ade.1}
$$
\begin{cases}
\displaystyle
\ \mathcal{a}_{\delta}(x,t):=\mathcal{a}(x,\ell_{\delta}(t)),\qquad \quad &\mathcal{a}_{\delta}^{\varepsilon}(x,t):=\mathcal{a}_{\delta}(x,t)+4\gamma\sigma_{\varepsilon}\ell_{1}(t^{2})^{2\gamma-1},\vspace{1.5mm}\\ \displaystyle 
\ \bar{\mathcal{a}}_{\delta}(x,t):=\mathcal{a}_{\delta}(x,t)+t\mathcal{a}_{\delta}'(x,t),\qquad \quad &\bar{\mathcal{a}}_{\delta}^{\varepsilon}(x,t):=\mathcal{a}_{\delta}^{\varepsilon}(x,t)+t(\mathcal{a}_{\delta}^{\varepsilon})'(x,t),
\end{cases}
$$
set corresponding integrands
\begin{flalign}\label{adade}
A_{\delta}(x,t):=\int_{0}^{t}\mathcal{a}_{\delta}(x,s)s\ds,\qquad \mbox{and}\qquad A_{\delta}^{\varepsilon}(x,t):=A_{\delta}(x,t)+\sigma_{\varepsilon} \ell_{1}(t^{2})^{2\gamma},
\end{flalign}
for all $(x,z)\in \Omega\times \mathbb{R}^{N\times n}$, and name 
\begin{flalign*}
&\tx{f}_{\delta}(x,z):=A_{\delta}(x,\snr{z}),\qquad \qquad \quad \tx{f}_{\delta}^{\varepsilon}(x,z):=A_{\delta}^{\varepsilon}(x,\snr{z}),
\end{flalign*}
see \cite{de08,dsv09}. Define also the corrected ``eigenfunctions''
\begin{flalign*}
&\lambda_{\delta}(x,t):=\lambda(x,\ell_{\delta}(t)),\qquad \quad \ \lambda_{\delta}^{\varepsilon}(x,t):=\lambda_{\delta}(x,t)+\sigma_{\varepsilon}\ell_{1}(t^{2})^{2\gamma-1}\\
&\Lambda_{\delta}(x,t):=\Lambda(x,\ell_{\delta}(t)),\qquad \quad \Lambda_{\delta}^{\varepsilon}(x,t):=\Lambda_{\delta}(x,t)+\sigma_{\varepsilon}\ell_{1}(t^{2})^{2\gamma-1},
\end{flalign*}
that now make sense for all $(x,t)\in \Omega\times [0,\infty)$, and the primitive
$$
\tx{a}_{\delta}(x,t):=\int_{0}^{t}\lambda_{\delta}(x,s)s\ds\qquad \mbox{and}\qquad \tx{a}_{\delta}^{\varepsilon}(x,t):=\int_{0}^{t}\lambda_{\delta}^{\varepsilon}(x,s)s\ds.
$$
By construction, \eqref{a.1} and \eqref{a.3}, we have 
\eqn{regreg}
$$
z\mapsto \tx{f}_{\delta}^{\varepsilon}(\cdot,z)\in C^{2}_{\loc}(\mathbb{R}^{N\times n})\qquad \mbox{and}\qquad x\mapsto \tx{f}_{\delta}^{\varepsilon}(x,\cdot)\in C^{0,\alpha}(\Omega).
$$
In fact, $\eqref{regreg}_{1}$ follows directly by construction, while $\eqref{regreg}_{2}$ is detailed in \eqref{a12.x}$_{1}$ below. Let us show that the newly defined maps preserve the original structural features of $A$ and $b$.
\begin{lemma}\label{btbt}
    There exists $\tx{t}\equiv \tx{t}(A,\tx{g})\ge 1$ such that 
    \eqn{d2.3}
    $$
    \begin{cases}
    \displaystyle
    \ t\ge \tx{t} \ \Longrightarrow \ t\le c\tx{g}(t)\vspace{1.5mm}\\
    \displaystyle
    \  t\tx{g}(t)+\sigma_{\varepsilon} t^{4\gamma}\le cA_{\delta}^{\varepsilon}(x,t)+c,
    \end{cases}
    $$
    for all $(x,t)\in \Omega\times [0,\infty)$, with $c\equiv c(A,\tx{g},\gamma)$. Moreover,
    \eqn{add}
    $$
    \begin{cases}
    \displaystyle
    \ t\mapsto \frac{A_{\delta}(\cdot,t)}{t} \ \ \mbox{is almost increasing for all} \ \ t\in  (0,\infty)\vspace{1.5mm}\\
    \displaystyle
    \ t\mapsto \frac{A_{\delta}(\cdot,t)}{t^{\gamma}} \ \ \mbox{is almost decreasing for all} \ \ t\in  [1,\infty),
    \end{cases}
    $$
    and for every constant $\tx{d}\in [0,\infty)$,
    \eqn{d2}
$$
A_{\delta}^{\varepsilon}(x,\tx{d}t)\le c(A,\gamma,\tx{d})\left(1+A_{\delta}^{\varepsilon}(x,t)\right),
$$
holds for all $(x,t)\in \Omega\times [0,\infty)$. Finally,
\eqn{d2.2}
$$
A_{\delta}'(x,t)\le c\left(\frac{A_{\delta}(x,t)}{\ell_{\delta}(t)}+1\right)
$$
for all $(x,t)\in \Omega\times [0,\infty)$, with $c\equiv c(A,\gamma)$.
\end{lemma}
\begin{proof}
By $\eqref{a.2}$, there exists a constant $c\equiv c(A)>0$ such that $A(x,t)\ge c$ for all $t\ge 1$, so the convexity of $t\mapsto A(x,t)$ implied by \eqref{a.5.1}, \eqref{a.2.x}$_{1}$ and \eqref{adade} yield \eqref{add} and
\eqn{ad1}
$$
A_{\delta}(x,t)\approx_{A} A(x,\ell_{\delta}(t))\qquad \mbox{for all} \ \ (x,t)\in \Omega\times [1,\infty).
$$
Implication \eqref{d2.3}$_{2}$ then follows via \eqref{ad1}, \eqref{a.4}$_{1,2,4}$ and \eqref{adade}, while $\eqref{d2.3}_{1}$ comes by \eqref{a.4}$_{1,2,3}$. The $\Delta_{2}$-condition in \eqref{d2} is a consequence of \eqref{add}$_{2}$ and the construction of $A_{\delta}^{\varepsilon}$ in \eqref{adade}. As $t\mapsto A_{\delta}^{\varepsilon}(x,t)$ is strictly convex, \eqref{d2.2} is derived from \eqref{add}$_{1}$ and \eqref{d2.2} following the arguments in \cite[Lemma 2.1]{mar89}.
\end{proof}
\noindent In the next lemma we describe the main features of the regularized eigenvalues
\begin{lemma}
    The regularized eigenvalues satisfy
\eqn{a.55}
$$
\begin{cases}
    \displaystyle
  \   t\mapsto \ell_{\delta}(t)^{\mu}\lambda_{\delta}(\cdot,t)\quad &\mbox{is almost increasing for all} \ \ t\in [0,\infty)\vspace{1.5mm}\\ \displaystyle
  \ t\mapsto \max\{\ell_{\delta}(t)^{-\vartheta},\ell_{\delta}(1)\}\lambda_{\delta}(\cdot,t)\quad &\mbox{is almost decreasing for all} \ \ t\in  [0,\infty),
\end{cases}
$$
where $\mu\ge 1$, $\vartheta\in [0,1)$ are the same exponents in $\eqref{a.5}$. 
Moreover, the ellipticity ratio bounds
\eqn{a.555}
$$
\frac{\Lambda_{\delta}(x,t)}{\lambda_{\delta}(x,t)}\le c\rrr_{*}(t)\qquad \mbox{and}\qquad \frac{\Lambda_{\delta}^{\varepsilon}(x,t)}{\lambda_{\delta}^{\varepsilon}(x,t)}\le c\rrr_{*}(t),
$$
hold for all $(x,t)\in \Omega\times [0,\infty)$ and some $c\equiv c(n,N,A,\gamma)$, where $\rrr_{*}$ has been defined in \eqref{a.5.x}, and
\eqn{rrr}
    $$\rrr_{*}(t)\le c(\tx{g},\mu,\omega_{\mu})\ell_{1}(t)^{2(\mu-1+\omega_{\mu})}\qquad \mbox{for all} \ \ t\in [0,\infty), \ \ \mbox{and any} \ \ \omega_{\mu}>0.$$
\end{lemma}
\begin{proof}
The content of displays \eqref{a.55}--\eqref{a.555} follows immediately from \eqref{a.5}--\eqref{a.5.x}, and the definition of the approximating eigenvalues. Next, estimate \eqref{rrr} comes by $\eqref{a.4}_{3}$ choosing $\omega=\mu-1$ if $\mu>1$ and $\omega=\omega_{\mu}$ with $\omega_{\mu}>0$ being any number if $\mu=1$. 
\end{proof}
\noindent We then collect the main growth/ellipticity properties of our approximating integrands.
\begin{lemma}
The following hold.
\begin{itemize}
\item For all $x\in \Omega$, $z,\xi\in \mathbb{R}^{N\times n}$,
\eqn{lala}
$$
\begin{cases}
\displaystyle
\ \langle\partial^{2}\tx{f}_{\delta}^{\varepsilon}(x,z)\xi,\xi\rangle\ge c\lambda_{\delta}^{\varepsilon}(x,\snr{z})\snr{\xi}^{2}\qquad \mbox{and}\qquad \snr{\partial^{2}\tx{f}_{\delta}^{\varepsilon}(x,z)}\le c\Lambda_{\delta}^{\varepsilon}(x,\snr{z})\vspace{1.5mm}\\
\displaystyle
\ \langle\partial\tx{f}_{\delta}^{\varepsilon}(x,z_{1})-\partial\tx{f}_{\delta}^{\varepsilon}(x,z_{2}),z_{1}-z_{2}\rangle\ge c\lambda_{\delta}^{\varepsilon}(x,\snr{z_{1}},\snr{z_{2}})\snr{z_{1}-z_{2}}^{2},
\end{cases}
$$
with $c\equiv c(n,N,A,\gamma,\vartheta)$. 
\item For all $x\in \Omega$, $t\in [0,\infty)$, 
\eqn{a.7}
$$
\begin{cases}
\displaystyle
\ \bar{\mathcal{a}}_{\delta}^{\varepsilon}(x,t)\ge c\lambda_{\delta}^{\varepsilon}(x,t),\qquad \mathcal{a}_{\delta}^{\varepsilon}(x,t)\ge c\lambda_{\delta}^{\varepsilon}(x,t)\vspace{0.5mm}\\ \displaystyle
\ \bar{\mathcal{a}}_{\delta}^{\varepsilon}(x,t)\le c\Lambda_{\delta}^{\varepsilon}(x,t),\qquad \mathcal{a}_{\delta}^{\varepsilon}(x,t)\le c\Lambda_{\delta}^{\varepsilon}(x,t)\vspace{0.5mm}\\ \displaystyle
\ t\snr{(\mathcal{a}_{\delta}^{\varepsilon})'(x,t)}\le c\Lambda_{\delta}^{\varepsilon}(x,t),
\end{cases}
$$
with $c\equiv c(n,N,A,\gamma)$.
\end{itemize}
\end{lemma}
\begin{proof}
     Observe that
    $$
    \partial^{2}\tx{f}_{\delta}^{\varepsilon}(x,z)=\mathcal{a}_{\delta}^{\varepsilon}(x,\snr{z})\mathds{I}_{N\times n}+(\mathcal{a}_{\delta}^{\varepsilon})'(x,\snr{z})\snr{z}\left(\frac{z\otimes z}{\snr{z}^{2}}\right).
    $$
The bounds in \eqref{a.7} can then be derived from \eqref{a.5.1} and \eqref{adade} as done in \cite[Section 4.1]{dm21}. While estimate \eqref{lala}$_{1}$ is a direct consequence of \eqref{a.5.1} and of the definition of the regularized eigenvalues, the bound in $\eqref{lala}_{2}$ deserves a brief discussion. Let $z_{1},z_{2}\in \mathbb{R}^{N\times n}$, and, for $s\in (0,1)$, set $z_{s}:=z_{2}+s(z_{1}-z_{2})$. By the mean value theorem we have
\begin{eqnarray*}
\langle\partial\tx{f}_{\delta}^{\varepsilon}(x,z_{1})-\partial\tx{f}_{\delta}^{\varepsilon}(x,z_{2}),z_{1}-z_{2}\rangle&=&\int_{0}^{1}\langle\partial^{2}\tx{f}_{\delta}^{\varepsilon}(x,z_{s})(z_{1}-z_{2}),z_{1}-z_{2}\rangle\ds\nonumber \\
&\stackrel{\eqref{lala}_{1}}{\ge}&c\left(\int_{0}^{1}\lambda_{\delta}^{\varepsilon}(x,\snr{z_{s}})\ds\right)\snr{z_{1}-z_{2}}^{2}\nonumber \\
&\stackrel{\eqref{l60}}{\ge}&c\left(\int_{0}^{1}\lambda_{\delta}(x,\snr{z_{s}})\ds\right)\snr{z_{1}-z_{2}}^{2}\nonumber \\
&&+c\sigma_{\varepsilon}\ell_{1}(\snr{z_{1}}^{2}+\snr{z_{2}}^{2})^{2\gamma-1}\snr{z_{1}+z_{2}}^{2}\nonumber\\
&\stackrel{\eqref{a.55}_{2}}{\ge}&c\lambda_{\delta}(x,\snr{z_{1}}+\snr{z_{2}})\max\left\{\ell_{\delta}(1),\ell_{\delta}(\snr{z_{1}}+\snr{z_{2}})^{-\vartheta}\right\}\nonumber \\
&&\cdot\left(\int_{0}^{1}\min\{\ell_{\delta}(1)^{-1},\ell_{\delta}(\snr{z_{s}})^{\vartheta}\}\ds\right)\snr{z_{1}-z_{2}}^{2}\nonumber \\
&&+c\sigma_{\varepsilon}\ell_{1}(\snr{z_{1}}^{2}+\snr{z_{2}}^{2})^{2\gamma-1}\snr{z_{1}+z_{2}}^{2}\nonumber\\
&\stackrel{\eqref{l60}}{\ge}&c\lambda_{\delta}(x,\snr{z_{1}}+\snr{z_{2}})\max\left\{\ell_{\delta}(1),\ell_{\delta}(\snr{z_{1}}+\snr{z_{2}})^{-\vartheta}\right\}\nonumber \\
&&\cdot \ell_{\delta}(\snr{z_{1}}+\snr{z_{2}})^{\vartheta}\min\left\{1,\frac{1}{\ell_{\delta}(1)\ell_{\delta}(\snr{z_{1}}+\snr{z_{2}})^{\vartheta}}\right\}\snr{z_{1}-z_{2}}^{2}\nonumber \\
&&+c\sigma_{\varepsilon}\ell_{1}(\snr{z_{1}}^{2}+\snr{z_{2}}^{2})^{2\gamma-1}\snr{z_{1}+z_{2}}^{2}\nonumber\\
&\ge&c\lambda_{\delta}^{\varepsilon}(x,\snr{z_{1}}+\snr{z_{2}})\snr{z_{1}-z_{2}}^{2},
\end{eqnarray*}
for $c\equiv c(n,N,A,\gamma,\vartheta)$, and the proof is complete.
\end{proof}
\begin{remark}\label{laed}
By construction, \eqref{a.7} holds also if we replace $\mathcal{a}_{\delta}^{\varepsilon}$, $\bar{\mathcal{a}}_{\delta}^{\varepsilon}$ with $\mathcal{a}_{\delta}$ and $\bar{\mathcal{a}}_{\delta}$ respectively. Of course, now $\Lambda_{\delta}$, $\lambda_{\delta}$ substitute $\Lambda_{\delta}^{\varepsilon}$, $\lambda_{\delta}^{\varepsilon}$ respectively.
\end{remark}
\noindent We highlight some mutual bounds for the auxiliary function $\tx{a}_{\delta}^{\varepsilon}$, the main integrand $A_{\delta}^{\varepsilon}$ and eigenvalue $\lambda_{\delta}^{\varepsilon}$.
\begin{lemma}
For all $(x,t)\in \Omega\times [0,\infty)$, 
\eqn{a.7.1.x}
$$
\begin{cases}
\displaystyle
\ \tx{a}_{\delta}^{\varepsilon}(x,t)\le A_{\delta}^{\varepsilon}(x,t)\le c\rrr_{*}(t)\tx{a}_{\delta}^{\varepsilon}(x,t),\vspace{0.5mm}\\ \displaystyle
\ \lambda_{\delta}^{\varepsilon}(x,t)\ell_{\delta}(t)^{2}\le c\tx{a}_{\delta}^{\varepsilon}(x,t)+c,\vspace{0.5mm}\\ \displaystyle
\ \ell_{1}(t)^{-\mu}+\sigma_{\varepsilon}\ell_{1}(t^{2})^{2\gamma-1}\le c\lambda_{\delta}^{\varepsilon}(x,t), \vspace{0.5mm}\\ \displaystyle
\ \ell_{1}(t)^{2-\mu}+\sigma_{\varepsilon}\ell_{1}(t^{2})^{2\gamma}\le c\tx{a}_{\delta}^{\varepsilon}(x,t)+c, 
\end{cases}
$$
for $c\equiv c(n,N,A,\tx{g},\mu,\gamma,\vartheta_{*})$.
\end{lemma}
\begin{proof}
 The first inequality in $\eqref{a.7.1.x}_{1}$ is a direct consequence of $\eqref{a.7}_{1}$ and \eqref{adade}, while for the second one we have, using \eqref{a.555},
\begin{flalign*}
A_{\delta}^{\varepsilon}(x,t)\stackrel{\eqref{a.7}_{2}}{\le}\int_{0}^{t}\left(\frac{\Lambda_{\delta}^{\varepsilon}(x,s)}{\lambda_{\delta}^{\varepsilon}(x,s)}\right)\lambda_{\delta}^{\varepsilon}(x,s)s\ds\le c\rrr_{*}(t)\tx{a}_{\delta}^{\varepsilon}(x,t),
\end{flalign*}
for $c\equiv c(n,N,A,\tx{g},\mu,\gamma)$. Before proceeding further, let us record that
\eqn{iinf}
$$
\lambda_{\delta}^{\varepsilon}(x,1)=\lambda(x,\ell_{\delta}(1))+\sigma_{\varepsilon}2^{2\gamma-1}\stackrel{\eqref{a.5}_{2}}{\ge} \lambda(x,2)\stackrel{\eqref{a.5}_{1}}{\ge}c\lambda(x,1) \ \Longrightarrow \ \inf_{x\in \Omega}\lambda_{\delta}(x,1)\ge c(A,\mu,\vartheta)>0.
$$
Moreover, by \eqref{a.5.x} and \eqref{a.3} we also obtain
\eqn{suup.1}
$$
\sup_{x\in \Omega}\lambda_{\delta}(x,1)\le c(A,\mu,\vartheta_{*})<\infty.
$$
To achieve instead $\eqref{a.7.1.x}_{2}$, we first notice that if $t\le 1$, then by \eqref{a.55}$_{1}$ and \eqref{a.5.x} we get
$$
\lambda_{\delta}^{\varepsilon}(x,t)\ell_{\delta}(t)^{2}=\left(\lambda_{\delta}^{\varepsilon}(x,t)\ell_{\delta}(t)^{\mu}\right)\ell_{\delta}(t)^{2-\mu}\le c\left(\lambda(x,1)+1\right)\le c,
$$
with $c\equiv c(A,\mu,\gamma)$. On the other hand, if $t\ge 1$, then $\ell_{\delta}(1)\ge \ell_{\delta}(t)^{-\vartheta}$, thus
\begin{eqnarray*}
\lambda_{\delta}^{\varepsilon}(x,t)\ell_{\delta}(t)^{2}&\le& c\lambda_{\delta}(x,t)\left(\int_{1}^{t}s\ds\right)+c\lambda_{\delta}(x,t)\left(\int_{0}^{1}s\ds\right)+c\sigma_{\varepsilon}\ell_{1}(t^{2})^{2\gamma}\nonumber \\
&\stackrel{\eqref{a.55}}{\le}& c\int_{1}^{t}\lambda_{\delta}(x,s)s\ds+c\lambda_{\delta}(x,1)+c\sigma_{\varepsilon}\ell_{1}(t^{2})^{2\gamma}\le c\tx{a}_{\delta}^{\varepsilon}(x,t)+c,
\end{eqnarray*}
for $c\equiv c(A,\gamma,\mu,\vartheta_{*})$. We next control, for $t\in [0,1]$,
\begin{eqnarray*}
\ell_{1}(t)^{-\mu}+\sigma_{\varepsilon}\ell_{1}(t^{2})^{2\gamma-1}&\stackrel{\eqref{iinf}}{\le}&c\ell_{1}(t)^{-\mu}\left(\inf_{x\in \Omega}\lambda_{\delta}(x,1)\right)+\sigma_{\varepsilon}\ell_{1}(t^{2})^{2\gamma-1}\nonumber \\
&\le& c\ell_{1}(t)^{-\mu}\lambda_{\delta}(x,1)+\sigma_{\varepsilon}\ell_{1}(t^{2})^{2\gamma-1}\stackrel{\eqref{a.55}}{\le}c\lambda_{\delta}^{\varepsilon}(x,t),
\end{eqnarray*}
for $c\equiv c(A,\mu)$, while if $t>1$ we have
\begin{eqnarray*}
\lambda_{\delta}(x,t)+\sigma_{\varepsilon}\ell_{1}(t^{2})^{2\gamma-1}&\stackrel{\eqref{a.55}_{1}}{\ge}& c\left(\inf_{x\in \Omega}\lambda_{\delta}(x,1)\right)\ell_{1}(t)^{-\mu}+\sigma_{\varepsilon}\ell_{1}(t^{2})^{2\gamma-1}\nonumber \\
&\stackrel{\eqref{iinf}}{\ge}&c\ell_{1}(t)^{-\mu}+\sigma_{\varepsilon}\ell_{1}(t^{2})^{2\gamma-1},
\end{eqnarray*}
with $c\equiv c(A,\mu)$, and $\eqref{a.7.1.x}_{3}$ is proven. Finally, $\eqref{a.7.1.x}_{4}$ follows from $\eqref{a.7.1.x}_{2,3}$ and the definition of $\tx{a}_{\delta}^{\varepsilon}$.
\end{proof}

\noindent The main oscillation properties of integrand $A_{\delta}^{\varepsilon}$ are described in the following lemma.
\begin{lemma}
     Given any ball $B_{r}\subset \Omega$, 
\eqn{a12.x}
$$
\left\{
\begin{array}{c}
\displaystyle
\snr{\tx{a}_{\delta}^{\varepsilon}(x_{1},t)-\tx{a}_{\delta}^{\varepsilon}(x_{2},t)}+\snr{A_{\delta}^{\varepsilon}(x_{1},t)-A_{\delta}^{\varepsilon}(x_{2},t)}\le c\snr{B_{r}}^{\frac{\alpha}{n}}\ell_{\delta}(t)+c\snr{B_{r}}^{\frac{\alpha}{n}}\left(\inf_{x\in B}A_{\delta}(x,t)\right)^{\vartheta_{*}}\ell_{\delta}(t),\\[10pt]\displaystyle
\snr{\lambda_{\delta}^{\varepsilon}(x_{1},t)-\lambda_{\delta}^{\varepsilon}(x_{2},t)}t+\snr{\mathcal{a}_{\delta}^{\varepsilon}(x_{1},t)-\mathcal{a}_{\delta}^{\varepsilon}(x_{2},t)}t\le c\snr{B_{r}}^{\frac{\alpha}{n}}+c\snr{B_{r}}^{\frac{\alpha}{n}}\left(\inf_{x\in B}A_{\delta}(x,t)\right)^{\vartheta_{*}},
\end{array}
\right.
$$
for all $x_{1},x_{2}\in B$, $t\in [0,\infty)$ and some $c\equiv c(A,\vartheta_{*})$. Moreover, if \eqref{a.5.2s} is in force, then
\eqn{a.5.3s}
$$
\snr{\lambda_{\delta}^{\varepsilon}(x_{1},t)-\lambda_{\delta}^{\varepsilon}(x_{2},t)}t+\snr{\mathcal{a}_{\delta}^{\varepsilon}(x_{1},t)-\mathcal{a}_{\delta}^{\varepsilon}(x_{2},t)}t\le c\snr{B_{r}}^{\frac{\alpha}{n}}\ell_{1}(t)^{\gamma-1},
$$
for $c\equiv c(A,\vartheta_{*},\gamma)$.
\end{lemma}
\begin{proof}
To gain $\eqref{a12.x}_{1}$, we estimate
\begin{eqnarray*}
\snr{A_{\delta}^{\varepsilon}(x_{1},t)-A_{\delta}^{\varepsilon}(x_{2},t)}&=&\snr{A_{\delta}(x_{1},t)-A_{\delta}(x_{2},t)}\le\int_{0}^{t}\snr{\mathcal{a}_{\delta}(x_{1},s)-\mathcal{a}_{\delta}(x_{2},s)}\ell_{\delta}(s)\ds\nonumber \\
&\stackrel{\eqref{a.3}}{\le}&c\snr{B_{r}}^{\frac{\alpha}{n}}t+c\snr{B_{r}}^{\frac{\alpha}{n}}\left(\inf_{x\in B}A(x,\ell_{\delta}(t))\right)^{\vartheta_{*}}t\nonumber \\
&\stackrel{\eqref{ad1}}{\le}&c\snr{B_{r}}^{\frac{\alpha}{n}}\ell_{\delta}(t)+c\snr{B_{r}}^{\frac{\alpha}{n}}\left(\inf_{x\in B}A_{\delta}(x,t)\right)^{\vartheta_{*}}\ell_{\delta}(t),
\end{eqnarray*}
with $c\equiv c(A,\vartheta_{*})$. Inequalities $\eqref{a12.x}_{2}$-\eqref{a.5.3s} are a straightforward consequence of \eqref{a.3}, \eqref{ad1}, and \eqref{a.5.2s} and \eqref{ade.1}, respectively.
\end{proof}
\noindent Let us discuss the limiting behavior of $A_{\delta}^{\varepsilon}$ as $\delta\to 0$, $\varepsilon\to 0$.
\begin{lemma}\label{l2.3}
We have 
\eqn{difdif.1}
$$A_{\delta}^{\varepsilon}(x,t)\to_{\delta\to 0} A(x,t)+\sigma_{\varepsilon} \ell_{1}(t^{2})^{2\gamma}\to_{\varepsilon\to 0} A(x,t),$$ 
uniformly on bounded subsets of $\Omega\times [0,\infty)$. Specifically, 
\eqn{difdif}
$$
\snr{A_{\delta}(x,t)-A(x,t)}\le c\delta\ell_{1}(t)^{\gamma-1},
$$
with $c\equiv c(A,\gamma)$.
      
\end{lemma}
\begin{proof}
To get \eqref{difdif}, via \eqref{adade} we rearrange
$$
A_{\delta}(x,t)=A(x,\ell_{\delta}(t))-\frac{\delta A(x,\ell_{\delta}(t))}{\ell_{\delta}(t)}+\delta\int_{0}^{t}\frac{A(x,\ell_{\delta}(s))}{\ell_{\delta}(s)^{2}}\ds,
$$
and estimate by \eqref{add}, \eqref{d2.2} and the mean value theorem,
\begin{eqnarray*}
\snr{A_{\delta}(x,t)-A(x,t)}&\le&\snr{A(x,\ell_{\delta}(t))-A(x,t)}+\frac{\delta A(x,\ell_{\delta}(t))}{\ell_{\delta}(t)}+\delta\int_{0}^{t}\frac{A(x,\ell_{\delta}(s))}{\ell_{\delta}(s)^{2}}\ds\nonumber \\
&\le&\delta\left(\int_{0}^{1}A'(x,t+s\delta)\ds\right)+c\delta \ell_{\delta}(t)^{\gamma-1}\le c\delta \ell_{1}(t)^{\gamma-1},
\end{eqnarray*}
for $c\equiv c(A,\gamma)$. The convergence in \eqref{difdif.1} is a direct consequence of \eqref{difdif}. 
\end{proof}
\noindent Let us show that integrand $\tx{f}_{\delta}^{\varepsilon}$ can be traced back to standard, controlled polynomial growth condition, that of course will hold in a nonuniform fashion with respect to $\delta$ and $\varepsilon$.
\begin{lemma}\label{l3232}
The integrand $\tx{f}_{\delta}^{\varepsilon}$ satisfies growth/ellipticity conditions
\eqn{corfl.1r}
    $$
    \begin{cases}
    \displaystyle
    \ \sigma_{\varepsilon}\ell_{1}(\snr{z}^{2})^{2\gamma}\le \tx{f}_{\delta}^{\varepsilon}(x,z)\le c\ell_{1}(\snr{z}^{2})^{2\gamma}\vspace{0.5mm}\\\displaystyle
     \ c\sigma_{\varepsilon}\ell_{1}(\snr{z}^{2})^{2\gamma-1}\snr{\xi}^{2}\le \langle\partial^{2}\tx{f}_{\delta}^{\varepsilon}(x,z)\xi,\xi\rangle\vspace{0.5mm}\\ \displaystyle
     \ \snr{\partial^{2}\tx{f}_{\delta}^{\varepsilon}(x,z)}\le c_{\delta}\ell_{1}(\snr{z}^{2})^{2\gamma-1}
    \end{cases}
    $$
    for all $x\in \Omega$, $z,\xi\in \mathbb{R}^{N\times n}$, with $c\equiv c(\data_{0})$, $c_{\delta}\equiv c_{\delta}(\data_{0},\delta)$. Moreover, given any ball $B_{r}\subset \Omega$ and points $x_{1},x_{2}\in B_{r}$, the oscillation bound
    \eqn{corfl.2r}
    $$
    \snr{\partial \tx{f}_{\delta}(x_{1},z)-\partial \tx{f}_{\delta}(x_{2},z)}\le c\snr{B_{r}}^{\frac{\alpha}{n}}\ell_{1}(\snr{z}^{2})^{\frac{4\gamma-1}{2}},
    $$
    is satisfied for $c\equiv c(n,A,\vartheta_{*},\gamma,\alpha)$. 
\end{lemma}
\begin{proof}
The first line in \eqref{corfl.1r} follows from \eqref{adade} and \eqref{add}. Moreover, $\eqref{corfl.1r}_{2}$ is a direct consequence of \eqref{lala}$_{1}$, while for $\eqref{corfl.1r}_{3}$, by \eqref{a.55}$_{2}$, $\eqref{lala}_{1}$, and \eqref{rrr} we have
\begin{eqnarray*}
\snr{\partial^{2}\tx{f}_{\delta}^{\varepsilon}(x,z)}&\le&c\Lambda_{\delta}^{\varepsilon}(x,\snr{z})\le c\left(\frac{\Lambda_{\delta}(x,\snr{z})}{\lambda_{\delta}(x,\snr{z})}\right)\lambda_{\delta}(x,\snr{z})+c\sigma_{\varepsilon}\ell_{1}(\snr{z}^{2})^{2\gamma-1}\nonumber \\
&\le&c\ell_{1}(\snr{z}^{2})^{\mu-1+\omega_{\mu}}\delta^{-\vartheta}\nr{\lambda(\cdot,\delta)}_{L^{\infty}(\Omega)}+ c\sigma_{\varepsilon}\ell_{1}(\snr{z}^{2})^{2\gamma-1}\le c_{\delta}\ell_{1}(\snr{z}^{2})^{2\gamma-1},
\end{eqnarray*}
for $c_{\delta}\equiv c_{\delta}(n,N,A,\tx{g},\mu,\gamma,\vartheta,\delta)$, where we used that $\mu<2$, chose $\omega_{\mu}:=(2\gamma-\mu)/2$ and $\nr{\lambda(\cdot,\delta)}_{L^{\infty}(\Omega)}<\infty$ by continuity. Concerning \eqref{corfl.2r}, it directly follow from \eqref{a12.x}--\eqref{a.5.2s} and \eqref{add}$_{2}$ recalling that $\vartheta_{*}<1$. The proof is complete.

\end{proof}
\noindent The standard uniformly elliptic setting \cite{lie91,dsv09} can be recovered on bounded subsets.
\begin{lemma}\label{holhol.x}
    For any constant $M>0$, H\"older-type conditions
    \eqn{corfl.3}
    $$
    \begin{cases}
    \displaystyle
    \ \snr{(A_{\delta}^{\varepsilon})''(x,t+\tau)-(A_{\delta}^{\varepsilon})''(x,t)}\le c_{M}\left(\frac{\snr{\tau}}{t}\right)^{\beta}(A_{\delta}^{\varepsilon})''(x,t)\vspace{0.5mm}\\ \displaystyle
    \ \snr{(A_{\delta}^{\varepsilon})''(x,t+\tau)-(A_{\delta}^{\varepsilon})''(x,t)}\le c_{\delta;M}\left(\frac{\snr{\tau}}{t}\right)^{\beta}\ell_{1}(t^{2})^{2\gamma-1},
    \end{cases}
    $$
    are satisfied for all $(x,t)\in \Omega\times (0,M]$, $\tau\in \mathbb{R}$ with $\snr{\tau}<t/2$, for $c_M\equiv c_M(A,\tx{g},\mu,\gamma,\vartheta,\beta,M)$, and $c_{\delta;M}\equiv c_{\delta;M}(A,\tx{g},\mu,\gamma,\delta,\theta,\beta,M)$. Furthermore, the uniform ellipticity condition
    \eqn{corfl.4}
    $$
    (A_{\delta}^{\varepsilon})''(x,t)t\approx (A_{\delta}^{\varepsilon})'(x,t)
    $$
    holds for all $(x,t)\in \Omega\times (0,M]$, up to constants depending on $(A,M)$.
\end{lemma}
\begin{proof}
By \eqref{adade} we have
    \eqn{sim}
    $$
    \begin{cases}
    \displaystyle
\ A_{\delta}''(x,t)=\frac{A''(x,\ell_{\delta}(t))t}{\ell_{\delta}(t)}+\frac{\delta \mathcal{a}_{\delta}(x,t)}{\ell_{\delta}(t)}=\bar{\mathcal{a}}_{\delta}(x,t)\vspace{1.5mm}\\
\displaystyle
\ \mathcal{a}_{\delta}'(x,t)=\frac{A''(x,\ell_{\delta}(t))}{\ell_{\delta}(t)}-\frac{\mathcal{a}_{\delta}(x,t)}{\ell_{\delta}(t)},\qquad \quad A_{\delta}'(x,t)=\mathcal{a}_{\delta}(x,t)t.
\end{cases}
    $$
Notice that all quantities involved in $\eqref{sim}_{1}$ are nonnegative by strict convexity \eqref{lala}. Since the validity of \eqref{corfl.3} is standard for power-type functions, \cite[Remark 2.3]{dsv09}, we will focus on controlling the difference $\snr{A_{\delta}''(x,t+\tau)-A_{\delta}''(x,t)}$. Let $\tau\in \mathbb{R}$, $\snr{\tau}<t/2$ and split
\begin{eqnarray*}
\snr{A_{\delta}''(x,t+\tau)-A_{\delta}''(x,t)}&\le& \left|\frac{A''( x,\ell_{\delta}(t+\tau))(t+\tau)}{\ell_{\delta}(t+\tau)}-\frac{A''(x,\ell_{\delta}(t))t}{\ell_{\delta}(t)}\right|\nonumber \\
&&+\delta\left|\frac{\mathcal{a}_{\delta}(x,t+\tau)}{\ell_{\delta}(t+\tau)}-\frac{\mathcal{a}_{\delta}(x,t)}{\ell_{\delta}(t)}\right|
=:\mbox{(I)}+\mbox{(II)}.
\end{eqnarray*}
Keeping in mind Remark \ref{laed}, we bound
\begin{eqnarray*}
\mbox{(I)}&\le&\snr{A''(x,\ell_{\delta}(t+\tau))-A''(x,\ell_{\delta}(t))}\left(\frac{\snr{t+\tau}}{\ell_{\delta}(t+\tau)}\right)\nonumber \\
&&+A''(x,\ell_{\delta}(t))\left|\frac{t+\tau}{\ell_{\delta}(t+\tau)}-\frac{t}{\ell_{\delta}(t)}\right|\nonumber \\
&\stackrel{\eqref{a.5.2}}{\le}&c\left(\frac{A''(x,\ell_{\delta}(t))t}{\ell_{\delta}(t)}\right)\left(\frac{\snr{\tau}}{t}\right)^{\beta}+\frac{c\delta\snr{\tau}A''(x,\ell_{\delta}(t))}{\ell_{\delta}(t)\ell_{\delta}(t+\tau)}\nonumber\\
&\stackrel{\eqref{sim}}{\le}&cA_{\delta}''(x,t)\left(\frac{\snr{\tau}^{\beta}}{t^{\beta}}+\frac{\snr{\tau}}{t}\right)\le cA_{\delta}''(x,t)\left(\frac{\snr{\tau}}{t}\right)^{\beta},
\end{eqnarray*}
for $c\equiv c(A,M)$, and
\begin{eqnarray*}
\mbox{(II)}&\le&\frac{\delta\snr{A'(x,\ell_{\delta}(t+\tau))-A'(x,\ell_{\delta}(t))}}{\ell_{\delta}(t+\tau)^{2}}+\frac{c\delta\snr{\tau}\mathcal{a}_{\delta}(x,t)}{\ell_{\delta}(t)\ell_{\delta}(t+\tau)}\nonumber \\
&\stackrel{\eqref{sim}}{\le}&\frac{c\delta \snr{\tau}}{\ell_{\delta}(t+\tau)^{2}}\left(\int_{0}^{1}A''(x,\ell_{\delta}(t+s\tau))\ds\right)+c\left(\frac{\snr{\tau}}{t}\right)A_{\delta}''(x,t)\nonumber \\
&\le&\frac{c\delta\snr{\tau}}{t\ell_{\delta}(t+\tau)}\left(\int_{0}^{1}\bar{\mathcal{a}}_{\delta}(x,t+s\tau)\ds\right)+c\left(\frac{\snr{\tau}}{t}\right)A_{\delta}''(x,t)\nonumber \\
&\stackrel{\eqref{a.7}_{2}}{\le}&\frac{c\delta\snr{\tau}}{t\ell_{\delta}(t+\tau)}\left(\int_{0}^{1}\Lambda_{\delta}( x,\ell_{\delta}(t+s\tau))\ds\right)+c\left(\frac{\snr{\tau}}{t}\right)A_{\delta}''(x,t)\nonumber \\
&\stackrel{\eqref{a.555}}{\le}&\frac{c\delta\snr{\tau}\rrr_{*}(M)}{t\ell_{\delta}(t+\tau)}\left(\int_{0}^{1}\lambda_{\delta}(x,t+s\tau)\ds\right)+c\left(\frac{\snr{\tau}}{t}\right)A_{\delta}''(x,t)\nonumber \\
&\stackrel{\eqref{a.55}_{2},\eqref{l60}}{\le}&\frac{c\delta\snr{\tau}\rrr_{*}(M)\lambda_{\delta}(x,t)}{t\ell_{\delta}(t+\tau)}+c\left(\frac{\snr{\tau}}{t}\right)A_{\delta}''(x,t)\nonumber \\
&\stackrel{\eqref{a.7}_{1}}{\le}&c\rrr_{*}(M)\left(\frac{\snr{\tau}}{t}\right)\left(\frac{\delta\mathcal{a}_{\delta}(x,t)}{\ell_{\delta}(t)}\right)+c\left(\frac{\snr{\tau}}{t}\right)A_{\delta}''(x,t)\stackrel{\eqref{sim}}{\le}c\left(\frac{\snr{\tau}}{t}\right)A_{\delta}''(x,t),
\end{eqnarray*}
with $c\equiv c(A,\tx{g},\mu,\vartheta,M)$. Merging the content of the two previous displays we obtain \eqref{corfl.3}$_{1}$. Finally, $\eqref{corfl.3}_{2}$ can be derived from $\eqref{corfl.3}_{1}$, as
\begin{eqnarray*}
    \snr{(A_{\delta}^{\varepsilon})''(x,t+\tau)-(A_{\delta}^{\varepsilon})''(x,t)}&\stackrel{\eqref{sim},\eqref{a.7}_{2}}{\le}& c\left(\frac{\snr{\tau}}{t}\right)^{\beta}\left(\Lambda_{\delta}(x,t)+\sigma_{\varepsilon}\ell_{1}(t^{2})^{2\gamma-1}\right)\nonumber \\
    &\stackrel{\eqref{a.555}}{\le}& c\left(\frac{\snr{\tau}}{t}\right)^{\beta}\left(\rrr_{*}(M)\lambda_{\delta}(x,t)+\sigma_{\varepsilon}\ell_{1}(t^{2})^{2\gamma-1}\right)\nonumber \\
    &\stackrel{\eqref{a.55}_{2}}{\le}&c\left(\frac{\snr{\tau}}{t}\right)^{\beta}\left(\delta^{-\vartheta}\nr{\lambda(\cdot,\delta)}_{L^{\infty}(\Omega)}+\sigma_{\varepsilon}\ell_{1}(t^{2})^{2\gamma-1}\right)\nonumber \\
    &\le&c_{\delta}\left(\frac{\snr{\tau}}{t}\right)^{\beta}\ell_{1}(t^{2})^{2\gamma-1},
\end{eqnarray*}
for $c_{\delta}\equiv c_{\delta}(A,\tx{g},\mu,\gamma,\delta)$. Next, notice that power-type functions such as $t\mapsto \ell_{1}(t^{2})^{2\gamma}$ are well-known to be uniformly elliptic, so let us take care of $A_{\delta}$. By $\eqref{sim}$ and Remark \ref{laed} we have
\begin{eqnarray*}
\begin{cases}
\displaystyle
\ A_{\delta}''(x,t)t\stackrel{\eqref{a.7}_{1}}{\ge}t\lambda_{\delta}(x,t)\stackrel{\eqref{a.555}}{\ge}\frac{t\Lambda_{\delta}(x,t)}{\rrr_{*}(M)}\stackrel{\eqref{a.7}_{2}}{\ge}\frac{t\mathcal{a}_{\delta}(x,t)}{\rrr_{*}(M)}\stackrel{\eqref{sim}_{2}}{=}\frac{A_{\delta}'(x,t)}{\rrr_{*}(M)}\vspace{1.5mm}\\ \displaystyle
\ A'_{\delta}(x,t)\stackrel{\eqref{a.7}_{1}}{\ge}t\lambda_{\delta}(x,t)\stackrel{\eqref{a.555}}{\ge}\frac{t\Lambda_{\delta}(x,t)}{\rrr_{*}(M)}\stackrel{\eqref{a.7}_{2}}{\ge}\frac{t\bar{\mathcal{a}}_{\delta}(x,t)}{\rrr_{*}(M)}\stackrel{\eqref{sim}_{1}}{=}\frac{tA_{\delta}''(x,t)}{\rrr_{*}(M)},
\end{cases}
\end{eqnarray*}
and the proof is complete.
\end{proof}
\noindent We conclude this section with a direct consequence of Lemma \ref{holhol.x}.
\begin{corollary}\label{c39}
There exists $M_{*}\equiv M_{*}(A,\gamma)\ge 1$ such that for any $M\ge M_{*}$, a strictly convex integrand $A_{M}\colon \Omega\times [0,\infty)\to \mathbb{R}$ can be constructed satisfying
\eqn{h'''}
$$
A_{M}(x,t)=A_{\delta}^{\varepsilon}(x,t)\qquad \mbox{for all} \ \ (x,t)\in \Omega\times [0,M].
$$
as well as
\eqn{h''}
$$
\begin{cases}
\displaystyle
\ t\mapsto A_{M}(\cdot,t)\in C^{2}_{\loc}[0,\infty),\qquad \quad x\mapsto A_{M}(x,\cdot)\in C^{0,\alpha}(\Omega)\vspace{1.5mm}\\ 
    \displaystyle
    \ A_{M}''(x,t)t\approx_{A,\gamma,\mu,M} A_{M}'(x,t)\vspace{1.5mm}\\
    \displaystyle
    \ \snr{A_{M}''(x,t+\tau)-A_{M}''(x,t)}\le c\left(\frac{\snr{\tau}}{t}\right)^{\beta}A_{M}''(x,t),
\end{cases}
$$
for all $(x,t)\in \Omega\times (0,\infty)$, $\tau\in \mathbb{R}$ with $\snr{\tau}\le t/2$, and some $\beta\in (0,1)$, $c\equiv c(A,\gamma,\mu,\vartheta,\beta,M)$. 
\end{corollary}
\begin{proof}
Following \cite[Section 5]{bs15}, we pick a cut-off function $\eta_{M}\in C^{\infty}_{c}(\mathbb{R}^{N\times n})$ such that $\mathds{1}_{B_{2M}}\le \eta_{M}\le \mathds{1}_{B_{4M}}$, $\snr{D\eta_{M}}\le 2M^{-1}$, $\snr{D^{2}\eta_{M}}\le 4M^{-2}$, $\snr{D^{3}\eta_{M}}\le 8M^{-3}$, and set
\eqn{amam}
$$
\tx{h}(t):=(t^{2}-M^{2})_{+}^{4\gamma},\qquad \quad A_{M}(x,t):=\eta_{M}(t)\left(A_{\delta}^{\varepsilon}(x,t)-A_{\delta}^{\varepsilon}(x,0)\right)+\tx{h}(t).
$$
By construction and \eqref{regreg}, we see that $t\mapsto A_{M}(\cdot,t)\in C^{2}_{\loc}[0,\infty)$, $x\mapsto A_{M}(x,\cdot)\in C^{0,\alpha}(\Omega)$, and $A_{M}(x,t)=A_{\delta}^{\varepsilon}(x,t)$ for all $(x,t)\in \Omega\times [0,M]$, so that $\eqref{h''}_{1}$ and \eqref{h'''} are proven. We then deal separately with the proof of $\eqref{h''}_{2}$, $\eqref{h''}_{3}$, and the strict convexity of $t\mapsto A_{M}(\cdot,t)$. 
\subsubsection*{Proof of $\eqref{h''}_{2}$} A direct computation gives
$$
\begin{cases}
\displaystyle
\ A_{M}'(x,t)=\eta_{M}'(t)\left(A_{\delta}^{\varepsilon}(x,t)-A_{\delta}^{\varepsilon}(x,0)\right)+\eta_{M}(t)(A_{\delta}^{\varepsilon})'(x,t)+\tx{h}'(t)\vspace{1.5mm}\\ \displaystyle
\ A_{M}''(x,t)=\eta_{M}''(t)\left(A_{\delta}^{\varepsilon}(x,t)-A_{\delta}^{\varepsilon}(x,0)\right)+2\eta_{M}'(t)(A_{\delta}^{\varepsilon})'(x,t)+\eta_{M}(t)(A_{\delta}^{\varepsilon})''(x,t)+\tx{h}''(t),
\end{cases}
$$
where (for the reader's sake),
\eqn{hhhh}
$$
\tx{h}'(t)=8\gamma(t^{2}-M^{2})_{+}^{4\gamma-1}t,\qquad \quad \tx{h}''(t)=8\gamma(t^{2}-M^{2})_{+}^{4\gamma-2}(2(4\gamma-1)t^{2}+(t^{2}-M^{2})_{+}).
$$
Before proceeding further, let $\theta\in (0,1)$ to be fixed in a few lines, and record the elementary implications:
\eqn{hh.6}
$$
\begin{cases}
    \displaystyle
    \ M+\theta\le t \ \Longrightarrow \ (t^{2}-M^{2})_{+}\ge \theta^{2}\vspace{1.5mm}\\ \displaystyle
    \ t\ge 2M \ \Longrightarrow \ \tx{h}'(t)\ge c(\gamma)t^{8\gamma-1}\vspace{1.5mm}\\ \displaystyle
    \ t\ge 2M \ \Longrightarrow\ \tx{h}''(t)t\approx_{\gamma,M}\tx{h}'(t).
\end{cases}
$$
Moreover, since $t\mapsto A_{\delta}(\cdot,t)$ is strictly convex, cf. \eqref{lala}, $t\mapsto A_{\delta}'(\cdot,t)$ is increasing, so
\eqn{hh.1}
$$
A'_{\delta}(x,t)\stackrel{\eqref{adade}}{=}\mathcal{a}_{\delta}(x,t)t\ge \mathcal{a}_{\delta}(x,1)\stackrel{\eqref{a.7}_{1}}{\ge} c\lambda_{\delta}(x,1)\stackrel{\eqref{iinf}}{\ge} c_{-}(A,\gamma,\mu,\vartheta)>0,
$$
where we also used Remark \ref{laed}. Set 
\eqn{the}
$$
\theta:=\min\left\{\frac{1}{2},\left(\frac{c_{-}}{3^{4\gamma+2}\gamma M^{4\gamma+1}}\right)^{\frac{1}{2(2\gamma-1)}}\right\}\in (0,1).
$$ 
Now assume that $0<t\le M+\theta$. Clearly, we can suppose $M\le t\le M+\theta$, otherwise if $0< t\le M$, $A_{M}(x,t)=A_{\delta}^{\varepsilon}(x,t)$ and we can conclude this part of the proof directly via \eqref{corfl.3}--\eqref{corfl.4}. By definition, now $A_{M}(x,t)=A_{\delta}^{\varepsilon}(x,t)+\tx{h}(t)$ and
\begin{eqnarray}\label{hh.2}
A_{M}''(x,t)t&=&(A_{\delta}^{\varepsilon})''(x,t)t+\tx{h}''(t)t\nonumber \\
&\stackrel{\eqref{corfl.4}}{\ge}&c(A_{\delta}^{\varepsilon})'(x,t)+16\gamma(4\gamma-1)(t^{2}-M^{2})_{+}^{4\gamma-2}t^{3}+8\gamma(t^{2}-M^{2})_{+}^{4\gamma-1}t\nonumber \\
&\stackrel{\eqref{hhhh}}{\ge}&c(A_{\delta}^{\varepsilon})'(x,t)+\tx{h}'(t)=cA_{M}'(x,t)\nonumber \\
&\stackrel{\eqref{corfl.4}}{\ge}&cA_{\delta}'(x,t)+c(A_{\delta}^{\varepsilon})''(x,t)t+4\gamma(t^{2}-M^{2})_{+}^{4\gamma-1}t\nonumber \\
&&+4\gamma(t^{2}-M^{2})_{+}^{4\gamma-2}t^{3}-4\gamma(t^{2}-M^{2})_{+}^{4\gamma-2}M^{2}t\nonumber \\
&\stackrel{\eqref{hh.1}}{\ge}&c_{-}+c(A_{\delta}^{\varepsilon})''(x,t)t+4\gamma(t^{2}-M^{2})_{+}^{4\gamma-1}t\nonumber \\
&&+4\gamma(t^{2}-M^{2})_{+}^{4\gamma-2}t^{3}-3^{4\gamma+2}\gamma M^{4\gamma+1}\theta^{2(2\gamma-1)}\stackrel{\eqref{the}}{\ge}cA_{M}''(x,t)t,
\end{eqnarray}
for all $(x,t)\in \Omega\times [M,M+\theta]$ and some $c\equiv c(A,\gamma,\mu,M)$. Next, if $M+\theta\le t\le 2M$ we look back at \eqref{hh.2}, third line, and estimate
\begin{eqnarray*}
A_{M}''(x,t)t&\ge&c(A_{\delta}^{\varepsilon})'(x,t)+\tx{h}'(t)=cA_{M}'(x,t)\nonumber \\
&\stackrel{\eqref{hh.6}_{1}}{\ge}&c(A_{\delta}^{\varepsilon})''(x,t)t+4\gamma(t^{2}-M^{2})_{+}^{4\gamma-1}t+4\theta^{2}M^{-2}\gamma(t^{2}-M^{2})_{+}^{4\gamma-2}t^{3}\nonumber \\
&\stackrel{\eqref{hhhh}}{\ge}&c(A_{\delta}^{\varepsilon})''(x,t)t+c\tx{h}''(t)t=cA_{M}''(x,t)t,
\end{eqnarray*}
for $c\equiv c(A,\gamma,\mu,M)$. Furthermore, if $t\ge 4M$, then $A_{M}(x,t)=\tx{h}(t)$ and $\eqref{hh.6}_{3}$ holds. We only need to check what happens for $2M\le t\le 4M$. As in this case $\eqref{hh.6}_{3}$ holds, we only need to control the terms depending on $A_{\delta}^{\varepsilon}$ and $\eta_{M}$. In this respect, we bound
\begin{eqnarray}\label{hh.3}
\tx{X}(x,t)&:=&\left(\snr{\eta'(t)}+\snr{\eta''(t)}\right)\snr{A_{\delta}^{\varepsilon}(x,t)-A_{\delta}^{\varepsilon}(x,0)}\nonumber \\
&&+2\snr{\eta'(t)}(A_{\delta}^{\varepsilon})'(x,t)\le\frac{12}{M}(A_{\delta}^{\varepsilon})'(x,t)t\stackrel{\eqref{d2.2}}{\le}\frac{ct^{4\gamma}}{M}\stackrel{\eqref{hh.6}_{2}}{\le}\frac{c_{\gamma}(A,\gamma)\tx{h}'(t)}{M^{4\gamma-1}},
\end{eqnarray}
 where we used again that $t\mapsto (A_{\delta}^{\varepsilon})'(\cdot,t)$ is increasing. Set 
 $$
 M_{*}:=\max\left\{\left(2^{10\gamma}c_{\gamma}\right)^{\frac{1}{4\gamma-2}},2\right\} \ \Longrightarrow \ M_{*}\equiv M_{*}(A,\gamma).
 $$
 We then estimate
\begin{eqnarray}\label{hh.20}
A_{M}''(x,t)t&\ge&\eta_{M}(t)(A_{\delta}^{\varepsilon})''(x,t)t+\tx{h}''(t)t-\tx{X}(x,t)\nonumber \\
&\stackrel{\eqref{hh.3},\eqref{corfl.4}}{\ge}&c\eta_{M}(t)(A_{\delta}^{\varepsilon})'(x,t)+\tx{h}'(t)-\tx{X}(x,t)\nonumber \\
&\ge&cA_{M}'(x,t)+\frac{\tx{h}'(t)}{2}-(1+c)\tx{X}(x,t)\nonumber \\
&\stackrel{\eqref{hh.3}}{\ge}&cA'_{M}(x,t)+\frac{\tx{h}'(t)}{2}\left(1-\frac{4c_{\gamma}}{M^{4\gamma-1}}\right)\nonumber \\
&\stackrel{M\ge M_{*}}{\ge}&cA_{M}'(x,t)+\frac{\tx{h}'(t)}{4}\ge cA_{M}'(x,t),
\end{eqnarray}
for $c\equiv c(A,\gamma,\mu,M)$ and, similarly,
\begin{eqnarray}\label{hh.21}
A_{M}'(x,t)&\ge&c\left(\eta_{M}(t)(A_{\delta}^{\varepsilon})'(x,t)+\tx{h}'(t)\right)+\frac{\tx{h}'(t)}{2}-c\tx{X}(x,t)\nonumber \\
&\stackrel{\eqref{corfl.4}}{\ge}&c\left(\eta_{M}(t)(A_{\delta}^{\varepsilon})''(x,t)t+\tx{h}''(t)t\right)+\frac{\tx{h}'(t)}{2}-c\tx{X}(x,t)\nonumber \\
&\ge&cA_{M}''(x,t)+\frac{\tx{h}'(t)}{2}-4\tx{X}(x,t)t\nonumber \\
&\stackrel{\eqref{hh.3}}{\ge}&cA_{M}''(x,t)+\tx{h}'(t)\left(\frac{1}{8}-\frac{16c_{\gamma}}{M^{4\gamma-2}}\right)\stackrel{M\ge M_{*}}{\ge}cA_{M}''(x,t),
\end{eqnarray}
with $c\equiv c(A,\gamma,\mu,M)$.\footnote{Since we are bounding $A_{M}''(x,t)t$ and $A_{M}'(x,t)$ below, there is no loss of generality in taking constant $c$ appearing in displays \eqref{hh.20}--\eqref{hh.21} less than one.} 
\subsubsection*{Proof of the strict convexity of $t\mapsto A_{M}(\cdot,t)$} Concerning the strict convexity of $t\mapsto A_{M}(\cdot,t)$, if $0\le t\le 2M$, $A_{M}(x,t)=A_{\delta}^{\varepsilon}(x,t)+\tx{h}(t)$, and the conclusion follows by definition and \eqref{lala}. On the other hand, if $2M<t<\infty$, 
\begin{eqnarray}\label{hh.22}
A_{M}'(x,t)&\ge& \eta_{M}(t)(A_{\delta}^{\varepsilon})'(x,t)+\tx{h}'(t)-\tx{X}(x,t)\nonumber \\
&\ge& \eta_{M}(t)(A_{\delta}^{\varepsilon})'(x,t)+c\tx{h}'(t)\left(1-\frac{c_{\gamma}}{M^{5\gamma-1}}\right)\nonumber \\
&\stackrel{M\ge M_{*}}{\ge}&\eta_{M}(t)(A_{\delta}^{\varepsilon})'(x,t)+c\tx{h}'(t)\nonumber \\
&\stackrel{\eqref{hh.6}_{2}}{\ge}&\eta_{M}(t)(A_{\delta}^{\varepsilon})'(x,t)+c\stackrel{\eqref{lala}}{\ge}c(\gamma,A)>0,
\end{eqnarray}
so by $\eqref{h''}_{2}$, we have the strictly convex of $t\mapsto A_{M}(\cdot,t)$ for all $(x,t)\in \Omega\times [0,\infty).$
\subsubsection*{Proof of $\eqref{h''}_{3}$} If $0<t\le M/2$, then $t/2\le t+\tau\le M$ and by \eqref{h'''} we can conclude with $\eqref{corfl.3}_{1}$, while if $t>8M$ we can apply standard estimates for power-type functions as $A_{M}=\tx{h}$. This means that we can assume that $M/2<t\le 8M$. Direct computations and the mean value theorem give, for $M/2<t\le M+1$,
\begin{eqnarray*}
\snr{\tx{h}''(t+\tau)-\tx{h}''(t)}&\le& \snr{\tau}\max\{\tx{h}'''(t+\tau),\tx{h}'''(t)\}\le c\snr{\tau}\nonumber \\
&\stackrel{\eqref{hh.22}}{\le}&c\snr{\tau}A_{M}'(x,t)\stackrel{\eqref{h''}_{2}}{\le} c\snr{\tau}A_{M}''(x,t)t\stackrel{t\le M+1}{\le} c\left(\frac{\snr{\tau}}{t}\right)A_{M}''(x,t),
\end{eqnarray*}
where we also used that if $t\le M+1$, then $A_{M}=A_{\delta}^{\varepsilon}+\tx{h}$ and so $c\equiv c(A,\gamma,\mu,M)$. On the other hand, if $M+1<t\le 8M$, since $(t^{2}-M^{2})_{+}\ge 1$, we directly have
$$
\snr{\tx{h}''(t+\tau)-\tx{h}''(t)}\le c(\gamma,M)\left(\frac{\snr{\tau}}{t}\right)\tx{h}''(t).
$$
With the three previous estimate at hand, by \eqref{d2.2}, the mean value theorem, \eqref{add}, \eqref{a.55}--\eqref{a.7}, \eqref{corfl.3}--\eqref{h''}$_{2}$, and \eqref{hh.22} we finally bound
\begin{eqnarray*}
\snr{A_{M}''(x,t+\tau)-A_{M}''(x,t)}&\le&\snr{\eta_{M}''(t+\tau)}\snr{A_{\delta}^{\varepsilon}(x,t+\tau)-A_{\delta}^{\varepsilon}(x,t)}\nonumber \\
&&+\snr{A_{\delta}^{\varepsilon}(x,t)-A_{\delta}^{\varepsilon}(x,0)}\snr{\eta_{M}''(t+\tau)-\eta_{M}''(t)}\nonumber \\
&&+2\snr{\eta_{M}'(t+\tau)}\snr{(A_{\delta}^{\varepsilon})'(x,t+\tau)-(A_{\delta}^{\varepsilon})'(x,t)}\nonumber \\
&&+2\snr{\eta_{M}'(t+\tau)-\eta_{M}'(t)}(A_{\delta}^{\varepsilon})'(x,t)\nonumber \\
&&+\eta_{M}(t+\tau)\snr{(A_{\delta}^{\varepsilon})''(x,t+\tau)-(A_{\delta}^{\varepsilon})''(x,t)}\nonumber \\
&&+(A_{\delta}^{\varepsilon})''(x,t)\snr{\eta_{M}(t+\tau)-\eta_{M}(t)}+\snr{\tx{h}''(t+\tau)-\tx{h}''(t)}\nonumber \\
&\le&c\left(\frac{\snr{\tau}}{t}\right)M^{4\gamma}+c\left(\frac{\snr{\tau}}{t}\right)\left(\rrr_{*}(M)\nr{\lambda(\cdot,1)}_{L^{\infty}(\Omega)}+\ell_{1}(t^{2})^{2\gamma-1}\right)\nonumber \\
&&+c\left(\frac{\snr{\tau}}{t}\right)^{\beta}(A_{\delta}^{\varepsilon})''(x,t)+c\left(\frac{\snr{\tau}}{t}\right)\left(A_{M}''(x,t)+\tx{h}''(t)\right)\nonumber \\
&\le&c\left(\frac{\snr{\tau}}{t}\right)^{\beta}A_{M}'(x,t)\le c\left(\frac{\snr{\tau}}{t}\right)^{\beta}A_{M}''(x,t),
\end{eqnarray*}
for $c\equiv c(A,\gamma,\mu,\beta,\vartheta,M)$. The proof is complete.
\end{proof}
\subsection{Regularized Dirichlet problems}\label{rdp}
Let $u\in W^{1,1}_{\loc}(\Omega,\mathbb{R}^{N})$ be a local minimizer of functional $\mathcal{F}$, and $B\Subset \Omega$ be a ball. Recalling \eqref{a.3} and the first bullet of Remark \ref{rmp}, \cite[Theorem 2.3]{buli} applies\footnote{In \cite{buli}, quantitatively superlinear integrands are considered, that is $t^{p}-1\lesssim A(x,t)\le 1+t^{q}$ for some $1<p\le q$ and all $(x,t)\in \Omega\times [0,\infty)$. Here, we do not control below any power larger than one, however, thanks to \eqref{a.4}--\eqref{binf} Dunford \& Pettis and de la Vallée Poussin theorem assure the compactness of approximating sequences in $W^{1,1}$, see also \cite[Section 5.2]{dm23b} and \cite[Lemma 3.2]{ddp24}.} and yields a sequence $\ti{u}_{\varepsilon}\in C^{\infty}(B,\mathbb{R}^{N})$ such that
\eqn{conv}
$$
\ti{u}_{\varepsilon}\to u \ \ \mbox{strongly in} \ \ W^{1,1}(B,\mathbb{R}^{N})\qquad \mbox{and}\qquad \mathcal{F}(\ti{u}_{\varepsilon};B)\to \mathcal{F}(u;B).
$$
\noindent Let us denote $\sigma_{\varepsilon}:=\left(10+\varepsilon^{-1}+\varepsilon^{-1}\nr{D\ti{u}_{\varepsilon}}_{L^{4\gamma}(B)}^{8\gamma}\right)^{-1}$, observe that 
\eqn{oe}
$$\texttt{o}(\varepsilon)\equiv \sigma_{\varepsilon}\nr{D\ti{u}_{\varepsilon}}_{L^{4\gamma}(B)}^{4\gamma}\to 0,$$ and introduce functional
\eqn{funed}
$$
W^{1,4\gamma}(B,\mathbb{R}^{N})\ni w\mapsto \mathcal{F}_{\delta}^{\varepsilon}(w;B):=\int_{B}\tx{f}_{\delta}^{\varepsilon}(x,Dw)\dx,
$$
with integrand $\tx{f}_{\delta}^{\varepsilon}$ constructed as in Section \ref{dgc}, corresponding to the choice of $\sigma_{\varepsilon}$ made above. By strict convexity, cf. $\eqref{lala}_{1}$, and direct methods, the Dirichlet problem
\eqn{pded}
$$
\ti{u}_{\varepsilon}+W^{1,4\gamma}_{0}(B,\mathbb{R}^{N})\ni w\mapsto \min_{\ti{u}_{\varepsilon}+W^{1,4\gamma}_{0}(B,\mathbb{R}^{N})}\mathcal{F}_{\delta}^{\varepsilon}(w;B)
$$
admits a unique solution $u_{\delta}^{\varepsilon}\in \ti{u}_{\varepsilon}+W^{1,4\gamma}_{0}(B,\mathbb{R}^{N})$, that is a weak solution to system
\eqn{elsed}
$$
\int_{B}\langle\partial\tx{f}_{\delta}^{\varepsilon}(x,Du_{\delta}^{\varepsilon}),Dw\rangle\dx=0\qquad \mbox{for all} \ \ w\in W^{1,4\gamma}_{0}(B,\mathbb{R}^{N}).
$$
Moreover, by Lemma \ref{l3232} we see that the integrand $\tx{f}_{\delta}^{\varepsilon}$ satisfies standard, controlled growth conditions \eqref{corfl.1r}, thus by now classical regularity theory applies \cite{tol83} and 
\eqn{areg}
$$
u_{\delta}^{\varepsilon}\in W^{1,\infty}_{\loc}(B,\mathbb{R}^{N}).
$$
This will be crucial for deriving uniform Lipschitz estimates, cf. Section \ref{lip}.

\begin{remark}[Notation alert]\label{remnot}
To streamline expressions, we suppress the indices $\varepsilon$ and $\delta$ throughout this section. Whenever a quantity depends on both parameters, we write it without decoration, e.g.: $\tx{f}_{\delta}^{\varepsilon}\equiv \tx{f}$; whenever it depends only on $\delta$, we place a tilde on the symbol, e.g.: $\tx{f}_{\delta}\equiv \ti{\tx{f}}$. All objects should therefore be understood in their indexed form, and the full notation will be reinstated at the end of Section \ref{lip}.
\end{remark}
\subsection{Quantitative higher integrability} Let us prove a quantitative higher integrability result for the solution to problem \eqref{pded} under (unbalanced) polynomial growth conditions. More precisely, we have the following theorem.
\begin{theorem}\label{thi}
Assume \eqref{add}, \eqref{lala}, \eqref{a.5.3s} and
\eqn{gm}
$$
1<\gamma<2-\mu+\frac{\alpha}{n}.
$$
The solution $u\in \ti{u}_{\varepsilon}+W^{1,4\gamma}_{0}(B,\mathbb{R}^{N})$ to the Dirichlet problem \eqref{pded} satisfies
\eqn{12.2.1}
$$
Du\in L^{t}_{\loc}(B,\mathbb{R}^{N\times n})\qquad \mbox{for all} \ \ t\in \left[1,\frac{n(2-\mu)}{n-2\alpha}\right).
$$
Moreover, for every ball $B_{r}\Subset B$ with radius $r\in (0,1]$, 
\eqn{12.2}
    $$
\nr{Du}_{L^{t}(B_{r/2})}\le \frac{c}{r^{\tx{d}}}\left(\nr{Du}_{L^{1}(B_{r})}+\sqrt{\sigma_{\varepsilon}}\nr{Du}_{L^{4\gamma}(B_{r})}^{2\gamma}+1\right)^{\tx{d}},    
    $$
with $c\equiv c(n,N,A,\mu,\alpha,\gamma)$ and $\tx{d}\equiv \tx{d}(n,\alpha,\gamma,\mu)$. Specifically, $t=2\gamma-2+\mu$ is admissible in \eqref{12.2.1}--\eqref{12.2}.
\end{theorem}
\begin{proof}
Let $B_{r}\Subset B$ be a ball with radius $r\in (0,1]$, $r/2\le \tau_{2}<\tau_{1}\le r$ be parameters, set $\ti{\tau}_{2}:=(\tau_{1}+\tau_{2})/2$, $\ti{\tau}_{1}:=(2\tau_{1}+\tau_{2})/3$, let $\eta\in C^{2}_{c}(B_{r})$ be a cut-off function satisfying $\mathds{1}_{B_{\ti{\tau}_{2}}}\le \eta\le \mathds{1}_{B_{\ti{\tau}_{1}}}$ and $\snr{D\eta}^{2}+\snr{D^{2}\eta}\lesssim (\tau_{1}-\tau_{2})^{-2}$, and pick a vector $h\in \mathbb{R}^{n}\setminus \{0\}$ with $\snr{h}<10^{-4}(\tau_{1}-\tau_{2})$. Define 
$$
\mathcal{D}_{h}:=\snr{Du(x)}+\snr{Du(x+h)},\qquad \quad \mathcal{N}_{\infty}:=1+\nr{D\eta}_{L^{\infty}(B)}^{2}+\nr{D^{2}\eta}_{L^{\infty}(B)},
$$
and test \eqref{elsed} against $w:=\tau_{-h}(\eta^{2}\tau_{h}u)$, admissible as $u\in W^{1,4\gamma}(B,\mathbb{R}^{N})$ to gain, after a few standard manipulations involving the integration-by-parts formula for finite differences and Leibnitz rule \eqref{prod},
\begin{eqnarray*}
\mbox{(I)}&:=&\int_{B}\eta^{2}\langle\partial \tx{f}(x,Du(x+h))-\partial\tx{f}(x,Du(x)),\tau_{h}Du\rangle\dx\nonumber \\
&=&-2\int_{B}\eta\langle \tau_{h}(\partial \tx{f}(\cdot,Du)),\tau_{h}u\otimes D\eta\rangle\dx\nonumber \\
&&-\int_{B}\eta^{2}\langle\partial \tx{f}(x+h,Du(x+h))-\partial\tx{f}(x,Du(x+h)),\tau_{h}Du\rangle\dx=:\mbox{(II)}+\mbox{(III)}.
\end{eqnarray*}
We then estimate, by the mean value theorem and \eqref{l60},
\begin{eqnarray*}
\mbox{(I)}&=&\int_{B}\eta^{2}\left(\int_{0}^{1}\langle\partial^{2}\tx{f}(x,Du(x)+s\tau_{h}Du(x))\tau_{h}Du,\tau_{h}Du\rangle\ds\right)\dx\nonumber \\
&\stackrel{\eqref{lala}}{\ge}&\int_{B}\eta^{2}\left(\int_{0}^{1}\lambda(x,\snr{Du+s\tau_{h}Du})\ds\right)\snr{\tau_{h}Du}^{2}\dx\nonumber \\
&\stackrel{\eqref{a.55}}{\ge}&c\int_{B}\eta^{2}\ti{\lambda}(x,\mathcal{D}_{h})\snr{\tau_{h}Du}^{2}\dx\nonumber \\
&&+c\sigma_{\varepsilon}\int_{B}\eta^{2}\left(\int_{0}^{1}\ell_{1}(\snr{Du+s\tau_{h}Du}^{2})^{2\gamma-1}\ds\right)\snr{\tau_{h}Du}^{2}\dx\nonumber \\
&\stackrel{\eqref{l60}}{\ge}&c\int_{B}\eta^{2}\left(\ti{\lambda}(x,\mathcal{D}_{h})+\sigma_{\varepsilon}\ell_{1}(\mathcal{D}_{h}^{2})^{2\gamma-1}\right)\snr{\tau_{h}Du}^{2}\dx\nonumber\\
&\stackrel{\eqref{a.7.1.x}_{3}}{\ge}&c\int_{B}\eta^{2}\left(\ell_{1}(\mathcal{D}_{h})^{-\mu}+\sigma_{\varepsilon}\ell_{1}(\mathcal{D}_{h}^{2})^{2\gamma-1}\right)\snr{\tau_{h}Du}^{2}\dx\nonumber \\
&\stackrel{\eqref{Vm}}{\ge}&c\int_{B}\eta^{2}\left(\snr{\tau_{h}V_{1,2-\mu}(Du)}^{2}+\sigma_{\varepsilon}\snr{\tau_{h}V_{1,4\gamma}(Du)}^{2}\right)\dx,
\end{eqnarray*}
with $c\equiv c(n,N,A,\mu,\gamma)$. Moreover, we bound using \eqref{af}, \eqref{d2.2}, \eqref{Vm}, Cauchy-Schwarz and Young's inequalities,
\begin{eqnarray*}
\snr{\mbox{(II)}}&\le &2\snr{h}\left|\int_{B}\left(\int_{0}^{1}\langle\partial\ti{\tx{f}}(x+th,Du(x+th)),\partial_{h/\snr{h}}(\eta D\eta\otimes \tau_{h}u)\rangle\dt\right)\right|\nonumber \\
&&+8\gamma\sigma_{\varepsilon}\int_{B}\langle\tau_{h}\left(\ell_{1}(\snr{Du}^{2})^{2\gamma-1}Du\right),\eta D\eta\otimes \tau_{h}u\rangle\dx\nonumber \\
&\le&c\snr{h}\int_{B}\left(\int_{0}^{1}\ell_{1}(\snr{Du(x+th)})^{\gamma-1}\dt\right)\left(\snr{D\eta}^{2}+\eta\snr{D^{2}\eta}\right)\snr{\tau_{h}u}\dx\nonumber \\
&&+c\snr{h}\int_{B}\left(\int_{0}^{1}\ell_{1}(\snr{Du(x+th)})^{\gamma-1}\dt\right)\eta\snr{D\eta}\snr{\tau_{h}Du}\dx\nonumber \\
&&+\omega\sigma_{\varepsilon}\int_{B}\eta^{2}\ell_{1}(\mathcal{D}_{h}^{2})^{2\gamma-1}\snr{\tau_{h}Du}^{2}\dx+\frac{c\sigma_{\varepsilon}}{\omega}\int_{B}\snr{D\eta}^{2}\ell_{1}(\mathcal{D}_{h}^{2})^{2\gamma-1}\snr{\tau_{h}u}^{2}\dx\nonumber \\
&\le&\omega\int_{B}\eta^{2}\left(\ell_{1}(\mathcal{D}_{h})^{-\mu}+\sigma_{\varepsilon}\ell_{1}(\mathcal{D}_{h}^{2})^{2\gamma-1}\right)\snr{\tau_{h}Du}^{2}\dx\nonumber \\
&&+\frac{c\snr{h}^{2}\mathcal{N}_{\infty}}{\omega}\int_{B_{\ti{\tau}_{1}}}\left(\int_{0}^{1}\ell_{1}(\snr{Du(x+th)})^{\gamma-1}\dt\right)^{2}\ell_{1}(\mathcal{D}_{h})^{\mu}\dx\nonumber \\
&&+c\snr{h}\mathcal{N}_{\infty}\left(\int_{B_{\ti{\tau}_{1}}}\left(\int_{0}^{1}\ell_{1}(\snr{Du(x+th)})^{\gamma-1}\dt\right)^{\frac{\gamma}{\gamma-1}}\dx\right)^{\frac{\gamma-1}{\gamma}}\left(\int_{B_{\ti{\tau}_{1}}}\snr{\tau_{h}u}^{\gamma}\dx\right)^{\frac{1}{\gamma}}\nonumber \\
&&+\frac{c\sigma_{\varepsilon}\mathcal{N}_{\infty}}{\omega}\left(\int_{B_{\ti{\tau}_{1}}}\ell_{1}(\mathcal{D}_{h}^{2})^{2\gamma}\dx\right)^{\frac{2\gamma-1}{2\gamma}}\left(\int_{B_{\ti{\tau}_{1}}}\snr{\tau_{h}u}^{4\gamma}\dx\right)^{\frac{1}{2\gamma}}\nonumber \\
&\le&\omega\int_{B}\eta^{2}\left(\ell_{1}(\mathcal{D}_{h})^{-\mu}+\sigma_{\varepsilon}\ell_{1}(\mathcal{D}_{h}^{2})^{2\gamma-1}\right)\snr{\tau_{h}Du}^{2}\dx\nonumber \\
&&+\frac{c\snr{h}^{2}\mathcal{N}_{\infty}}{\omega}\int_{B_{\ti{\tau}_{1}}}\left(\int_{0}^{1}\ell_{1}(\snr{Du(x+th)})^{\gamma-1}\dt\right)^{\frac{2\gamma-2+\mu}{\gamma-1}}+\ell_{1}(\mathcal{D}_{h})^{2\gamma-2+\mu}\dx\nonumber \\
&&+c\snr{h}^{2}\mathcal{N}_{\infty}\int_{B_{\tau_{1}}}\ell_{1}(\snr{Du})^{\gamma}\dx+\frac{c\sigma_{\varepsilon}\snr{h}^{2}\mathcal{N}_{\infty}}{\omega}\int_{B_{\tau_{1}}}\ell_{1}(\snr{Du}^{2})^{2\gamma}\dx\nonumber \\
&\le&\omega\int_{B}\eta^{2}\left(\snr{\tau_{h}V_{1,2-\mu}(Du)}^{2}+\sigma_{\varepsilon}\snr{\tau_{h}V_{1,4\gamma}(Du)}^{2}\right)\dx\nonumber \\
&&+\frac{c\snr{h}^{2}\mathcal{N}_{\infty}}{\omega}\int_{B_{\tau_{1}}}\ell_{1}(\snr{Du})^{2\gamma-2+\mu}+\sigma_{\varepsilon}\ell_{1}(\snr{Du}^{2})^{2\gamma}\dx,
\end{eqnarray*}
for $c\equiv c(n,N,A,\mu,\gamma)$. Finally, using again \eqref{a.5.3s}, Cauchy-Schwarz and Young's inequalities and \eqref{Vm} we control
\begin{eqnarray*}
\snr{\mbox{(III)}}&\le&c\snr{h}^{\alpha}\int_{B}\eta^{2}\ell_{1}(\snr{Du(x+h)})^{\gamma-1}\snr{\tau_{h}Du}\dx\nonumber \\
&\le&\omega\int_{B}\eta^{2}\ell_{1}(\mathcal{D}_{h})^{-\mu}\snr{\tau_{h}Du}^{2}\dx+\frac{c\snr{h}^{2\alpha}}{\omega}\int_{B}\eta^{2}\ell_{1}(\mathcal{D}_{h})^{2\gamma-2+\mu}\dx\nonumber \\
&\le&c\omega\int_{B}\eta^{2}\snr{\tau_{h}V_{1,2-\mu}(Du)}^{2}\dx+\frac{c\snr{h}^{2\alpha}}{\omega}\int_{B_{\tau_{1}}}\ell_{1}(\snr{Du})^{2\gamma-2+\mu}\dx,
\end{eqnarray*}
with $c\equiv c(n,N,A,\mu,\gamma)$. Merging all previous estimates and choosing $\omega\in (0,1)$ sufficiently small, after standard manipulations we obtain
\begin{flalign*}
\int_{B}\eta^{2}\snr{\tau_{h}V_{1,2-\mu}(Du)}^{2}\dx\le \frac{c\snr{h}^{2\alpha}}{(\tau_{1}-\tau_{2})^{2}}\int_{B_{\tau_{1}}}\ell_{1}(\snr{Du})^{2\gamma-2+\mu}+\sigma_{\varepsilon}\ell_{1}(\snr{Du}^{2})^{2\gamma}\dx.
\end{flalign*}
Lemma \ref{fraim} and \eqref{Vm} then yield
\begin{eqnarray}\label{12.0}
\nr{\ell_{1}(\snr{Du})}_{L^{\frac{n(2-\mu)}{n-2\beta}}(B_{\tau_{2}})}^{\frac{2-\mu}{2}}&\le& \frac{c}{(\tau_{1}-\tau_{2})^{1-\alpha+\beta}r^{\beta}}\nr{\ell_{1}(\snr{Du})}_{L^{2\gamma-2+\mu}(B_{\tau_{1}})}^{\frac{2\gamma-2+\mu}{2}}\nonumber \\
&&+\frac{c\sqrt{\sigma_{\varepsilon}}}{(\tau_{1}-\tau_{2})^{1-\alpha+\beta}r^{\beta}}\nr{\ell_{1}(\snr{Du})}_{L^{4\gamma}(B_{\tau_{1}})}^{2\gamma}\nonumber \\
&&+\frac{c}{(\tau_{1}-\tau_{2})^{\beta}}\left(\nr{V_{1,2-\mu}(Du)}_{L^{2}(B_{\tau_{2}})}+\snr{B_{\tau_{2}}}^{\frac{n-2\beta}{2n}}\right),
\end{eqnarray}
for all $\beta\in (\alpha/2,\alpha)$. Now notice that \eqref{gm} implies that $1\le \mu<1+\alpha/n$ and
\eqn{12.1}
$$
\beta\in \left(\max\left\{n(\gamma+\mu-2),\frac{\alpha}{2}\right\},\alpha\right) \ \Longrightarrow \ \frac{n(2-\mu)}{n-2\beta}>2\gamma-2+\mu \quad \mbox{and}\quad \frac{n(2\gamma+\mu-3)}{2\beta-n(\mu-1)}<1,
$$
therefore in \eqref{12.0} we can apply the interpolation inequality
\eqn{inter}
$$
\nr{\ell_{1}(\snr{Du})}_{L^{2\gamma-2+\mu}(B_{\tau_{1}})}\le \nr{\ell_{1}(\snr{Du})}_{L^{\frac{n(2-\mu)}{n-2\beta}}(B_{\tau_{1}})}^{\theta}\nr{\ell_{1}(\snr{Du})}_{L^{1}(B_{\tau_{1}})}^{1-\theta},
$$
where $\theta\in (0,1)$ is derived via
$$
\frac{1}{2\gamma-2+\mu}=\frac{\theta(n-2\beta)}{n(2-\mu)}+1-\theta \ \Longrightarrow \ \theta=\frac{n(2-\mu)(2\gamma+\mu-3)}{(2\beta-n(\mu-1))(2\gamma-2+\mu)}.
$$
Moreover, via \eqref{12.1} we see that $\theta(2\gamma-2+\mu)<2-\mu$, so we can apply Young's inequality with conjugate exponents
$$
(\tx{s}_{1},\tx{s}_{2}):=\left(\frac{2-\mu}{\theta(2\gamma-2+\mu)},\frac{2-\mu}{(2-\mu)(1+\theta)-2\theta\gamma}\right),
$$
to get
\begin{eqnarray*}
\nr{\ell_{1}(\snr{Du})}_{L^{\frac{n(2-\mu)}{n-2\beta}}(B_{\tau_{2}})}^{\frac{2-\mu}{2}}&\stackrel{\eqref{inter}}{\le}&\frac{c}{(\tau_{1}-\tau_{2})^{1-\alpha+\beta}r^{\beta}}\nr{\ell_{1}(\snr{Du})}_{L^{\frac{n(2-\mu)}{n-2\beta}}(B_{\tau_{1}})}^{\frac{\theta(2\gamma-2+\mu)}{2}}\nr{\ell_{1}(\snr{Du})}_{L^{1}(B_{\tau_{1}})}^{\frac{(1-\theta)(2\gamma-2+\mu)}{2}}\nonumber \\
&&+\frac{c}{(\tau_{1}-\tau_{2})^{\beta}}\left(\nr{V_{1,2-\mu}(Du)}_{L^{2}(B_{\tau_{2}})}+\snr{B_{\tau_{2}}}^{\frac{n-2\beta}{2n}}\right)\nonumber \\
&&+\frac{c\sqrt{\sigma_{\varepsilon}}}{(\tau_{1}-\tau_{2})^{1-\alpha+\beta}r^{\beta}}\nr{\ell_{1}(\snr{Du})}_{L^{4\gamma}(B_{\tau_{1}})}^{2\gamma}\nonumber \\
&\le&\frac{1}{4}\nr{\ell_{1}(\snr{Du})}_{L^{\frac{n(2-\mu)}{n-2\beta}}(B_{\tau_{1}})}^{\frac{2-\mu}{2}}+\frac{c}{(\tau_{1}-\tau_{2})^{(1-\alpha+\beta)\tx{s}_{2}}r^{\beta\tx{s}_{2}}}\nr{\ell_{1}(\snr{Du})}_{L^{1}(B_{\tau_{1}})}^{\frac{(1-\theta)(2\gamma-2+\mu)\tx{s}_{2}}{2}}\nonumber \\
&&+\frac{c}{(\tau_{1}-\tau_{2})^{\beta}}\left(\snr{B_{\tau_{2}}}^{\frac{\mu-1}{2}}\nr{Du}_{L^{1}(B_{\tau_{2}})}^{\frac{2-\mu}{2}}+\snr{B_{\tau_{2}}}^{\frac{n-2\beta}{2n}}\right)\nonumber \\
&&+\frac{c\sqrt{\sigma_{\varepsilon}}}{(\tau_{1}-\tau_{2})^{1-\alpha+\beta}r^{\beta}}\nr{\ell_{1}(\snr{Du})}_{L^{4\gamma}(B_{\tau_{1}})}^{2\gamma},
\end{eqnarray*}
with $c\equiv c(n,N,A,\mu,\alpha,\gamma)$. Lemma \ref{iterlem} eventually gives 
\begin{eqnarray}\label{12.6}
\nr{\ell_{1}(\snr{Du})}_{L^{\frac{n(2-\mu)}{n-2\beta}}(B_{r/2})}^{\frac{2-\mu}{2}}&\le&\frac{c}{r^{(1+\beta)\tx{s}_{2}}}\nr{\ell_{1}(\snr{Du})}_{L^{1}(B_{r})}^{\frac{(1-\theta)(2\gamma-2+\mu)\tx{s}_{2}}{2}}+cr^{\frac{n(\mu-1)}{2}-2\beta}\nr{\ell_{1}(\snr{Du})}_{L^{1}(B_{r})}^{\frac{2-\mu}{2}}\nonumber \\
&&+\frac{c\sqrt{\sigma_{\varepsilon}}}{r^{(1+\beta)\tx{s}_{2}}}\nr{\ell_{1}(\snr{Du})}_{L^{4\gamma}(B_{r})}^{2\gamma}.
\end{eqnarray}
The proof is then completed via standard manipulations.
\end{proof}
\begin{remark}
Theorem \ref{thi} is the nonautonomous, higher integrability counterpart of \cite[Theorem 2.2]{fm00}, see also \cite[Theorem 1.1]{fps24} for the nonautonomous, $\mu$-elliptic case, and \cite[Theorem 1.1]{bgs25} for the autonomous, $(\mu,q)$-elliptic setting. A closer inspection of the proof shows that, subject to \eqref{gm}, bounds \eqref{12.2.1}--\eqref{12.2} hold for minima of $(\mu,\gamma)$-elliptic integrands $\ti{\tx{f}}\in C^{2}(\mathbb{R}^{N\times n})$ satisfying\footnote{Only the structural properties of the (regularized) original integrand are listed; the perturbation term $z\mapsto \sigma_{\varepsilon}\ell_{1}(\snr{z}^{2})^{2\gamma}$ is omitted, cf. Remark \ref{remnot}.}
$$
    \begin{cases}
    \ \snr{z}-1\lesssim \ti{\tx{f}}(x,z)\lesssim 1+\snr{z}^{\gamma}\vspace{1.5mm}\\
    \displaystyle
    \ \snr{\partial \ti{\tx{f}}(x_{1},z)-\partial \ti{\tx{f}}(x_{2},z)}\lesssim \snr{x_{1}-x_{2}}^{\alpha}\ell_{1}(\snr{z})^{\gamma-1}\vspace{1.5mm}\\
    \displaystyle
    \ \langle\partial^{2}\ti{\tx{f}}(x,z)\xi,\xi\rangle\gtrsim \ell_{1}(\snr{z})^{-\mu}\snr{\xi}^{2},
    \end{cases}
$$
for all $z,\xi\in \mathbb{R}^{N\times n}$, $x,x_{1},x_{2}\in B$, and some $1\le \mu<2$. Consequently, \eqref{12.2} can be used as an a-priori estimate for solutions to more general, nonautonomous $(\mu,\gamma)$-elliptic problems, included those satisfying linear growth below.
\end{remark}
\subsection{Blow up}\label{bw} Let $B_{\rr}(x_{0})\Subset B$ be a ball with radius $\rr\in (0,1]$. Keeping in mind Remark \ref{remnot}, we scale on $B_{1}(\equiv B_{1}(0))$ the function $u$ by letting $u_{\rr}(x):=u(x_{0}+\rr x)\rr^{-1}$, and integrands $\tx{f},\ti{\tx{f}}$ and all related control functions by setting\footnote{Here $\mathcal{w}$ represents the regularized integrand or the auxiliary functions defined at the beginning of Section \ref{dgc}.} $\mathcal{w}_{\rr}(x,z):=\mathcal{w}(x_{0}+\rr x,z)$ if both indexes $\varepsilon$, $\delta$ appear in their original expression, and $\ti{\mathcal{w}}_{\rr}(x,t):=\ti{\mathcal{w}}(x_{0}+\rr x,z)$ if they depend only on $\delta$. Such maps are now defined for all $x\in\left\{x\in \mathbb{R}^{n}\colon x_{0}+\rr x\in B\right\}=:\mathcal{B}$, and any $z\in \mathbb{R}^{N\times n}$. Since $B_{\rr}(x_{0})\Subset B$, after scaling, $B_{1}\Subset \mathcal{B}$, and thanks also to \eqref{areg}, $u_{\rr}\in W^{1,\infty}(B_{1},\mathbb{R}^{N})$ minimizes the integral
\eqn{funrr}
$$
W^{1,4\gamma}(B_{1},\mathbb{R}^{N})\ni w\mapsto \mathcal{F}_{\rr}(w;B_{1}):=\int_{B_{1}}\tx{f}_{\rr}(x,Dw)\dx,
$$
and solves the related Euler--Lagrange system
\eqn{elrr}
$$
\int_{B_{1}}\langle\partial\tx{f}_{\rr}(x,Du_{\rr}),Dw\rangle\dx=0\qquad \mbox{for all} \ \ w\in W^{1,4\gamma}_{0}(B_{1},\mathbb{R}^{N}).
$$
Notice that the integrand $\tx{f}_{\rr}$ is of the same type as those described in Section \ref{dgc}, with obvious, minimal variations impacting only the oscillation properties \eqref{a12.x}, in which now the bounding constant $c$ is replaced by $c(A,\vartheta_{*})\rr^{\alpha}$ and $B_{r}$ is any ball contained in $\mathcal{B}$.

\section{Frozen Lipschitz bounds}\label{fi} \noindent In the following, an important role will be played by the constant-coefficient counterpart of functional $\mathcal{F}_{\rr}$. Specifically, let $B_{\sigma}(\equiv B_{\sigma}(x_{\textnormal{c}}))\Subset B_{1}$ be a ball, and, for $z\in \mathbb{R}^{N\times n}$, define 
\begin{flalign}\label{fi.0}
\begin{array}{c}
\displaystyle
\tx{f}_{\textnormal{c}}(z):=\tx{f}_{\rr}(x_{\textnormal{c}},z),\qquad A_{\textnormal{c}}(\snr{z}):=A_{\rr}(x_{\textnormal{c}},\snr{z}),\qquad \tx{a}_{\textnormal{c}}(\snr{z}):=\tx{a}_{\rr}(x_{\textnormal{c}},\snr{z}),\\ [8pt]\displaystyle
\lambda_{\textnormal{c}}(\snr{z}):=\lambda_{\rr}(x_{\textnormal{c}},\snr{z}),\qquad \Lambda_{\textnormal{c}}(\snr{z}):=\Lambda_{\rr}(x_{\textnormal{c}},\snr{z}).
\end{array}
\end{flalign}
This is nothing but the integrand $\tx{f}_{\rr}$ and related auxiliary functions made autonomous by "freezing" the space-depending component in the center of ball $B_{\sigma}$. Needles to say, all properties collected in Section \ref{dgc} continue to hold for $\tx{f}_{\textnormal{c}}$ (just replace $x$ with $x_{\textnormal{c}}$ there), while the oscillation conditions \eqref{a12.x} will be used to compare $\tx{f}_{\textnormal{c}}$ to $\tx{f}_{\rr}$ later on. In particular, the convention established in Remark \ref{remnot} holds also for the frozen rescaled maps above, see also at the beginning of Section \ref{bw}. Next, introduce the variational integral
$$
W^{1,4\gamma}(B_{\sigma},\mathbb{R}^{N})\ni w\mapsto \mathcal{F}_{\textnormal{c}}(w;B_{\sigma}):=\int_{B_{\sigma}}\tx{f}_{\textnormal{c}}(Dw)\dx,
$$
and, with $v_{0}\in W^{1,\infty}(\bar{B}_{\sigma},\mathbb{R}^{N})$, let us look at the Dirichlet problem
\eqn{pd0}
$$
v_{0}+W^{1,4\gamma}_{0}(B_{\sigma},\mathbb{R}^{N})\ni w\mapsto \min_{v_{0}+W^{1,4\gamma}(B_{\sigma},\mathbb{R}^{N})}\mathcal{F}_{\textnormal{c}}(w;B_{\sigma}).
$$
Existence and uniqueness of the solution $v\in v_{0}+W^{1,4\gamma}(B_{\sigma},\mathbb{R}^{N})$ follows via standard direct methods and strict convexity arguments. We aim at providing some Lipschitz and higher differentiability estimates for problem \eqref{pd0} that will play a crucial role in the rest of the paper. This is the content of the following theorem.
\begin{theorem}\label{t4.1}
Let $v_{0}\in W^{1,\infty}(\bar{B}_{\sigma},\mathbb{R}^{N})$ be a Lipschitz-regular function and $v\in (v_{0}+W^{1,4\gamma}_{0}(B_{\sigma},\mathbb{R}^{N}))$ be the solution of Dirichlet problem \eqref{pd0}. Then, whenever $B_{\tau}\Subset B_{\sigma}$ is a ball and $\tx{M}_{0}\ge 1$ is a constant such that
\eqn{m0}
    $$
\tx{M}_{0}\ge \max\left\{\tx{t},\nr{Dv}_{L^{\infty}(B_{\tau})}\right\},
    $$
the Caccioppoli inequality
\eqn{4.4}
    $$
\tau\nr{D(\tx{a}_{\textnormal{c}}(\snr{Dv})-\kappa)_{+}}_{L^{2}(B_{\tau/2})}\le c\rrr_{*}(\tx{M}_{0})^{\frac{1}{2}}\nr{(\tx{a}_{\textnormal{c}}(\snr{Dv})-\kappa)_{+}}_{L^{2}(B_{3\tau/4})}
    $$
holds for all $\kk\ge 0$, with $c\equiv c(\data_{0})$. Moreover there exist thresholds $\mu_{0}\equiv \mu_{0}(n)\in (1,3/2)$, $\ti{\omega}_{\mu}\equiv \ti{\omega}_{\mu}(n)\in (0,1)$ such that if $1\le \mu<\mu_{0}$ and $0<\omega_{\mu}<\ti{\omega}_{\mu}$, the Lipschitz bound
    \eqn{4.8}
    $$
\nr{Dv}_{L^{\infty}(B_{3\sigma/4})}\le c\tx{a}_{\textnormal{c}}\left(\nr{Dv_{0}}_{L^{\infty}(B_{\sigma})}\right)^{(\mu-1+\omega_{\mu})\delta_{0}}\nr{Dv_{0}}_{L^{\infty}(B_{\sigma})}+c
    $$
is satisfied for some $\delta_{0}\equiv \delta_{0}(n)$ and $c\equiv c(\data_{0})$.
\end{theorem}
\begin{proof}
By minimality, the solution $v\in v_{0}+W^{1,4\gamma}(B_{\sigma},\mathbb{R}^{N})$ to problem \eqref{pd0} solves the Euler-Lagrange system
\eqn{4.1}
$$
\int_{B_{\sigma}}\langle\partial \tx{f}_{\textnormal{c}}(Dv),Dw\rangle\dx=0\qquad \mbox{for all} \ \ w\in W^{1,4\gamma}_{0}(B_{\sigma},\mathbb{R}^{N}),
$$
and satisfies the energy estimate
\eqn{4.6}
$$
\int_{B_{\sigma}}\tx{f}_{\textnormal{c}}(Dv)\dx\le \int_{B_{\sigma}}\tx{f}_{\textnormal{c}}(Dv_{0})\dx.
$$
Now notice that by \eqref{corfl.1r} (take $x=x_{0}+\rr x_{\textnormal{c}}$ there), the integrand $\tx{f}_{\textnormal{c}}$ satisfies the assumptions in \cite[Section 8]{bm20}, therefore, recalling the definition of $\tx{a}_{\textnormal{c}}$, we have
\eqn{4.0}
$$
v\in W^{1,\infty}_{\loc}(B_{\sigma},\mathbb{R}^{N})\cap W^{2,2}_{\loc}(B_{\sigma},\mathbb{R}^{N}),\qquad \partial\tx{f}_{\textnormal{c}}(Dv)\in W^{1,2}_{\loc}(B_{\sigma},\mathbb{R}^{N\times n}),\qquad \tx{a}_{\textnormal{c}}(\snr{Dw})\in W^{1,2}_{\loc}(B_{\sigma}).
$$
Thanks to \eqref{4.0} we can further differentiate \eqref{4.1} thus getting the system
\eqn{4.2}
$$
\sum_{s=1}^{n}\int_{B_{\sigma}}\langle\partial^{2}\tx{f}_{\textnormal{c}}(Dv)D_{s}Dv,Dw\rangle\dx=0\qquad \mbox{for all} \ \ w\in W^{1,2}(B_{\sigma},\mathbb{R}^{N}) \ \mbox{with} \ \supp(w)\Subset B_{\sigma}.
$$
Let $B_{\tau}(\equiv B_{\tau}(\hat{x}))\Subset B_{\sigma}$ be a ball, $\eta\in C^{1}_{c}(B_{\sigma})$ a cut-off function such that $\mathds{1}_{B_{\tau/2}}\le \eta\le \mathds{1}_{B_{3\tau/4}}$ and $\snr{D\eta}\lesssim \tau^{-1}$, $\kappa\ge \kappa_{0}\ge 0$ be a number and, for $s\in \{1,\cdots,n\}$, set $w_{\kappa}:=\eta^{2}(\tx{a}_{\textnormal{c}}(\snr{Dv})-\kappa)_{+}D_{s}v$. By \eqref{4.0} and the features of $\eta$ we see that $w_{\kappa}$ is admissible in \eqref{4.2}, so we obtain
\begin{eqnarray*}
0&=&\sum_{s=1}^{n}\int_{B_{\sigma}}\eta^{2}(\tx{a}_{\textnormal{c}}(\snr{Dv})-\kappa)_{+}\langle\partial^{2}\tx{f}_{\textnormal{c}}(Dv)D_{s}Dv,D_{s}Dv\rangle\dx\nonumber \\
&&+\int_{B_{\sigma}}\eta^{2}\snr{Dv}\langle \tx{f}_{\textnormal{c}}(Dv) D\snr{Dv},D(\tx{a}_{\textnormal{c}}(\snr{Dv})-\kappa)_{+}\rangle\dx\nonumber \\
&&+2\sum_{s=1}^{n}\int_{B_{\sigma}}\eta(\tx{a}_{\textnormal{c}}(\snr{Dv})-\kappa)_{+}\langle\partial^{2}\tx{f}_{\textnormal{c}}(Dv)D_{s}Dv,D\eta\otimes D_{s}v\rangle\dx=:\mbox{(I)}+\mbox{(II)}+\mbox{(III)}.
\end{eqnarray*}
Keeping in mind that that $D\tx{a}_{\textnormal{c}}(\snr{Dv})=\lambda_{\textnormal{c}}(\snr{Dv})\snr{Dv}D\snr{Dv}$, we bound
\begin{eqnarray*}
\mbox{(I)}+\mbox{(II)}&\ge&c\int_{B_{\sigma}}\eta^{2}\lambda_{\textnormal{c}}(\snr{Dv})(\tx{a}_{\textnormal{c}}(\snr{Dv})-\kappa)_{+}\snr{D^{2}v}^{2}\dx\nonumber \\
&&+c\int_{B_{\sigma}}\eta^{2}\left\langle\frac{\partial^{2}\tx{f}_{\textnormal{c}}(Dv)}{\lambda_{\textnormal{c}}(\snr{Dv})}D(\tx{a}_{\textnormal{c}}(\snr{Dv})-\kk)_{+},D(\tx{a}_{\textnormal{c}}(\snr{Dv})-\kk)_{+}\right\rangle\dx\nonumber \\
&\stackrel{\eqref{lala}_{1}}{\ge}&c\int_{B_{\sigma}}\eta^{2}\lambda_{\textnormal{c}}(\snr{Dv})(\tx{a}_{\textnormal{c}}(\snr{Dv})-\kappa)_{+}\snr{D^{2}v}^{2}\dx+c\int_{B_{\sigma}}\eta^{2}\snr{D(\tx{a}_{\textnormal{c}}(\snr{Dv})-\kappa)_{+}}^{2}\dx,
\end{eqnarray*}
for $c\equiv c(n,N,\gamma)$. Moreover, by the Cauchy-Schwarz and Young inequalities we gain
\begin{eqnarray*}
\snr{\mbox{(III)}}&\le&\omega\int_{B_{\sigma}}\frac{\eta^{2}\partial^{2}\tx{f}_{\textnormal{c}}(Dv)}{\lambda_{\textnormal{c}}(\snr{Dv})}\langle D(\tx{a}_{\textnormal{c}}(\snr{Dv})-\kappa)_{+},D(\tx{a}_{\textnormal{c}}(\snr{Dv})-\kappa)_{+}\rangle\dx\nonumber \\
&&+\frac{c}{\omega}\int_{B_{\sigma}}\snr{D\eta}^{2}(\tx{a}_{\textnormal{c}}(\snr{Dv})-\kappa)_{+}^{2}\frac{\snr{\partial^{2} \tx{f}_{\textnormal{c}}(Dv)}}{\lambda_{\textnormal{c}}(\snr{Dv})}\dx\nonumber \\
&\stackrel{\eqref{lala}_{1}}{\le}&\omega\int_{B_{\sigma}}\eta^{2}\left\langle\frac{\partial^{2}\tx{f}_{\textnormal{c}}(Dv)}{\lambda_{\textnormal{c}}(\snr{Dv})}D(\tx{a}_{\textnormal{c}}(\snr{Dv})-\kappa)_{+},D(\tx{a}_{\textnormal{c}}(\snr{Dv})-\kappa)_{+}\right\rangle\dx\nonumber \\
&&+\frac{c}{\omega}\int_{B_{\sigma}}\frac{\Lambda_{\textnormal{c}}(\snr{Dv})}{\lambda_{\textnormal{c}}(\snr{Dv})}\snr{D\eta}^{2}(\tx{a}_{\textnormal{c}}(\snr{Dv})-\kappa)_{+}^{2}\dx,
\end{eqnarray*}
with $c\equiv c(\data_{0})$. Choosing $\omega$ sufficiently small and combining the previous displays we obtain
\begin{flalign}\label{4.3}
\int_{B_{\sigma}}\eta^{2}\snr{D(\tx{a}_{\textnormal{c}}(\snr{Dv})-\kappa)_{+}}^{2}\dx\le c\int_{B_{\sigma}}\frac{\Lambda_{\textnormal{c}}(\snr{Dv})}{\lambda_{\textnormal{c}}(\snr{Dv})}\snr{D\eta}^{2}(\tx{a}_{\textnormal{c}}(\snr{Dv})-\kappa)_{+}^{2}\dx,
\end{flalign}
for $c\equiv c(\data_{0})$. With $\tx{M}_{0}\ge 1$ being any constant satisfying \eqref{m0}, \eqref{a.555} and the properties of $\eta$, we upgrade \eqref{4.3} to \eqref{4.4}. The Sobolev embedding theorem and estimate \eqref{4.4} yield 
\begin{flalign}\label{4.5}
\left(\mint_{B_{\tau/2}}(\tx{a}_{\textnormal{c}}(\snr{Dv})-\kappa)_{+}^{2\chi_{0}}\dx\right)^{\frac{1}{2\chi_{0}}}\le c\rrr_{*}(\tx{M}_{0})^{\frac{1}{2}}\left(\mint_{B_{\tau}}(\tx{a}_{\textnormal{c}}(\snr{Dv})-\kappa)_{+}^{2}\dx\right)^{\frac{1}{2}},
\end{flalign}
where $\chi_{0}\equiv \chi_{0}(n)\in (1,2)$ comes from the Sobolev theorem and $c\equiv c(\data_{0})$. With $\varsigma>0$, we fix parameters $3\varsigma/4\le \tau_{2}<\tau_{1}\le 5\varsigma/6$ and related concentric balls $B_{3\varsigma/4}(\ti{x})\subset B_{\tau_2}(\ti{x})\subset B_{\tau_1}(\ti{x})\subset B_{5\varsigma/6}(\ti{x})\subset B_{\varsigma}(\ti{x})\subseteq B_{\sigma}$. Pick any $\hat{x}\in B_{\tau_{2}}(\ti{x})$ and set $r_{0}:=(\tau_{1}-\tau_{2})/8$, so that $ B_{r_{0}}(\hat{x}) \subset B_{\tau_1}(\ti{x})$, observe that there is no loss of generality in assuming $\nr{Dv}_{L^{\infty}(B_{3\varsigma/4}(\ti{x}))}\ge \tx{t}$ (otherwise there would be nothing to prove), choose $\tx{M}_{0}= \nr{Dv}_{L^{\infty}(B_{\tau_1}(\ti{x}))}$ and take any concentric ball $B_{\tau}(\hat{x})\subseteq B_{r_0}(\hat{x})$ in \eqref{4.5}. Lemma \ref{revlem} applies with $M_{0}=\rrr_{*}(\tx{M}_{0})^{\frac{1}{2}}$, $f=0$, $w=(\tx{a}_{\textnormal{c}}(\snr{Dv})-\kappa)_{+}$, $\kk_{0}=0$, and gives
\begin{flalign}\label{4.55}
\tx{a}_{\textnormal{c}}(\snr{Dv(\hat{x})})\le c\rrr_{*}(\tx{M}_{0})^{\frac{\chi_{0}}{2(\chi_{0}-1)}}\left(\mint_{B_{r_{0}}(\hat{x})}\tx{a}_{\textnormal{c}}(\snr{Dv})^{2}\dx\right)^{\frac{1}{2}},
\end{flalign}
with $c\equiv c(\data_{0})$. Since $\hat{x}\in B_{\tau_{2}}(\ti{x})$ is arbitrary, we keep bounding
\begin{eqnarray}\label{4.555}
\tx{a}_{\textnormal{c}}\left(\nr{Dv}_{L^{\infty}(B_{\tau_{2}}(\ti{x}))}\right)&\le& \frac{c}{(\tau_{1}-\tau_{2})^{n/2}}\rrr_{*}\left(\tx{M}_{0}\right)^{\frac{\chi_{0}}{2(\chi_{0}-1)}}\tx{a}_{\textnormal{c}}\left(\tx{M}_{0}\right)^{\frac{1}{2}}\nr{\tx{a}_{\textnormal{c}}(\snr{Dv})}_{L^{1}(B_{\tau_{1}}(\ti{x}))}^{\frac{1}{2}}\nonumber \\
&\stackrel{\eqref{a.7.1.x}_{1}}{\le}&\frac{c}{(\tau_{1}-\tau_{2})^{n/2}}\rrr_{*}(\tx{M}_{0})^{\frac{\chi_{0}}{2(\chi_{0}-1)}}\tx{a}_{\textnormal{c}}\left(\tx{M}_{0}\right)^{\frac{1}{2}}\nr{\tx{f}_{\textnormal{c}}(Dv)}_{L^{1}(B_{\tau_{1}}(\ti{x}))}^{\frac{1}{2}}\nonumber \\
&\stackrel{\eqref{rrr}}{\le}&\frac{c}{(\tau_{1}-\tau_{2})^{n/2}}\tx{M}_{0}^{\frac{(\mu-1+\omega_{\mu})\chi_{0}}{\chi_{0}-1}}\tx{a}_{\textnormal{c}}\left(\tx{M}_{0}\right)^{\frac{1}{2}}\nr{\tx{f}_{\textnormal{c}}(Dv)}_{L^{1}(B_{\tau_{1}}(\ti{x}))}^{\frac{1}{2}}\nonumber \\
&\stackrel{\eqref{a.7.1.x}_{4}}{\le}&\frac{c}{(\tau_{1}-\tau_{2})^{n/2}}\tx{a}_{\textnormal{c}}\left(\tx{M}_{0}\right)^{\frac{1}{2}+\frac{(\mu-1+\omega_{\mu})\chi_{0}}{(2-\mu)(\chi_{0}-1)}}\nr{\tx{f}_{\textnormal{c}}(Dv)}_{L^{1}(B_{\tau_{1}}(\ti{x}))}^{\frac{1}{2}},
\end{eqnarray}
for $c\equiv c(\data_{0})$. A proper choice of threshold $\mu_{0}\equiv \mu_{0}(n)\in (1,3/2)$ such as
\eqn{4.10}
$$
1<\mu_{0}<1+\min\left\{\frac{1}{2},\frac{\chi_{0}-1}{16\chi_{0}}\right\} \ \Longrightarrow \ \frac{(\mu-1+\omega_{\mu})\chi_{0}}{(2-\mu)(\chi_{0}-1)}\le \frac{2\chi_{0}(\mu-1+\omega_{\mu})}{\chi_{0}-1}<\frac{1}{2},
$$
and $\omega_{\mu}$ so small that 
\eqn{4.10.1}
$$
0<\omega_{\mu}<\frac{\chi_{0}-1}{16\chi_{0}}=:\ti{\omega}_{\mu}\equiv \ti{\omega}_{\mu}(n),
$$
implies that $\tx{a}_{0}(\tx{M}_{0})$ is raised to a positive exponent less than one, and we can apply Young's inequality with conjugate exponents
\eqn{4.5555}
$$
(2\tx{d}_{1},2\tx{d}_{2}):=\left(\frac{2(2-\mu)(\chi_{0}-1)}{2(\mu-1+\omega_{\mu})\chi_{0}+(\chi_{0}-1)(2-\mu)},\frac{2(2-\mu)(\chi_{0}-1)}{(2-\mu)(\chi_{0}-1)-2(\mu-1+\omega_{\mu})\chi_{0}}\right)
$$
to gain
\begin{eqnarray*}
\tx{a}_{\textnormal{c}}\left(\nr{Dv}_{L^{\infty}(B_{\tau_{2}}(\ti{x}))}\right)&\le& \frac{1}{4}\tx{a}_{\textnormal{c}}\left(\nr{Dv}_{L^{\infty}(B_{\tau_{1}}(\ti{x}))}\right)+c\nonumber \\
&&+\frac{c}{(\tau_{1}-\tau_{2})^{n\tx{d}_{2}}}\left(\int_{B_{5\varsigma/4}(\ti{x})}\tx{f}_{\textnormal{c}}(Dv)\dx\right)^{\tx{d}_{2}}.
\end{eqnarray*}
This, together with Lemma \ref{iterlem} leads to
\begin{flalign}\label{4.7.1}
\tx{a}_{\textnormal{c}}\left(\nr{Dv}_{L^{\infty}(B_{3\varsigma/4}(\ti{x}))}\right)\le c\left(\mint_{B_{5\varsigma/4}(\ti{x})}\tx{f}_{\textnormal{c}}(Dv)\dx\right)^{\tx{d}_{2}}+c,
\end{flalign}
with $c\equiv c(\data_{0})$. Choosing in particular $B_{\varsigma}(\ti{x})=B_{\sigma}$, we keep estimating,
\begin{eqnarray}\label{4.7}
\tx{a}_{\textnormal{c}}\left(\nr{Dv}_{L^{\infty}(B_{3\sigma/4})}\right)&\stackrel{\eqref{4.7.1}}{\le}&c\left(\mint_{B_{\sigma}}\tx{f}_{\textnormal{c}}(Dv)\dx\right)^{\tx{d}_{2}}+c\stackrel{\eqref{4.6}}{\le}c\left(\mint_{B_{\sigma}}\tx{f}_{\textnormal{c}}(Dv_{0})\dx\right)^{\tx{d}_{2}}+c\nonumber \\
&\stackrel{\eqref{a.7.1.x}_{1}}{\le}&c\left(\mint_{B_{\sigma}}\rrr_{*}(\snr{Dv_{0}})\tx{a}_{\textnormal{c}}(\snr{Dv_{0}})\dx\right)^{\tx{d}_{2}}+c\nonumber \\
&\stackrel{\eqref{rrr}}{\le}&c\left(\mint_{B_{\sigma}}\ell_{1}(\snr{Dv_{0}})^{2(\mu-1+\omega_{\mu})}\tx{a}_{\textnormal{c}}(\snr{Dv_{0}})\dx\right)^{\tx{d}_{2}}+c\nonumber \\
&\le&c\tx{a}_{\textnormal{c}}(\nr{Dv_{0}}_{L^{\infty}(B_{\sigma})})^{\frac{2(\mu-1+\omega_{\mu})(2\chi_{0}-1)}{(2-\mu)(\chi_{0}-1)-2(\mu-1+\omega_{\mu})\chi_{0}}}\tx{a}_{\textnormal{c}}(\nr{Dv_{0}}_{L^{\infty}(B_{\sigma})})+c\nonumber \\
&\stackrel{\eqref{4.10}}{\le}&c\tx{a}_{\textnormal{c}}(\nr{Dv_{0}}_{L^{\infty}(B_{\sigma})})^{\frac{24(\mu-1+\omega_{\mu})}{\chi_{0}-1}}\tx{a}_{\textnormal{c}}(\nr{Dv_{0}}_{L^{\infty}(B_{\sigma})})+c
\end{eqnarray}
for $c\equiv c(\data_{0})$. By construction, $\tx{a}_{\textnormal{c}}$ is invertible, so we can apply $\tx{a}_{\textnormal{c}}^{-1}$ to both sides of \eqref{4.7}, set $\delta_{0}:=64/(\chi_{0}-1)$ and recall that in the large $\tx{a}_{\textnormal{c}}^{-1}$ grows as $t^{1/(2-\mu)}$ to get \eqref{4.8}, and the proof is complete. 
\end{proof}

\section{Intrinsic Lipschitz bounds}\label{lip} 
\noindent Let $u_{\rr}\in W^{1,\infty}(B_{1},\mathbb{R}^{N})$ be a local minimizer of functional $\mathcal{F}_{\rr}$ in \eqref{funrr}. Recall that $u_{\rr}$ solves \eqref{elrr} and its Lipschitz continuity is granted by \eqref{areg}, so, in particular, numbers
\eqn{mm}
$$
\tx{M}\ge \max\left\{\tx{t},\nr{Du}_{L^{\infty}(B_{\rr}(x_{0}))}\right\}\qquad \mbox{and}\qquad \mathcal{M}:=\max\left\{\tx{t},\nr{\tx{a}(\cdot,\snr{Du})}_{L^{\infty}(B_{\rr}(x_{0}))}\right\}
$$
are both finite. Let us introduce a series of parameters that will play a crucial role in the reminder of this section.
\subsubsection*{Step 1: numerology} Name
\eqn{ab0}
$$
\beta_{0}:=\frac{2}{2+\alpha},\qquad \quad \alpha_{0}:=\frac{\alpha}{2+\alpha},
$$
and, with $i\in \{1,2,3\}$, define numbers $\omega_{i}, \ti{\omega}\in (0,1)$ as
\eqn{omega}
$$
\begin{array}{c}
\displaystyle
\ \omega_{i}:=1-\frac{i}{2}\left(\frac{\alpha}{n}+\vartheta_{*}\right),\qquad\quad  \ti{\omega}:=\frac{3+\alpha-\gamma(1+2\alpha)}{2}.
\end{array}
$$
Next, for $i\in \{1,2,3\}$, shorten 
\eqn{bbb}
$$
\begin{array}{c}
\displaystyle
\mathcal{b}_{0}:=36\delta_{0}(\mu-1+\omega_{\mu}),\qquad \quad \mathcal{b}_{i}:=80\gamma\delta_{0}(\mu-1+\omega_{\mu})+\frac{i\vartheta_{*}}{2}+\frac{1+\omega_{i}}{2}\\[10pt]\displaystyle
\ti{\mathcal{b}}:=80\gamma\delta_{0}(\mu-1+\omega_{\mu})+\vartheta_{*}+\frac{1+\ti{\omega}}{2}.
\end{array}
$$
and finally reduce the size of $\mu$ by requiring that $1\le \mu<\mu_{\textnormal{max}}:=\min\left\{\mu_{0},\mu_{*}\right\},$
where $\mu_{0}\equiv \mu_{0}(n)\in (1,3/2)$ is the limiting threshold from Theorem \ref{t4.1}, cf. \eqref{4.10}, and $\mu_{*}\equiv \mu_{*}(\alpha,n,\vartheta_{*})$ is such that
\eqn{mu*}
$$
\mu_{*}:=1+\min\left\{\frac{\chi-1}{144\delta_{0}\chi},\frac{\chi-1}{1280\gamma\delta_{0}\chi}\left(\frac{\alpha}{n}-\vartheta_{*}\right)\right\}.
$$
Here we will also use that $1\le \mu<3/2$, implied by \eqref{mu*}, to control $2-\mu>1/2$. Similarly, we pick $0<\omega_{\mu}<\bar{\omega}_{\mu}$, where $\omega_{\mu}$ is so small that 
\eqn{tio}
$$
0<\bar{\omega}_{\mu}<\min\left\{\mu_{*}-1,\ti{\omega}_{\mu}\right\},
$$
with $\ti{\omega}_{\mu}\equiv \ti{\omega}_{\mu}(n)$ coming from Theorem \ref{t4.1}, see \eqref{4.10.1}. In the supercritical case $n=2$ and $\alpha\ge 2/3$, we restrict further the size of $\mu_{\textnormal{max}}$, $\bar{\omega}_{\mu}$ by requiring that 
\eqn{muma}
$$
\begin{array}{c}
\displaystyle
\mu_{\textnormal{max}}:=\min\left\{\mu_{0},\mu_{*},\frac{(3-\alpha+\gamma(2\alpha-3))(\chi-1)}{1280\delta_{0}\gamma\chi},2+\frac{\alpha}{2}-\gamma\right\}>1,\\ [12pt]\displaystyle 0<\bar{\omega}_{\mu}<\min\{\mu_{\textnormal{max}}-1,\ti{\omega}_{\mu}\},
\end{array}
$$ which makes sense thanks to \eqref{a.5.2s}.
\subsubsection*{Step 2: auxiliary estimates} Pick a vector $h\in \mathbb{R}^{n}$ with $\snr{h}\in \left(0,2^{-8/\beta_{0}}\right)$, and for $x_{\textnormal{c}}\in B_{\frac{1}{2}+2\snr{h}^{\beta_{0}}}$ set $B_{h}:=B_{\snr{h}^{\beta_{0}}}(x_{\textnormal{c}})$, which, by construction, satisfies the inclusion $8B_{h}\Subset B_{3/4}\subset B_{1}$. Next, let $\tx{f}_{\textnormal{c}}$ be the integrand in Section \ref{fi}, frozen at point $x_{\textnormal{c}}$, the center of ball $B_{h}$, and let $v\in u_{\rr}+W^{1,4\gamma}_{0}(8B_{h},\mathbb{R}^{N})$ be the solution of Dirichlet problem
\eqn{5.0}
$$
u_{\rr}+W^{1,4\gamma}_{0}(8B_{h},\mathbb{R}^{N})\ni w\mapsto \min_{w\in u_{\rr}+W^{1,4\gamma}_{0}(8B_{h},\mathbb{R}^{N})}\int_{8B_{h}}\tx{f}_{\textnormal{c}}(Dw)\dx.
$$
Existence and uniqueness of $v$ follow by standard direct methods. By minimality, $v$ solves the Euler-Lagrange system
\eqn{5.3}
$$
\int_{8B_{h}}\langle\partial\tx{f}_{\textnormal{c}}(Dv),Dw\rangle\dx=0\qquad \mbox{for all} \ \ w\in W^{1,4\gamma}_{0}(8B_{h},\mathbb{R}^{N}),
$$
and satisfies the energy estimate
\eqn{5.4}
$$
\int_{8B_{h}}\tx{f}_{\textnormal{c}}(Dv)\dx\le \int_{8B_{h}}\tx{f}_{\textnormal{c}}(Du_{\rr})\dx.
$$
Moreover, keeping in mind also \eqref{areg}, Theorem \ref{t4.1} applies, so $v\in W^{1,\infty}_{\loc}(8B_{h},\mathbb{R}^{N})\cap W^{2,2}_{\loc}(8B_{h},\mathbb{R}^{N})$ with $\partial\tx{f}_{\textnormal{c}}(Dv)\in W^{1,2}_{\loc}(8B_{h},\mathbb{R}^{N\times n})$ and $\tx{a}_{\textnormal{c}}(\snr{Dv})\in W^{1,2}_{\loc}(8B_{h})$, cf. \eqref{4.0}. Before entering into the key bounds of this section, let us record a few preliminary estimates. Since $\tx{a}_{\textnormal{c}}$ is increasing and continuous, we bound
\begin{eqnarray}\label{10.0}
\tx{a}_{\textnormal{c}}(\tx{M})&\le&\nr{\tx{a}_{\textnormal{c}}(\snr{Du_{\rr}})}_{L^{\infty}(B_{1})}\nonumber \\
&\stackrel{\eqref{a12.x}_{1}}{\le}&\mathcal{M}+c\rr^{\alpha}\tx{M}\left(1+\left\|\inf_{x\in B_{1}}\ti{A}_{\rr}(x,\snr{Du_{\rr}})\right\|_{L^{\infty}(B_{1})}^{\vartheta_{*}}\right)\nonumber \\
&\stackrel{\eqref{a.7.1.x}_{1}}{\le}&\mathcal{M}+c\rr^{\alpha}\rrr_{*}(\tx{M})^{\vartheta_{*}}\tx{M}\mathcal{M}^{\vartheta_{*}}\nonumber \\
&\stackrel{\eqref{rrr}}{\le}&\mathcal{M}+c\rr^{\alpha}\tx{M}^{2\vartheta_{*}(\mu-1+\omega_{\mu})+1}\mathcal{M}^{\vartheta_{*}}\stackrel{\eqref{a.7.1.x}_{4}}{\le}\mathcal{M}+c\rr^{\alpha}\mathcal{M}^{6(\mu-1+\omega_{\mu})+\vartheta_{*}+1},
\end{eqnarray}
with $c\equiv c(A,\tx{g},\mu,\gamma,\alpha)$. Recalling \eqref{mm}, by \eqref{10.0}, estimates \eqref{4.4}--\eqref{4.8} can be rewritten as
\begin{flalign}\label{5.1}
\nr{Dv}_{L^{\infty}(6B_{h})}\le c\tx{a}_{\textnormal{c}}\left(\tx{M}\right)^{(\mu-1+\omega_{\mu})\delta_{0}}\tx{M}\le c\mathcal{M}^{\delta_{0}(6(\mu-1+\omega_{\mu})+\vartheta_{*}+1)(\mu-1+\omega_{\mu})}\tx{M}\le \mathcal{M}^{16\delta_{0}(\mu-1+\omega_{\mu})}\tx{M},
\end{flalign}
for $c\equiv c(\data_{0},\alpha)$, and, given any ball $B_{\tau}\subseteq 6B_{h}$ and number $\kk\ge 0$,
\begin{eqnarray}\label{5.2}
\tau\nr{D(\tx{a}_{\textnormal{c}}(\snr{Dv})-\kk)_{+}}_{L^{2}(B_{\tau/2})}&\stackrel{\eqref{5.1}}{\le}&c\rrr_{*}(\mathcal{M}^{16\delta_{0}(\mu-1+\omega_{\mu})}\tx{M})^{\frac{1}{2}}\nr{(\tx{a}_{\textnormal{c}}(\snr{Dv})-\kk)_{+}}_{L^{2}(B_{\tau})}\nonumber \\
&\stackrel{\eqref{rrr},\eqref{a.7.1.x}_{4}}{\le}&c\mathcal{M}^{36\delta_{0}(\mu-1+\omega_{\mu})}\nr{(\tx{a}_{\textnormal{c}}(\snr{Dv})-\kk)_{+}}_{L^{2}(B_{\tau})},
\end{eqnarray}
with $c\equiv c(\data_{0},\alpha)$. Moreover, on $6B_{h}$, 
\begin{eqnarray}\label{5.6}
\tx{a}_{\textnormal{c}}(\snr{Dv})&\stackrel{\eqref{5.1}}{\le}&\int_{0}^{c\mathcal{M}^{16\delta_{0}(\mu-1+\omega_{\mu})}\tx{M}}\ti{\lambda}_{\textnormal{c}}(s)s\dx+c\sigma_{\varepsilon}\ell_{1}\left(\mathcal{M}^{32\delta_{0}(\mu-1+\omega_{\mu})}\tx{M}^{2}\right)^{2\gamma}\nonumber \\
&\stackrel{\eqref{a.55}_{1}}{\le}& c\ti{\lambda}_{\textnormal{c}}(c\mathcal{M}^{16\delta_{0}(\mu-1+\omega_{\mu})}\tx{M})\mathcal{M}^{16\mu\delta_{0}(\mu-1+\omega_{\mu})}\tx{M}^{\mu}\int_{0}^{c\mathcal{M}^{16\delta_{0}(\mu-1+\omega_{\mu})}\tx{M}}s^{1-\mu}\dx\nonumber \\
&&+c\mathcal{M}^{64\gamma\delta_{0}(\mu-1+\omega_{\mu})}\sigma_{\varepsilon}\ell_{1}(\tx{M}^{2})^{2\gamma}\nonumber \\
&\stackrel{\eqref{a.55}_{2}}{\le}&c\mathcal{M}^{32\delta_{0}(\mu-1+\omega_{\mu})}\ti{\lambda}_{\textnormal{c}}(\tx{M})\tx{M}^{2}+c\mathcal{M}^{64\gamma\delta_{0}(\mu-1+\omega_{\mu})}\sigma_{\varepsilon}\ell_{1}(\tx{M}^{2})^{2\gamma-1}\tx{M}^{2}\nonumber \\
&\le&c\mathcal{M}^{64\gamma\delta_{0}(\mu-1+\omega_{\mu})}\lambda_{\textnormal{c}}(\tx{M})\tx{M}^{2}
\stackrel{\eqref{a.7.1.x}_{2}}{\le} c\mathcal{M}^{64\gamma\delta_{0}(\mu-1+\omega_{\mu})}\left(\nr{\tx{a}_{\textnormal{c}}(\snr{Du_{\rr}})}_{L^{\infty}(B_{1})}+1\right),
\end{eqnarray}
where $c\equiv c(\data_{0},\alpha)$. Combining \eqref{10.0} and \eqref{5.6} we obtain
\eqn{10.1}
$$
\nr{\tx{a}_{\textnormal{c}}(\snr{Dv})}_{L^{\infty}(6B_{h})}\le c\mathcal{M}^{1+64\gamma\delta_{0}(\mu-1+\omega_{\mu})}+c\rr^{\alpha}\mathcal{M}^{70\gamma\delta_{0}(\mu-1+\omega_{\mu})+\vartheta_{*}+1},
$$
for $c\equiv c(\data_{0},\alpha)$. We shall also need the following estimates. For $x\in 6B_{h}$, we control oscillation
\begin{eqnarray}\label{5.8}
\mathcal{A}_{1}&:=&\snr{\tx{a}_{\textnormal{c}}(\snr{Du_{\rr}})-\tx{a}_{\rr}(x,\snr{Du_{\rr}})}\nonumber \\
&\stackrel{\eqref{a12.x}_{1}}{\le}&c\rr^{\alpha}\snr{h}^{\alpha\beta_{0}}\ell_{\delta}(\snr{Du_{\rr}})+c\rr^{\alpha}\snr{h}^{\alpha\beta_{0}}\left(\inf_{x\in 6B_{h}}\ti{A}_{\rr}(x,\snr{Du_{\rr}})\right)^{\vartheta_{*}}\ell_{\delta}(\snr{Du_{\rr}})\nonumber \\
&\stackrel{\eqref{a.7.1.x}_{1,4}}{\le}&c\rr^{\alpha}\snr{h}^{\alpha\beta_{0}}+c\rr^{\alpha}\snr{h}^{\alpha\beta_{0}}\rrr_{*}(\tx{M})^{\vartheta_{*}}\ti{\tx{a}}_{\rr}(x,\snr{Du_{\rr}})^{\vartheta_{*}+\frac{1}{2-\mu}}\nonumber \\
&\stackrel{\eqref{rrr}}{\le}&c\rr^{\alpha}\snr{h}^{\alpha\beta_{0}}+c\rr^{\alpha}\snr{h}^{\alpha\beta_{0}}\mathcal{M}^{6(\mu-1+\omega_{\mu})}\ti{\tx{a}}_{\rr}(x,\snr{Du_{\rr}})^{\vartheta_{*}+1},
\end{eqnarray}
for $c\equiv c(A,\tx{g},\mu,\gamma,\alpha)$. Furthermore, after letting
$$
\mathcal{D}_{\textnormal{c}}:=\sqrt{\lambda_{\textnormal{c}}(\snr{Du_{\rr}}+\snr{Dv})}\snr{Du_{\rr}-Dv},
$$
by the mean value theorem, $\eqref{a.7.1.x}_{2}$, \eqref{a.55}, \eqref{l60}, \eqref{10.0}, and \eqref{10.1} we also have
\begin{eqnarray}\label{5.9}
 \mathcal{A}_{2}&:=&\snr{\tx{a}_{\textnormal{c}}(\snr{Du_{\rr}})-\tx{a}_{\textnormal{c}}(\snr{Dv})}\nonumber\\
 &\le&c\left(\int_{0}^{1}\ti{\lambda}_{\textnormal{c}}(\snr{Dv+s(Du_{\rr}-Dv)})\snr{Dv+s(Du_{\rr}-Dv)}\ds\right)\snr{Du_{\rr}-Dv}\nonumber \\
 &&+c\sigma_{\varepsilon}\left(\int_{0}^{1}\ell_{1}(\snr{Dv+s(Du_{\rr}-Dv)}^{2})^{2\gamma-1}\snr{Dv+s(Du_{\rr}-Dv)}\ds\right)\snr{Du_{\rr}-Dv}\nonumber \\
 &\le&c\ti{\lambda}_{\textnormal{c}}(\snr{Du_{\rr}}+\snr{Dv})\left(\snr{Du_{\rr}}+\snr{Dv}\right)^{\mu}\left(\int_{0}^{1}\snr{Dv+s(Du_{\rr}-Dv)}^{1-\mu}\ds\right)\snr{Du_{\rr}-Dv}\nonumber \\
 &&+c\sqrt{\sigma_{\varepsilon}}\ell_{1}(\snr{Du_{\rr}}^{2}+\snr{Dv}^{2})^{\gamma}\mathcal{D}_{\textnormal{c}}\nonumber \\
 &\le&c\left(\left(\ti{\lambda}_{\textnormal{c}}(\snr{Du_{\rr}})\snr{Du_{\rr}}^{2}\right)^{\frac{1}{2}}+\left(\ti{\lambda}_{\textnormal{c}}(\snr{Dv})\snr{Dv}^{2}\right)^{\frac{1}{2}}+\sqrt{\sigma_{\varepsilon}}\ell_{1}(\snr{Du_{\rr}}^{2}+\snr{Dv}^{2})^{\gamma}\right)\mathcal{D}_{\textnormal{c}}\nonumber \\
 &\le&c\left(1+\tx{a}_{\textnormal{c}}(\snr{Du_{\rr}})^{\frac{1}{2}}+\tx{a}_{\textnormal{c}}(\snr{Dv})^{\frac{1}{2}}\right)\mathcal{D}_{\textnormal{c}}\nonumber \\
 &\le& c\left(\mathcal{M}^{\frac{1}{2}+32\gamma\delta_{0}(\mu-1+\omega_{\mu})}+\rr^{\frac{\alpha}{2}}\mathcal{M}^{35\gamma\delta_{0}(\mu-1+\omega_{\mu})+\frac{\vartheta_{*}+1}{2}}\right)\mathcal{D}_{\textnormal{c}},
\end{eqnarray}
with $c\equiv c(\data_{0},\alpha)$. Now we are ready to enter the core of the proof.
\subsubsection*{Step 3: comparison} We jump back to \eqref{5.3}, and see that, upon extension as $v=u_{\rr}$ in $B_{1}\setminus 8B_{h}$, by \eqref{areg}, $v-u_{\rr}\in W^{1,4\gamma}_{0}(8B_{h},\mathbb{R}^{N})$ is an admissible test function in both \eqref{elrr} and \eqref{5.3}, so we bound via the mean value theorem,
\begin{eqnarray}\label{com}
\int_{8B_{h}}\mathcal{D}_{\textnormal{c}}^{2}\dx&\stackrel{\eqref{lala}_{2}}{\le}&c\int_{8B_{h}}\langle \partial\tx{f}_{\textnormal{c}}(Dv)-\partial\tx{f}_{\textnormal{c}}(Du_{\rr}),Dv-Du_{\rr}\rangle\dx\nonumber \\
&\stackrel{\eqref{5.3},\eqref{elrr}}{\le}&c\int_{8B_{h}}\langle\partial\tx{f}_{\rr}(x,Du_{\rr})-\partial\tx{f}_{\textnormal{c}}(Du_{\rr}),Dv-Du_{\rr}\rangle\dx\nonumber \\
&\stackrel{\eqref{a12.x}_{2}}{\le}&c\rr^{\alpha}\snr{h}^{\beta_{0}\alpha}\int_{8B_{h}}\left(1+\inf_{x\in 8B_{h}}\ti{A}_{\rr}(x,\snr{Du_{\rr}})\right)^{\vartheta_{*}}\snr{Du_{\rr}-Dv}\dx\nonumber \\
&\stackrel{\eqref{a.7.1.x}_{1,4}}{\le}&c\rr^{\alpha}\snr{h}^{\alpha\beta_{0}}\rrr_{*}(\mathcal{M}^{1/(2-\mu)})^{\vartheta_{*}}\mathcal{M}^{\vartheta_{*}}\int_{8B_{h}}\snr{Du_{\rr}}+\snr{Dv}\dx\nonumber \\
&\stackrel{\eqref{rrr}}{\le}&c\rr^{\alpha}\snr{h}^{\alpha\beta_{0}}\mathcal{M}^{\vartheta_{*}+4(\mu-1+\omega_{\mu})}\int_{8B_{h}}1+\ti{A}_{\rr}(x,\snr{Du_{\rr}})+\ti{A}_{\textnormal{c}}(\snr{Dv})\dx\nonumber \\
&\stackrel{\eqref{5.4}}{\le}&c\rr^{\alpha}\snr{h}^{\alpha\beta_{0}}\mathcal{M}^{\vartheta_{*}+4(\mu-1+\omega_{\mu})}\int_{8B_{h}}1+\ti{A}_{\rr}(x,\snr{Du_{\rr}})+A_{\textnormal{c}}(\snr{Du_{\rr}})\dx=:\mathcal{C},
\end{eqnarray}
where $c\equiv c(\data)$. 
\subsubsection*{Step 4: uniform bounds in Nikol'skii spaces} Let $\omega_{1},\omega_{2},\omega_{3}, \ti{\omega}\in (0,1)$ be the numbers in \eqref{omega}. Basic properties of translations, the $1$-Lipschitz character of truncations, the Poincar\'e inequality, $\eqref{a.7.1.x}_{1,4}$, \eqref{4.0}, \eqref{5.2} and \eqref{5.8}--\eqref{com} yield
\begin{eqnarray}\label{6.21.x}
\int_{B_{h}}\snr{\tau_{h}(\tx{a}_{\rr}(\cdot,\snr{Du_{\rr}})-\kk)_{+}}^{2}\dx&\le&c\int_{B_{h}}\snr{\tau_{h}(\tx{a}_{\textnormal{c}}(\snr{Dv})-\kk)_{+}}^{2}\dx+c\int_{2B_{h}}\mathcal{A}_{1}^{2}+\mathcal{A}_{2}^{2}\dx\nonumber \\
&\le&c\snr{h}^{2}\int_{B_{h}}\snr{D(\tx{a}_{\textnormal{c}}(\snr{Dv})-\kk)_{+}}^{2}\dx\nonumber \\
&&+c\rr^{2\alpha}\snr{h}^{2\alpha\beta_{0}}\mathcal{M}^{12(\mu-1+\omega_{\mu})}\int_{2B_{h}}\tx{a}_{\rr}(x,\snr{Du_{\rr}})^{2+2\vartheta_{*}}+1\dx\nonumber \\
&&+c\left(\mathcal{M}^{1+64\gamma\delta_{0}(\mu-1+\omega_{\mu})}+\rr^{\alpha}\mathcal{M}^{70\gamma\delta_{0}(\mu-1+\omega_{\mu})+\vartheta_{*}+1}\right)\int_{2B_{h}}\mathcal{D}_{\textnormal{c}}^{2}\dx\nonumber \\
&\le&c\mathcal{M}^{72\delta_{0}(\mu-1+\omega_{\mu})}\snr{h}^{2(1-\beta_{0})}\int_{2B_{h}}(\tx{a}_{\rr}(x,\snr{Du_{\rr}})-\kk)_{+}^{2}\dx\nonumber \\
&&+c\mathcal{M}^{72\delta_{0}(\mu-1+\omega_{\mu})}\snr{h}^{2(1-\beta_{0})}\int_{2B_{h}}\mathcal{A}_{1}^{2}+\mathcal{A}_{2}^{2}\dx\nonumber \\
&&+c\rr^{2\alpha}\snr{h}^{2\alpha\beta_{0}}\mathcal{M}^{12(\mu-1+\omega_{\mu})}\int_{2B_{h}}\tx{a}_{\rr}(x,\snr{Du_{\rr}})^{2+2\vartheta_{*}}+1\dx\nonumber \\
&&+c\left(\mathcal{M}^{1+64\gamma\delta_{0}(\mu-1+\omega_{\mu})}+\rr^{\alpha}\mathcal{M}^{70\gamma\delta_{0}(\mu-1+\omega_{\mu})+\vartheta_{*}+1}\right)\int_{2B_{h}}\mathcal{D}_{\textnormal{c}}^{2}\dx\nonumber \\
&\le&c\mathcal{M}^{72\delta_{0}(\mu-1+\omega_{\mu})}\snr{h}^{2(1-\beta_{0})}\int_{2B_{h}}(\tx{a}_{\rr}(x,\snr{Du_{\rr}})-\kk)_{+}^{2}\nonumber \\
&&+c\rr^{2\alpha}\snr{h}^{2\alpha\beta_{0}}\mathcal{M}^{84\delta_{0}(\mu-1+\omega_{\mu})+2\vartheta_{*}+\omega_{2}+1}\int_{2B_{h}}\tx{a}_{\rr}(x,\snr{Du_{\rr}})^{1-\omega_{2}}+1\dx\nonumber \\
&&+c\left(\mathcal{M}^{1+136\gamma\delta_{0}(\mu-1+\omega_{\mu})}+\rr^{\alpha}\mathcal{M}^{142\gamma\delta_{0}(\mu-1+\omega_{\mu})+\vartheta_{*}+1}\right)\mathcal{C},
\end{eqnarray}
for $c\equiv c(\data)$. To complete estimate \eqref{6.21.x}, we only need to distinguish two cases: $n\ge 3$ or $n=2$ and $0<\alpha<2/3$, and $n=2$ with $\alpha\ge 2/3$. In the first case, by $\eqref{d2.3}_{2}$, \eqref{rrr}, \eqref{a.7.1.x}$_{1,4}$, and \eqref{a12.x}$_{1}$ we bound
\begin{flalign*}
&\left(\mathcal{M}^{1+136\gamma\delta_{0}(\mu-1+\omega_{\mu})}+\rr^{\alpha}\mathcal{M}^{142\gamma\delta_{0}(\mu-1+\omega_{\mu})+\vartheta_{*}+1}\right)\mathcal{C}\nonumber \\
&\qquad \qquad \le c\snr{h}^{\alpha\beta_{0}}\left(\rr^{\alpha}\mathcal{M}^{140\gamma\delta_{0}(\mu-1+\omega_{\mu})+\vartheta_{*}+1}+\rr^{2\alpha}\mathcal{M}^{146\gamma\delta_{0}(\mu-1+\omega_{\mu})+2\vartheta_{*}+1}\right)\int_{8B_{h}}A_{\rr}(x,\snr{Du_{\rr}})+1\dx\nonumber \\
&\qquad \qquad \quad +c\snr{h}^{2\alpha\beta_{0}}\left(\rr^{2\alpha}\mathcal{M}^{140\gamma\delta_{0}(\mu-1+\omega_{\mu})+\vartheta_{*}+1}+\rr^{3\alpha}\mathcal{M}^{146\gamma\delta_{0}(\mu-1+\omega_{\mu})+2\vartheta_{*}+1}\right)\int_{8B_{h}}A_{\rr}(x,\snr{Du_{\rr}})^{\vartheta_{*}}\ell_{\delta}(\snr{Du_{\rr}})\dx \nonumber \\
&\qquad \qquad \le c\snr{h}^{\alpha\beta_{0}}\mathcal{M}^{160\gamma\delta_{0}(\mu-1+\omega_{\mu})+\vartheta_{*}+1+\omega_{1}}\rr^{\alpha}\int_{8B_{h}}\ell_{1}(\tx{a}_{\rr}(x,\snr{Du_{\rr}}))^{1-\omega_{1}}\dx\nonumber \\
&\qquad \qquad\quad + c\snr{h}^{\alpha\beta_{0}}\mathcal{M}^{160\gamma\delta_{0}(\mu-1+\omega_{\mu})+2\vartheta_{*}+1+\omega_{2}}\rr^{2\alpha}\int_{8B_{h}}\ell_{1}(\tx{a}_{\rr}(x,\snr{Du_{\rr}}))^{1-\omega_{2}}\dx\nonumber \\
&\qquad \qquad\quad + c\snr{h}^{\alpha\beta_{0}}\mathcal{M}^{160\gamma \delta_{0}(\mu-1+\omega_{\mu})+3\vartheta_{*}+1+\omega_{3}}\rr^{3\alpha}\int_{8B_{h}}\ell_{1}(\tx{a}_{\rr}(x,\snr{Du_{\rr}}))^{1-\omega_{3}}\dx,
\end{flalign*}
for $c\equiv c(\data)$, while in the second one, via $\eqref{d2.3}_{2}$, \eqref{rrr}, \eqref{a.7.1.x}$_{1}$, \eqref{a12.x}$_{1}$, and \eqref{a.5.3s} we have
\begin{flalign*}
&\left(\mathcal{M}^{1+136\gamma\delta_{0}(\mu-1+\omega_{\mu})}+\rr^{\alpha}\mathcal{M}^{142\gamma\delta_{0}(\mu-1+\omega_{\mu})+\vartheta_{*}+1}\right)\mathcal{C}\nonumber \\
&\qquad \qquad \le c\snr{h}^{\alpha\beta_{0}}\left(\rr^{\alpha}\mathcal{M}^{140\gamma\delta_{0}(\mu-1+\omega_{\mu})+1+\vartheta_{*}}+\rr^{2\alpha}\mathcal{M}^{146\gamma\delta_{0}(\mu-1+\omega_{\mu})+2\vartheta_{*}+1}\right)\int_{8B_{h}}1+\ti{A}_{\rr}(x,\snr{Du_{\rr}})\dx\nonumber \\
&\qquad \qquad \quad +c\snr{h}^{\alpha\beta_{0}}\left(\rr^{\alpha}\mathcal{M}^{140\gamma\delta_{0}(\mu-1+\omega_{\mu})+\vartheta_{*}+1}+\rr^{2\alpha}\mathcal{M}^{146\gamma\delta_{0}(\mu-1+\omega_{\mu})+2\vartheta_{*}+1}\right)\int_{8B_{h}}\ti{A}_{\textnormal{c}}(\snr{Du_{\rr}})+\sigma_{\varepsilon}\ell_{1}(\snr{Du_{\rr}}^{2})^{2\gamma}\nonumber \\
&\qquad \qquad \le c\snr{h}^{\alpha\beta_{0}}\left(\rr^{\alpha}\mathcal{M}^{140\gamma\delta_{0}(\mu-1+\omega_{\mu})+\vartheta_{*}+1}+\rr^{2\alpha}\mathcal{M}^{146\gamma\delta_{0}(\mu-1+\omega_{\mu})+2\vartheta_{*}+1}\right)\int_{8B_{h}}1+A_{\rr}(x,\snr{Du_{\rr}})\dx\nonumber \\
&\qquad \qquad \quad +c\snr{h}^{\alpha\beta_{0}}\left(\rr^{\alpha}\mathcal{M}^{140\gamma\delta_{0}(\mu-1+\omega_{\mu})+\vartheta_{*}+1}+\rr^{2\alpha}\mathcal{M}^{146\gamma\delta_{0}(\mu-1+\omega_{\mu})+2\vartheta_{*}+1}\right)\int_{8B_{h}}\ti{A}_{\textnormal{c}}(\snr{Du_{\rr}})\dx\nonumber \\
&\qquad \qquad \le c\snr{h}^{\alpha\beta_{0}}\mathcal{M}^{160\gamma\delta_{0}(\mu-1+\omega_{\mu})+\vartheta_{*}+1+\omega_{1}}\rr^{\alpha}\int_{8B_{h}}\ell_{1}(\tx{a}_{\rr}(x,\snr{Du_{\rr}}))^{1-\omega_{1}}\dx\nonumber \\
&\qquad \qquad\quad +c\snr{h}^{\alpha\beta_{0}}\mathcal{M}^{160\gamma\delta_{0}(\mu-1+\omega_{\mu})+2\vartheta_{*}+1+\omega_{2}}\rr^{2\alpha}\int_{8B_{h}}\ell_{1}(\tx{a}_{\rr}(x,\snr{Du_{\rr}}))^{1-\omega_{2}}\dx\nonumber \\
&\qquad \qquad\quad +c\snr{h}^{\alpha\beta_{0}}\mathcal{M}^{160\gamma\delta_{0}(\mu-1+\omega_{\mu})+2\vartheta_{*}+1+\ti{\omega}}\rr^{2\alpha}\int_{8B_{h}}\ell_{1}(\snr{Du_{\rr}})^{\gamma-\ti{\omega}}\dx,
\end{flalign*}
with $c\equiv c(\data)$. Plugging the content of the two previous displays into \eqref{6.21.x} and recalling \eqref{1111}, \eqref{ab0}--\eqref{bbb}, we obtain
\begin{eqnarray}\label{6.21}
\int_{B_{h}}\snr{\tau_{h}(\tx{a}_{\rr}(\cdot,\snr{Du_{\rr}})-\kk)_{+}}^{2}\dx&\le&c\mathcal{M}^{2\mathcal{b}_{0}}\snr{h}^{2\alpha_{0}}\int_{8B_{h}}(\tx{a}_{\rr}(x,\snr{Du_{\rr}})-\kk)_{+}^{2}\dx\nonumber \\
&&+c\snr{h}^{2\alpha_{0}}\sum_{i=1}^{3}\mathds{1}_{i}\mathcal{M}^{2\mathcal{b}_{i}}\rr^{i\alpha}\int_{8B_{h}}\ell_{1}(\tx{a}_{\rr}(x,\snr{Du_{\rr}}))^{1-\omega_{i}}\dx\nonumber \\
&&+c\snr{h}^{2\alpha_{0}}\ti{\mathds{1}}\mathcal{M}^{2\ti{b}}\rr^{2\alpha}\int_{8B_{h}}\ell_{1}(\snr{Du_{\rr}})^{\gamma-\ti{\omega}}\dx,
\end{eqnarray}
where in particular the choice made in \eqref{ab0} yields equality $\alpha\beta_{0}=2(1-\beta_{0})$.
Let us glue estimates \eqref{6.21} via a dyadic covering argument. Specifically, we take a lattice $\mathcal{L}_{\snr{h}^{\beta_{0}}/\sqrt{n}}$ of open, disjoint cubes $\{Q_{\snr{h}^{\beta_{0}}/\sqrt{n}}(y)\}_{y\in (2\snr{h}^{\beta_{0}}/\sqrt{n})\mathbb{Z}^{n}}$. From this lattice, we pick $\tx{n}\approx_{n}\snr{h}^{-n\beta_{0}}$ cubes centered at points $\{x_{\ccc}\}_{\ccc\le \tx{n}}\subset (2\snr{h}^{\beta_{0}}/\sqrt{n})\mathbb{Z}^{n}$ such that $\snr{x_{\ccc}}\le 1/2+2\snr{h}^{\beta_{0}}$, thus determining the corresponding family $\{Q_{\ccc}\}_{\ccc\le \tx{n}}\equiv \{Q_{\snr{h}^{\beta_{0}}/\sqrt{n}}(x_{\ccc})\}_{\ccc\le \tx{n}}$. Observe that, in general, if $\snr{x}>1/2+2\snr{h}^{\beta_{0}}$, then $Q_{\snr{h}^{\beta_{0}}/\sqrt{n}}(x)\cap B_{1/2}=\emptyset$ as $Q_{\snr{h}^{\beta_{0}}/\sqrt{n}}(x)\subset B_{\snr{h}^{\beta_{0}}}(x)$ and $B_{\snr{h}^{\beta_{0}}}(x)\cap B_{1/2}=\emptyset$. We indeed have
\begin{flalign}\label{com.13}
\left| \  B_{1/2}\setminus \bigcup_{\ccc\le \tx{n}}Q_{\ccc} \ \right|=0,\qquad\qquad  Q_{\ccc_{1}}\cap Q_{\ccc_{2}}=\emptyset \ \Longleftrightarrow \ \ccc_{1}\not =\ccc_{2}.
\end{flalign}
Such a family of cubes corresponds to a family of balls $\{B_{\ccc}\}_{\ccc\le \tx{n}}:=\{B_{\snr{h}^{\beta_{0}}}(x_{\ccc})\}_{\ccc\le \tx{n}}$ in the sense that $Q_{\ccc}$ is the largest hypercube concentric to $B_{\ccc}$, with sides parallel to the coordinate axes. By construction $8B_{\ccc}\Subset B_{1}$ for all $\ccc\le \tx{n}$. Moreover, each of the dilated balls $8B_{\ccc_{t}}$ intersects the similar ones $8B_{\ccc_{s}}$, $\ccc_{t}\not =\ccc_{s}$ a finite, quantified number of times, depending only on $n$ (uniform finite intersection property). In fact, notice that the family of outer cubes $\{Q_{\snr{h}^{\beta_{0}}}(x_{\ccc})\}_{\ccc\le \tx{n}}$ has the same property and $B_{\ccc}\subset Q_{\snr{h}^{\beta_{0}}}(x_{\ccc})$. This yields:
\eqn{com.13.1}
$$
\sum_{\ccc\le \tx{n}}\phi(8B_{\ccc})\lesssim_{n}\phi(B_{1}),
$$
for every Borel measure $\phi$ defined on $B_{1}$. By \eqref{com.13} it turns out that also $\{B_{\ccc}\}_{\ccc\le \tx{n}}$ is a measure covering of $B_{1/2}$, i.e.:
\eqn{com.13.2}
$$
\left| \  B_{1/2}\setminus \bigcup_{\ccc\le \tx{n}}B_{\ccc} \ \right|=0.
$$
By \eqref{com.13.2}, \eqref{6.21} and \eqref{com.13.1} we obtain
\begin{eqnarray}\label{6.22}
\int_{B_{1/2}}\snr{\tau_{h}(\tx{a}_{\rr}(\cdot,\snr{Du_{\rr}})-\kk)_{+}}^{2}\dx&\stackrel{\eqref{com.13.2}}{\le}&\sum_{\ccc\le \tx{n}}\int_{B_{\ccc}}\snr{\tau_{h}(\tx{a}_{\rr}(x,\snr{Du_{\rr}})-\kk)_{+}}^{2}\dx\nonumber \\
&\stackrel{\eqref{6.21}}{\le}& c\snr{h}^{2\alpha_{0}}\mathcal{M}^{2\mathcal{b}_{0}}\sum_{\ccc\le\tx{n}}\int_{8B_{\ccc}}(\tx{a}_{\rr}(x,\snr{Du_{\rr}})-\kk)_{+}^{2}\dx\nonumber \\
&&+c\snr{h}^{2\alpha_{0}}\sum_{i=1}^{3}\mathds{1}_{i}\mathcal{M}^{2\mathcal{b}_{i}}\rr^{i\alpha}\sum_{\ccc\le\tx{n}}\int_{8B_{\ccc}}\ell_{1}(\tx{a}_{\rr}(x,\snr{Du_{\rr}}))^{1-\omega_{i}}\dx\nonumber \\
&&+c\snr{h}^{2\alpha_{0}}\ti{\mathds{1}}\mathcal{M}^{2\ti{\mathcal{b}}}\rr^{2\alpha}\sum_{\ccc\le \tx{n}}\int_{8B_{\ccc}}\ell_{1}(\snr{Du_{\rr}})^{\gamma-\ti{\omega}}\dx\nonumber \\
    &\stackrel{\eqref{com.13.1}}{\le}& c\snr{h}^{2\alpha_{0}}\mathcal{M}^{2\mathcal{b}_{0}}\int_{B_{1}}(\tx{a}_{\rr}(x,\snr{Du_{\rr}})-\kk)_{+}^{2}\dx\nonumber \\
&&+c\snr{h}^{2\alpha_{0}}\sum_{i=1}^{3}\mathds{1}_{i}\mathcal{M}^{2\mathcal{b}_{i}}\rr^{i\alpha}\int_{B_{1}}\ell_{1}(\tx{a}_{\rr}(x,\snr{Du_{\rr}}))^{1-\omega_{i}}\dx\nonumber \\
&&+c\snr{h}^{2\alpha_{0}}\ti{\mathds{1}}\mathcal{M}^{2\ti{\mathcal{b}}}\rr^{2\alpha}\int_{B_{1}}\ell_{1}(\snr{Du_{\rr}})^{\gamma-\ti{\omega}}\dx,
\end{eqnarray}
for $c\equiv c(\data)$. The availability of \eqref{6.22} allows applying Lemma \ref{fraim}, so that, after scaling back to $B_{\rr}(x_{0})$,
\begin{eqnarray}\label{66.0}
\nra{(\tx{a}(\cdot,\snr{Du})-\kk)_{+}}_{L^{2\chi}(B_{\rr/2}(x_{0}))}&\le& c\mathcal{M}^{\mathcal{b}_{0}}\nra{(\tx{a}(\cdot,\snr{Du})-\kk)_{+}}_{L^{2}(B_{\rr}(x_{0}))} \nonumber \\
&&+c\sum_{i=1}^{3}\mathds{1}_{i}\mathcal{M}^{\mathcal{b}_{i}}\rr^{\frac{i\alpha}{2}}\left(\mint_{B_{\rr}(x_{0})}\ell_{1}(\tx{a}(x,\snr{Du}))^{1-\omega_{i}}\dx\right)^{\frac{1}{2}}\nonumber \\
&&+c\ti{\mathds{1}}\mathcal{M}^{\ti{\mathcal{b}}}\rr^{\alpha}\left(\mint_{B_{\rr}(x_{0})}\ell_{1}(\snr{Du})^{\gamma-\ti{\omega}}\dx\right)^{\frac{1}{2}},
\end{eqnarray}
for all $\kk\ge 0$, with $\chi:=n/(n-2\beta)>1$ for all $\beta\in (0,\alpha_{0})$ --- say $\beta=\alpha_{0}/2$ to fix dependencies --- and $c\equiv c(\data)$.
\subsubsection*{Step 6: Lipschitz bounds} Let $B_{r}\subset B_{2r}\Subset B$ be any ball with radius $r\in (0,1]$, consider concentric balls $B_{r/2}\subseteq B_{\tau_{2}}\Subset B_{\tau_{1}}\subseteq B_{3r/4}$ and notice that there is no loss of generality in assuming that $\nr{\tx{a}(\cdot,\snr{Du})}_{L^{\infty}(B_{r/2})}\ge \tx{t}$, otherwise there would be nothing to prove. By \eqref{a12.x} and \eqref{areg}, all $x_{0}\in B_{\tau_{2}}$ are Lebesgue points for $\tx{a}(\cdot,\snr{Du})$, see \cite[Section 6.6]{dm21}, so we set $r_{0}:=(\tau_{1}-\tau_{2})/8$ so that $B_{2r_{0}}(x_{0})\Subset B_{\tau_{1}}$, and, via \eqref{66.0} applied on $B_{r_{0}}(x_{0})$ with $\mathcal{M}:=\nr{\tx{a}(\cdot,\snr{Du})}_{L^{\infty}(B_{\tau_{1}})}$ (that satisfies \eqref{mm}) we can apply Lemma \ref{revlem} choosing $k=4$, $\kk_{0}=0$, $w(x)=\tx{a}(x,\snr{Du(x)})$, $M_{0}= \mathcal{M}^{b_{0}}$, $M_{i}=\mathds{1}_{i}M^{\mathcal{b}_{i}}$ for $i\in \{1,2,3\}$, $M_{4}:=\ti{\mathds{1}}\mathcal{M}^{\ti{\mathcal{b}}}$, $f_{i}=\ell_{1}(\tx{a}(x,\snr{Du}))^{1-\omega_{i}}$ as $i\in \{1,2,3\}$, $f_{4}:=\ell_{1}(\snr{Du})^{\gamma-\ti{\omega}}$, $\sigma_{i}=i\alpha/2$ if $i\in \{1,2,3\}$, $\sigma_{4}=\alpha$, and $\vartheta_{i}=1/2$ for all $i\in \{1,\cdots,4\}$, to get
\begin{eqnarray*}
 \tx{a}(x_{0},\snr{Du(x_{0})})&\le&c\nr{\tx{a}(\cdot,\snr{Du})}_{L^{\infty}(B_{\tau_{1}})}^{\frac{\mathcal{b}_{0}\chi}{\chi-1}}\left(\mint_{B_{r_{0}}(x_{0})}\tx{a}(x,\snr{Du})^{2}\dx\right)^{\frac{1}{2}}+c\nonumber \\
&&+c\sum_{i=1}^{3}\mathds{1}_{i}\nr{\tx{a}(\cdot,\snr{Du})}_{L^{\infty}(B_{\tau_{1}})}^{\frac{\mathcal{b}_{0}}{\chi-1}+\mathcal{b}_{i}}\mathbf{P}^{\frac{1}{2}}_{\frac{i\alpha}{2}}\left(\ell_{1}(\tx{a}(\cdot,\snr{Du}))^{1-\omega_{i}};x_{0} ,\frac{\tau_{1}-\tau_{2}}{4}\right)\nonumber \\
&&+c\ti{\mathds{1}}\nr{\tx{a}(\cdot,\snr{Du})}_{L^{\infty}(B_{\tau_{1}})}^{\frac{\mathcal{b}_{0}}{\chi-1}+\ti{\mathcal{b}}}\mathbf{P}^{\frac{1}{2}}_{\alpha}\left(\ell_{1}(\snr{Du})^{\gamma-\ti{\omega}};x_{0} ,\frac{\tau_{1}-\tau_{2}}{4}\right),
\end{eqnarray*}
with $c\equiv c(\data)$. The arbitrariness of $x_{0}\in B_{\tau_{2}}$ grants
\begin{eqnarray}\label{14.0}
\nr{\tx{a}(\cdot,\snr{Du})}_{L^{\infty}(B_{\tau_{2}})}&\le&\frac{c\nr{\tx{a}(\cdot,\snr{Du})}_{L^{\infty}(B_{\tau_{1}})}^{\frac{\mathcal{b}_{0}\chi}{\chi-1}+\frac{1}{2}}}{(\tau_{1}-\tau_{2})^{n/2}}\left(\int_{B_{\tau_{1}}}\tx{a}(x,\snr{Du})\dx\right)^{\frac{1}{2}}+c\nonumber \\
&&+c\sum_{i=1}^{3}\mathds{1}_{i}\nr{\tx{a}(\cdot,\snr{Du})}_{L^{\infty}(B_{\tau_{1}})}^{\frac{\mathcal{b}_{0}}{\chi-1}+\mathcal{b}_{i}}\left\|\mathbf{P}^{\frac{1}{2}}_{\frac{i\alpha}{2}}\left(\ell_{1}(\tx{a}(\cdot,\snr{Du}))^{1-\omega_{i}};\ \cdot \ ,\frac{\tau_{1}-\tau_{2}}{4}\right)\right\|_{L^{\infty}(B_{\tau_{2}})}\nonumber \\
&&+c\ti{\mathds{1}}\nr{\tx{a}(\cdot,\snr{Du})}_{L^{\infty}(B_{\tau_{1}})}^{\frac{\mathcal{b}_{0}}{\chi-1}+\ti{\mathcal{b}}}\left\|\mathbf{P}^{\frac{1}{2}}_{\alpha}\left(\ell_{1}(\snr{Du})^{\gamma-\ti{\omega}};\ \cdot \ ,\frac{\tau_{1}-\tau_{2}}{4}\right)\right\|_{L^{\infty}(B_{\tau_{2}})},
\end{eqnarray}
for $c\equiv c(\data)$. At this stage, we need to reabsorbe the $L^{\infty}$-norms of $\tx{a}$ on the right-hand side of \eqref{14.0} and, simultaneously, to keep under control the nonlinear potentials. To do so, let us premise that, since all the constraints in \eqref{a.3}, \eqref{a.5.2s} are strict, there is no loss of generality in assuming $\alpha\in (0,1)$ --- otherwise one could just replace it in estimates \eqref{6.21}, \eqref{6.22} with an $\ti{\alpha}\in (0,\alpha)$ arbitrarily close to $\alpha$. Keeping this in mind, let us treat separately the subcritical and the supercritical cases.
\subsubsection*{Step 7: case $n\ge 3$ or $n=2$ and $0<\alpha<2/3$} Looking back at \eqref{1111}, here $\mathds{1}_{3}=1$ and $\ti{\mathds{1}}=0$, so \eqref{14.0} becomes
\begin{eqnarray}\label{14.1}
\nr{\tx{a}(\cdot,\snr{Du})}_{L^{\infty}(B_{\tau_{2}})}&\le&\frac{c\nr{\tx{a}(\cdot,\snr{Du})}_{L^{\infty}(B_{\tau_{1}})}^{\frac{\mathcal{b}_{0}\chi}{\chi-1}+\frac{1}{2}}}{(\tau_{1}-\tau_{2})^{n/2}}\left(\int_{B_{\tau_{1}}}\tx{a}(x,\snr{Du})\dx\right)^{\frac{1}{2}}+c\nonumber \\
&&+c\sum_{i=1}^{3}\nr{\tx{a}(\cdot,\snr{Du})}_{L^{\infty}(B_{\tau_{1}})}^{\frac{\mathcal{b}_{0}}{\chi-1}+\mathcal{b}_{i}}\left\|\mathbf{P}^{\frac{1}{2}}_{\frac{i\alpha}{2}}\left(\ell_{1}(\tx{a}(\cdot,\snr{Du}))^{1-\omega_{i}};\ \cdot \ ,\frac{\tau_{1}-\tau_{2}}{4}\right)\right\|_{L^{\infty}(B_{\tau_{2}})}.
\end{eqnarray}
Thanks to \eqref{atat} and to the choices made in \eqref{omega}--\eqref{mu*}, we immediately see that
\eqn{14.3}
$$
\frac{\mathcal{b}_{0}\chi}{\chi-1}+\frac{1}{2}<1,\qquad\quad  \frac{\mathcal{b}_{0}}{\chi-1}+\mathcal{b}_{i}<1,\qquad\quad  \frac{n}{i\alpha}>1,\qquad\quad  0<\frac{n(1-\omega_{i})}{i\alpha}<1,
$$
so we can apply Lemma \ref{crit} with $m=(1-\omega_{i})^{-1}$, $i\in \{1,2,3\}$, to bound
\eqn{14.4}
$$
\left\|\mathbf{P}^{\frac{1}{2}}_{\frac{i\alpha}{2}}\left(\ell_{1}(\tx{a}(\cdot,\snr{Du}))^{1-\omega_{i}};\ \cdot \ ,\frac{\tau_{1}-\tau_{2}}{4}\right)\right\|_{L^{\infty}(B_{\tau_{2}})}\le c\nr{\ell_{1}(\tx{a}(\cdot,\snr{Du}))}_{L^{1}(B_{r})}^{\frac{1-\omega_{i}}{2}},
$$
for $c\equiv c(n,\alpha,\vartheta_{*})$, use Young's inequality with conjugate exponents
\eqn{14.6}
$$
\left(\frac{2(\chi-1)}{2\mathcal{b}_{0}\chi+\chi-1},\frac{2(\chi-1)}{\chi-1-2\mathcal{b}_{0}\chi}\right),\qquad \quad \left(\frac{\chi-1}{\mathcal{b}_{0}+\mathcal{b}_{i}(\chi-1)},\frac{\chi-1}{(\chi-1)(1-\mathcal{b}_{i})-\mathcal{b}_{0}}\right),
$$
with $i\in \{1,2,3\}$, and conclude with
\begin{eqnarray*}
\nr{\tx{a}(\cdot,\snr{Du})}_{L^{\infty}(B_{\tau_{2}})}&\le&\frac{1}{4}\nr{\tx{a}(\cdot,\snr{Du})}_{L^{\infty}(B_{\tau_{1}})}+\frac{c}{(\tau_{1}-\tau_{2})^{\frac{n(\chi-1)}{\chi-1-2\mathcal{b}_{0}\chi}}}\left(\int_{B_{r}}\tx{a}(x,\snr{Du})\dx\right)^{\frac{\chi-1}{\chi-1-2\mathcal{b}_{0}\chi}}\nonumber \\
&&+c\sum_{i=1}^{3}\nr{\ell_{1}(\tx{a}(\cdot,\snr{Du}))}_{L^{1}(B_{r})}^{\frac{(1-\omega_{i})(\chi-1)}{2((\chi-1)(1-\mathcal{b}_{i})-\mathcal{b}_{0})}}+c,
\end{eqnarray*}
for $c\equiv c(\data)$. Lemma \ref{iterlem} eventually yields \eqref{lipb} below, and we are done.
\subsubsection*{Step 8: case $n=2$ and $\alpha\ge 2/3$} Now $\mathds{1}_{3}=0$, $\ti{\mathds{1}}=1$, and \eqref{14.0} reads as
\begin{eqnarray}\label{14.2}
\nr{\tx{a}(\cdot,\snr{Du})}_{L^{\infty}(B_{\tau_{2}})}&\le&\frac{c\nr{\tx{a}(\cdot,\snr{Du})}_{L^{\infty}(B_{\tau_{1}})}^{\frac{\mathcal{b}_{0}\chi}{\chi-1}+\frac{1}{2}}}{(\tau_{1}-\tau_{2})^{n/2}}\left(\int_{B_{\tau_{1}}}\tx{a}(x,\snr{Du})\dx\right)^{\frac{1}{2}}+c\nonumber \\
&&+c\sum_{i=1}^{2}\nr{\tx{a}(\cdot,\snr{Du})}_{L^{\infty}(B_{\tau_{1}})}^{\frac{\mathcal{b}_{0}}{\chi-1}+\mathcal{b}_{i}}\left\|\mathbf{P}^{\frac{1}{2}}_{\frac{i\alpha}{2}}\left(\ell_{1}(\tx{a}(\cdot,\snr{Du}))^{1-\omega_{i}};\ \cdot \ ,\frac{\tau_{1}-\tau_{2}}{4}\right)\right\|_{L^{\infty}(B_{\tau_{2}})}\nonumber \\
&&+c\nr{\tx{a}(\cdot,\snr{Du})}_{L^{\infty}(B_{\tau_{1}})}^{\frac{\mathcal{b}_{0}}{\chi-1}+\ti{\mathcal{b}}}\left\|\mathbf{P}^{\frac{1}{2}}_{\alpha}\left(\ell_{1}(\snr{Du})^{\gamma-\ti{\omega}};\ \cdot \ ,\frac{\tau_{1}-\tau_{2}}{4}\right)\right\|_{L^{\infty}(B_{\tau_{2}})}.
\end{eqnarray}
The terms in the first two lines of \eqref{14.2} can be controlled via \eqref{14.3}--\eqref{14.4} with $i\in \{1,2\}$, so we only need to take care of the last one. Notice that
$$
\eqref{a.5.2s}\ \ \mbox{and}\ \ \alpha\ge \frac{2}{3} \ \stackrel{\vartheta_{*}\le \gamma-1}{\Longrightarrow} \ \frac{\mathcal{b}_{0}}{\chi-1}+\ti{\mathcal{b}}\le \frac{\mathcal{b}_{0}}{\chi-1}+80\gamma\delta_{0}(\mu-1+\omega_{\mu})+\gamma-1+\frac{1+\ti{\omega}}{2}\stackrel{\eqref{omega},\eqref{bbb}}{<}1,
$$
so the exponents
\eqn{14.4.2}
$$
\left(\frac{\chi-1}{\mathcal{b}_{0}+\ti{\mathcal{b}}(\chi-1)},\frac{\chi-1}{(\chi-1)(1-\ti{\mathcal{b}})-\mathcal{b}_{0}}\right)
$$
are both finite and larger than one. Moreover,
\eqn{14.4.1}
$$
\begin{cases}
\displaystyle
\ \eqref{a.5.2s}\ \ \mbox{and}\ \ \alpha\ge \frac{2}{3} \ \Longrightarrow \ \frac{1}{\alpha}<\frac{2\gamma-1}{\gamma-\ti{\omega}}\vspace{1.5mm}\\
\displaystyle
\ \eqref{a.5.2s}\ \ \mbox{and}\ \ \eqref{muma} \ \Longrightarrow \ \gamma<2-\mu+\frac{\alpha}{2}\vspace{1.5mm}\\
\displaystyle
\ \eqref{a.5.2s} \ \ \mbox{and}\ \ \eqref{muma} \ \Longrightarrow \ 2\gamma-1<\frac{2-\mu}{1-\alpha},
\end{cases}
$$
and Lemma \ref{crit} yields
\eqn{14.5}
$$
\left\|\mathbf{P}^{\frac{1}{2}}_{\alpha}\left(\ell_{1}(\snr{Du})^{\gamma-\ti{\omega}};\ \cdot \ ,\frac{\tau_{1}-\tau_{2}}{4}\right)\right\|_{L^{\infty}(B_{\tau_{2}})}\le c\nr{\ell_{1}(\snr{Du})}_{L^{2\gamma-1}(B_{r})}^{\frac{\gamma-\ti{\omega}}{2}},
$$
for $c\equiv c(n,\alpha,\gamma)$. In \eqref{14.2}, we can then apply Young's inequality with conjugate exponents \eqref{14.6}$_{i\in \{1,2\}}$ and \eqref{14.4.2}, and plug in \eqref{14.5} and \eqref{12.2} --- recall that by $\eqref{14.4.1}_{2,3}$, Theorem \ref{thi} applies --- to obtain
\begin{eqnarray*}
\nr{\tx{a}(\cdot,\snr{Du})}_{L^{\infty}(B_{\tau_{2}})}&\le&\frac{1}{4}\nr{\tx{a}(\cdot,\snr{Du})}_{L^{\infty}(B_{\tau_{1}})}+\frac{c}{(\tau_{1}-\tau_{2})^{\frac{n(\chi-1)}{\chi-1-2\mathcal{b}_{0}\chi}}}\left(\int_{B_{r}}\tx{a}(x,\snr{Du})\dx\right)^{\frac{\chi-1}{\chi-1-2\mathcal{b}_{0}\chi}}\nonumber \\
&&+c\sum_{i=1}^{2}\nr{\ell_{1}(\tx{a}(\cdot,\snr{Du}))}_{L^{1}(B_{r})}^{\frac{(1-\omega_{i})(\chi-1)}{2((\chi-1)(1-\mathcal{b}_{i})-\mathcal{b}_{0})}}+c\nr{\ell_{1}(\snr{Du})}_{L^{2\gamma-1}(B_{r})}^{\frac{(\gamma-\ti{\omega})(\chi-1)}{2((\chi-1)(1-\ti{\mathcal{b}})-\mathcal{b}_{0})}}+c\nonumber \\
&\le&\frac{1}{4}\nr{\tx{a}(\cdot,\snr{Du})}_{L^{\infty}(B_{\tau_{1}})}+\frac{c}{(\tau_{1}-\tau_{2})^{\frac{n(\chi-1)}{\chi-1-2\mathcal{b}_{0}\chi}}}\left(\int_{B_{r}}\tx{a}(x,\snr{Du})\dx\right)^{\frac{\chi-1}{\chi-1-2\mathcal{b}_{0}\chi}}\nonumber \\
&&+\frac{c}{r^{\tx{d}}}\left(\nr{Du}_{L^{1}(B_{2r})}+\sqrt{\sigma_{\varepsilon}}\nr{Du}_{L^{4\gamma}(B_{2r})}^{2\gamma}+1\right)^{\tx{d}}\nonumber \\
&&+c\sum_{i=1}^{2}\nr{\ell_{1}(\tx{a}(\cdot,\snr{Du}))}_{L^{1}(B_{r})}^{\frac{(1-\omega_{i})(\chi-1)}{2((\chi-1)(1-\mathcal{b}_{i})-\mathcal{b}_{0})}}+c,
\end{eqnarray*}
with $c\equiv c(\data)$ and $\tx{d}\equiv \tx{d}(n,\alpha,\gamma,\mu)$. Recalling that by $\eqref{14.4.1}_{2}$, Theorem \ref{thi} applies, we achieve again \eqref{lipb} via Lemma \ref{iterlem}, \eqref{d2.3}$_{2}$, and $\eqref{a.7.1.x}_{1}$. Restoring the original notation, cf. Remark \ref{remnot}, we have shown the validity of the following proposition.
\begin{proposition}\label{lipp} There exists a threshold $\mu_{\textnormal{max}}\equiv \mu_{\textnormal{max}}(n,\alpha,\gamma,\vartheta_{*})>1$ such that if $1\le \mu<\mu_{\textnormal{max}}$, whenever $B_{r}\subset B_{2r}\Subset B$ are balls with radius $r\in (0,1]$, the solution $u_{\delta}^{\varepsilon}\in (\ti{u}_{\varepsilon}+W^{1,4\gamma}_{0}(B,\mathbb{R}^{N}))$ to the Dirichlet problem \eqref{pded} satisfies
\begin{flalign}\label{lipb}
\nr{\tx{a}_{\delta}^{\varepsilon}(\cdot,\snr{Du_{\delta}^{\varepsilon}})}_{L^{\infty}(B_{r/2})}\le \frac{c}{r^{\tx{d}}}\left(\int_{B_{2r}}A_{\delta}^{\varepsilon}(x,\snr{Du_{\delta}^{\varepsilon}})\dx+1\right)^{\tx{d}}
\end{flalign}
for $c\equiv c(\data)$ and $\tx{d}\equiv \tx{d}(n,\mu,\gamma,\vartheta_{*})$.
\end{proposition}
\noindent A standard covering argument, \eqref{lipb}, $\eqref{a.7.1.x}_{1}$, \eqref{adade}, \eqref{oe} and the minimality of $u_{\delta}^{\varepsilon}$ in the Dirichlet class $\ti{u}_{\varepsilon}+W^{1,4\gamma}_{0}(B,\mathbb{R}^{N})$ gives for any open ball $\ti{B}\Subset B$,
\begin{flalign}\label{lipbc}
\nr{\tx{a}_{\delta}^{\varepsilon}(\cdot,\snr{Du_{\delta}^{\varepsilon}})}_{L^{\infty}(\ti{B})}\le c\left(\nr{A_{\delta}^{\varepsilon}(\cdot,\snr{Du_{\delta}^{\varepsilon}})}_{L^{1}(B)}+1\right)^{\tx{d}}\le c\left(\nr{A_{\delta}(\cdot,\snr{D\ti{u}_{\varepsilon}})}_{L^{1}(B)}+\texttt{o}(\varepsilon)+1\right)^{\tx{d}}
\end{flalign}
with $c\equiv c(\data,\tx{l}(\ti{B},B))$ and $\tx{d}\equiv \tx{d}(n,\gamma,\mu,\vartheta_{*})$. This will be helpful to prove gradient H\"older continuity in the next section.
\section{Vectorial Schauder}\label{vs}
\noindent Let $\varepsilon,\delta\in (0,1/4)$ be the small parameters introduced at the beginning of Section \ref{dgc}. Fix now $\varepsilon\in (0,1/4)$, and take $\delta\in (0,\sigma_{\varepsilon})$, cf. \eqref{oe}. We can then bound
\begin{eqnarray}\label{goodb}
\nr{A_{\delta}(\cdot,\snr{D\ti{u}_{\varepsilon}})}_{L^{1}(B)}&\stackrel{\eqref{difdif}}{\le}&\nr{A(\cdot,\snr{D\ti{u}_{\varepsilon}})}_{L^{1}(B)}+c\delta\nr{\ell_{1}(\snr{D\ti{u}_{\varepsilon}})^{\gamma-1}}_{L^{1}(B)}\nonumber \\
&\le&\nr{A(\cdot,\snr{D\ti{u}_{\varepsilon}})}_{L^{1}(B)}+c\delta\nr{\ell_{1}(\snr{D\ti{u}_{\varepsilon}}^{2})^{2\gamma}}_{L^{1}(B)}\nonumber \\
&\stackrel{\eqref{conv}}{\le}&\mathcal{F}(u;B)+\frac{c\delta}{\sigma_{\varepsilon}}+\texttt{o}(\varepsilon),
\end{eqnarray}
with $c\equiv c(A,\gamma)$. Next, let $\ti{B}\Subset B$ be a ball and notice that, whenever $B_{r}\subseteq \ti{B}$, 
\begin{eqnarray}\label{gh.1}
\nr{Du_{\delta}^{\varepsilon}}_{L^{\infty}(B_{r})}&\stackrel{\eqref{a.7.1.x}_{4}}{\le}&\nr{\tx{a}_{\delta}^{\varepsilon}(\cdot,\snr{Du_{\delta}^{\varepsilon}})}_{L^{\infty}(B_{r})}^{\frac{1}{2-\mu}}\nonumber \\
&\stackrel{\eqref{lipbc}}{\le}&c\left(\nr{A_{\delta}(\cdot,\snr{D\ti{u}_{\varepsilon}})}_{L^{1}(B)}+1+\texttt{o}(\varepsilon)\right)^{\frac{\tx{d}}{2-\mu}}\nonumber \\
&\stackrel{\eqref{goodb}}{\le}&c\left(\mathcal{F}(u;B)+\texttt{o}(\varepsilon)+1\right)^{\frac{\tx{d}}{2-\mu}}\nonumber \\
&\stackrel{\eqref{oe}}{\le}&c\left(\mathcal{F}(u;B)+1\right)^{\frac{\tx{d}}{2-\mu}}=:\tx{c}_{B}\equiv \tx{c}_{B}(\data,\tx{l}(\ti{B},B),\mathcal{F}(u;B)).
\end{eqnarray}
for $\tx{d}\equiv \tx{d}(n,\gamma,\mu,\vartheta_{*})$. In the next lines, the constant $\tx{c}_{B}$ will  possibly be magnified by a multiplicative factor depending on $\data$, or raised to a positive power depending at most on $(n,\mu,\gamma,\vartheta_{*})$, but it will maintain the same increasing monotonicity with respect to $\mathcal{F}(u;B)$ --- we shall keep on denoting it $\tx{c}_{B}$. 
We then take a ball $B_{2\sigma}(\equiv B_{2\sigma}(x_{\textnormal{c}}))\subseteq \ti{B}/4$, and let $v\in u_{\delta}^{\varepsilon}+W^{1,4\gamma}_{0}(B_{2\sigma},\mathbb{R}^{N})$ be the solution to the Dirichlet problem
\eqn{pd00}
$$
u_{\delta}^{\varepsilon}+W^{1,4\gamma}_{0}(B_{2\sigma},\mathbb{R}^{N})\ni w\mapsto \min_{u_{\delta}^{\varepsilon}+W^{1,4\gamma}_{0}(B_{2\sigma},\mathbb{R}^{N})}\int_{B_{2\sigma}}\tx{f}_{\textnormal{c}}(Dw)\dx,
$$
solving by minimality
\eqn{el00}
$$
\int_{B_{2\sigma}}\langle\partial \tx{f}_{\textnormal{c}}(Dv
),Dw \rangle \dx=0\qquad \mbox{for all} \ \ w\in W^{1,4\gamma}_{0}(B_{2\sigma},\mathbb{R}^{N}),
$$
where, as before, $\tx{f}_{\textnormal{c}}(z)=\tx{f}_{\delta}^{\varepsilon}(x_{\textnormal{c}},z)$ for all $z\in \mathbb{R}^{N\times n}$. By Theorem \ref{t4.1} and \eqref{gh.1} it follows that
\eqn{hh.0}
$$
\tx{a}_{\textnormal{c}}(\snr{Dv})\in W^{1,2}(B_{\sigma},\mathbb{R}^{N\times n})\qquad \mbox{and}\qquad \nr{Dv}_{L^{\infty}(B_{3\sigma/2})}+\nr{Du_{\delta}^{\varepsilon}}_{L^{\infty}(B_{2\sigma})}\le \tx{c}_{B}.
$$
Let $M_{*}\equiv M_{*}(A,\gamma)$ be the constant from Corollary \ref{c39}, set $M:=\max\{\tx{c}_{B},M_{*}\}+1$ and let $\tx{f}_{M}(z):=A_{M}(x_{\textnormal{c}},\snr{z})$ be the integrand constructed in \eqref{amam}. By \eqref{h''}--\eqref{h'''}, \eqref{hh.0} and the definition of $M\equiv M(\data,\tx{l}(\ti{B},B),\mathcal{F}(u;B))$, we have that $v\in W^{1,\infty}(B_{3\sigma/2},\mathbb{R}^{N})$ solves
\eqn{hh.8}
$$
\int_{B_{3\sigma/2}}\langle\partial\tx{f}_{M}(Dv),Dw\rangle\dx\stackrel{M> \tx{c}_{B}}{=}\int_{B_{3\sigma/2}}\langle\partial\tx{f}_{\textnormal{c}}(Dv),Dw\rangle\dx\stackrel{\eqref{el00}}{=}0\qquad \mbox{for all} \ \ w\in W^{1,8\gamma}_{0}(B_{3\sigma/2},\mathbb{R}^{N}).
$$
The convexity of $z\mapsto \tx{f}_{M}(z)$ and \eqref{hh.8} imply that $v\in W^{1,\infty}(B_{3\sigma/2},\mathbb{R}^{N})$ is a local minimizer of integral
$$
W^{1,8\gamma}(B_{3\sigma/2},\mathbb{R}^{N})\ni w\mapsto \mathcal{F}_{M}(w;B_{3\sigma/2}):=\int_{B_{3\sigma/2}}\tx{f}_{M}(Dw)\dx. 
$$
By \eqref{h''}, $\tx{f}_{M}$ satisfies the assumptions of \cite[Theorem 1.1]{dsv09}, so there exists $\ti{\beta}\equiv \ti{\beta}(\data,\tx{l}(\ti{B},B),\mathcal{F}(u;B))\in (0,1)$ such that $Dv\in C^{0,\ti{\beta}}_{\loc}(B_{\sigma},\mathbb{R}^{N\times n})$ with
\eqn{oscv}
$$
\osc_{B_{\theta\sigma}}Dv\le c(\data,\tx{l}(\ti{B},B),\mathcal{F}(u;B))\theta^{\ti{\beta}}\qquad \mbox{for all} \ \ \theta\in (0,1),
$$
cf. \cite[Theorem 6.4 and Lemma 2.10]{dsv09}. Furthermore, estimate \eqref{com} with $B_{2\sigma}$ replacing $8B_{h}$ yields
\begin{eqnarray}\label{hh.48}
\nra{Du_{\delta}^{\varepsilon}-Dv}_{L^{2}(B_{\sigma})}&\stackrel{\eqref{hh.0}_{2}}{\le}&\tx{c}_{B}\nra{\ell_{1}(\snr{Du_{\delta}^{\varepsilon}}+\snr{Dv})^{-\mu/2}\snr{Du_{\delta}^{\varepsilon}-Dv}}_{L^{2}(B_{\sigma})}\nonumber \\
&\stackrel{\eqref{a.7.1.x}_{3}}{\le}&\tx{c}_{B}\nra{\sqrt{\lambda_{0}(\snr{Du_{\delta}^{\varepsilon}}+\snr{Dv})}\snr{Du_{\delta}^{\varepsilon}-Dv}}_{L^{2}(B_{2\sigma})}\nonumber \\
&\stackrel{\eqref{com}}{\le}&\tx{c}_{B}\sigma^{\frac{\alpha}{2}}\nra{1+A_{\delta}(\cdot,\snr{Du_{\delta}^{\varepsilon}})+A_{\textnormal{c}}(\snr{Du_{\delta}^{\varepsilon}})}_{L^{1}(B_{2\sigma})}^{\frac{1}{2}}\stackrel{\eqref{hh.0}_{2}}{\le}\tx{c}_{B}\sigma^{\frac{\alpha}{2}}.
\end{eqnarray}
We then bound, for every ball $B_{\rr}\Subset  B_{\sigma}$, $\rr\in (0,\sigma)$,
\begin{eqnarray*}
\nra{Du_{\delta}^{\varepsilon}-(Du_{\delta}^{\varepsilon})_{B_{\rr}}}_{L^{2}(B_{\rr})}&\le& 2\nra{Du_{\delta}^{\varepsilon}-Dv}_{L^{2}(B_{\rr})}+2\nra{Dv-(Dv)_{B_{\rr}}}_{L^{2}(B_{\rr})}\nonumber \\
&\stackrel{\eqref{oscv}}{\le}&c\left(\sigma/\rr\right)^{\frac{n}{2}}\nra{Du_{\delta}^{\varepsilon}-Dv}_{L^{2}(B_{\sigma})}+c\left(\rr/\sigma\right)^{\ti{\beta}}\stackrel{\eqref{hh.48}}{\le}c\left(\sigma/\rr\right)^{\frac{n}{2}}\sigma^{\frac{\alpha}{2}}+c\left(\rr/\sigma\right)^{\ti{\beta}},
\end{eqnarray*}
for $c\equiv c(\data,\tx{l}(\ti{B},B),\mathcal{F}(u;B))$. Let us equalize the powers by choosing $\rr=(\sigma/2)^{\frac{n+2\ti{\beta}+\alpha}{n+2\ti{\beta}}}$ to gain
\eqn{hol.1}
$$
\nra{Du_{\delta}^{\varepsilon}-(Du_{\delta}^{\varepsilon})_{B_{\rr}}}_{L^{2}(B_{\rr})}\le c(\data,\tx{l}(\ti{B},B),\mathcal{F}(u;B))\rr^{\frac{\alpha\ti{\beta}}{n+2\ti{\beta}+\alpha}}.
$$
We point out that if $\sigma\le \rr\le 2\sigma$, estimate \eqref{hol.1} trivially follows. Set $\beta_{*}:=\alpha\ti{\beta}/(n+2\ti{\beta}+\alpha)$, so that $\beta_{*}\equiv \beta_{*}(\data,\tx{l}(\ti{B},B),\mathcal{F}(u;B))\in (0,1)$. A standard covering argument and Campanato-Meyers characterization of H\"older continuity yield that $Du_{\delta}^{\varepsilon}$ is locally H\"older continuous in $\ti{B}$, and, thanks to the arbitrariness of $\ti{B}\Subset B$ we further deduce the local H\"older continuity of $Du_{\delta}^{\varepsilon}$ in $B$. Specifically, given any ball $\ti{B}\Subset B$, $[Du_{\delta}^{\varepsilon}]_{0,\beta_{*};\ti{B}}\le c(\data(B),\tx{l}(\ti{B},B),\mathcal{F}(u;B))$. Overall, we have just proven the following proposition.
\begin{proposition}\label{holp}
Let $u_{\delta}^{\varepsilon}\in \ti{u}_{\varepsilon}+W^{1,4\gamma}_{0}(B,\mathbb{R}^{N})$ be the solution to the Dirichlet problem\footnote{Here we fixed $\varepsilon\in (0,1/4)$, and imposed restriction $\delta\in (0,\sigma_{\varepsilon})$, see the beginning of Section \ref{vs}.} \eqref{pded}. There exists a limiting parameter $\mu_{\textnormal{max}}\equiv \mu_{\textnormal{max}}(n,\alpha,\gamma,\vartheta_{*})>1$ such that if $1\le \mu<\mu_{\textnormal{max}}$, then $Du_{\delta}^{\varepsilon}$ is locally H\"older continuous in $B$. Specifically, given any ball $\ti{B}\Subset B$, $[Du_{\delta}^{\varepsilon}]_{0,\beta_{*};\ti{B}}\le c$, with H\"older exponent $\beta_{*}\equiv \beta_{*}(\data,\tx{l}(\ti{B},B),\mathcal{F}(u;B))\in (0,1)$, and bounding constant $c\equiv c(\data,\tx{l}(\ti{B},B),\mathcal{F}(u;B))$.
\end{proposition}
\section{Approximation scheme and proof of Theorem \ref{mt}}\label{asas} \noindent In this section we show that the sequence of approximating minimizers constructed in Section \ref{rdp} converges in a uniform fashion to our original minimum $u\in W^{1,1}_{\loc}(\Omega,\mathbb{R}^{N})$, thus transferring to $u$ the Lipschitz and H\"older bounds obtained in Propositions \ref{lipp} and \ref{holp}, respectively. 
\subsubsection*{Convergence of approximating minimizers} Keeping in mind the restrictions imposed at the beginning of Section \ref{vs}, by the minimality of $u_{\delta}^{\varepsilon}$ in the Dirichlet class $\ti{u}_{\varepsilon}+W^{1,4\gamma}_{0}(B,\mathbb{R}^{N})$, \eqref{goodb} and \eqref{oe} we obtain
\begin{flalign}\label{enesed}
\mathcal{F}_{\delta}^{\varepsilon}(u_{\delta}^{\varepsilon};B)\le \nr{A_{\delta}(\cdot,\snr{D\ti{u}_{\varepsilon}})}_{L^{1}(B)}+\sigma_{\varepsilon}\nr{\ell_{1}(\snr{D\ti{u}_{\varepsilon}}^{2})}_{L^{2\gamma}(B)}^{2\gamma}\le \mathcal{F}(u;B)+\frac{c\delta}{\sigma_{\varepsilon}}+\texttt{o}(\varepsilon),
\end{flalign}
for $c\equiv c(A,\gamma)$. Moreover,
\begin{eqnarray}\label{enesed.1}
\mathcal{F}(u_{\delta}^{\varepsilon};B)+\sigma_{\varepsilon}\nr{\ell_{1}(\snr{Du_{\delta}^{\varepsilon}}^{2})}_{L^{2\gamma}(B)}^{2\gamma}&\stackrel{\eqref{difdif}}{\le}& \mathcal{F}_{\delta}^{\varepsilon}(u_{\delta}^{\varepsilon};B)+c\delta\nr{\ell_{1}(\snr{Du_{\delta}^{\varepsilon}})^{\gamma-1}}_{L^{1}(B)}\nonumber \\
&\le&\mathcal{F}_{\delta}^{\varepsilon}(u_{\delta}^{\varepsilon};B)+c\delta\nr{\ell_{1}(\snr{Du_{\delta}^{\varepsilon}}^{2})}_{L^{2\gamma}(B)}^{2\gamma}\nonumber \\
&\le&\mathcal{F}_{\delta}^{\varepsilon}(u_{\delta}^{\varepsilon};B)\left(1+\frac{c\delta}{\sigma_{\varepsilon}}\right)\stackrel{\eqref{enesed}}{\le}\mathcal{F}(u;B)\left(1+\frac{c\delta}{\sigma_{\varepsilon}}\right)+\frac{c\delta}{\sigma_{\varepsilon}}+\texttt{o}(\varepsilon),
\end{eqnarray}
with $c\equiv c(A,\gamma)$. Estimate \eqref{enesed.1} implies that, for fixed $\varepsilon>0$, the sequence $\{u_{\delta}^{\varepsilon}\}_{\delta>0}$ converges to some $u^{\varepsilon}\in \ti{u}_{\varepsilon}+W^{1,4\gamma}_{0}(B,\mathbb{R}^{N})$ weakly in $W^{1,4\gamma}(B,\mathbb{R}^{N})$. As $\delta\to 0$, by weak lower semicontinuity we derive
\begin{eqnarray}\label{enesed.2}
\mathcal{F}(u^{\varepsilon};B)+\sigma_{\varepsilon}\nr{\ell_{1}(\snr{Du^{\varepsilon}}^{2})}_{L^{2\gamma}(B)}^{2\gamma}&\le& \liminf_{\delta\to 0}\left(\mathcal{F}(u_{\delta}^{\varepsilon};B)+\sigma_{\varepsilon}\nr{\ell_{1}(\snr{Du_{\delta}^{\varepsilon}}^{2})}_{L^{2\gamma}(B)}^{2\gamma}\right)\nonumber \\
&\stackrel{\eqref{enesed.1}}{\le}&\mathcal{F}(u;B)+\texttt{o}(\varepsilon).
\end{eqnarray}
Next, notice that
$$
c\nr{b(\snr{Du^{\varepsilon}})}\stackrel{\eqref{a.4}_{4}}{\le}\mathcal{F}(u^{\varepsilon};B)+\snr{B}\stackrel{\eqref{enesed.2}}{\le}\mathcal{F}(u;B)+\snr{B}+\texttt{o}(\varepsilon),
$$
for $c\equiv c(A)$, so, recalling \eqref{binf}, via Dunford \& Pettis and de la Vall\'ee Poussin theorems we deduce that, up to subsequences, $u^{\varepsilon}\rightharpoonup \ti{u}$ weakly in $W^{1,1}(B,\mathbb{R}^{N})$ for some $\ti{u}\in u+W^{1,1}(B,\mathbb{R}^{N})$. Letting $\varepsilon\to 0$ in \eqref{enesed.2} and using weak lower semicontinuity we obtain
$$
\mathcal{F}(\ti{u};B)\le \liminf_{\varepsilon\to 0}\left(\mathcal{F}(u^{\varepsilon};B)+\sigma_{\varepsilon}\nr{\ell_{1}(\snr{Du^{\varepsilon}}^{2})}_{L^{2\gamma}(B)}^{2\gamma}\right)\stackrel{\eqref{enesed.2}}{\le}\mathcal{F}(u;B).
$$
The content of the previous display, together with the minimality of $u$, the fact that $\left.\ti{u}\right|_{\partial B}=\left.u\right|_{\partial B}$ and the strict convexity of $\tx{f}$ yield that $\ti{u}=u$ on $B$. Summarizing, we have just proven that, up to (nonrelabelled) subsequences,
\eqn{convconv}
$$
u_{\delta}^{\varepsilon}\rightharpoonup u^{\varepsilon} \ \ \mbox{weakly in} \ \ W^{1,4\gamma}(B,\mathbb{R}^{N})\quad \mbox{and}\quad u^{\varepsilon}\rightharpoonup u\ \ \mbox{weakly in} \ \ W^{1,1}(B,\mathbb{R}^{N}).
$$
\subsubsection*{Proof of Theorem \ref{mt}} Let $\ti{B}\Subset B$ be a ball. The bounds from Propositions \ref{lipp} and \ref{holp} now read as
\eqn{bb.1}
$$
\nr{Du_{\delta}^{\varepsilon}}_{L^{\infty}(\ti{B})}+[Du_{\delta}^{\varepsilon}]_{0,\beta_{*};\ti{B}}\le c,
$$
where $c\equiv c(\data,\tx{l}(\ti{B},B),\mathcal{F}(u;B))$ and $\beta_{*}\equiv \beta_{*}(\data,\tx{l}(\ti{B},B),\mathcal{F}(u;B))$. We can then update (locally) the limits in \eqref{convconv} as $u_{\delta}^{\varepsilon}\to u^{\varepsilon}$ weak* in $W^{1,\infty}(\ti{B},\mathbb{R}^{N})$, and uniformly in $C^{1,\sigma}(\ti{B},\mathbb{R}^{N})$ for all $\sigma\in (0,\beta_{*})$ by \eqref{bb.1}, so we can pass to the limit as $\delta\to 0$ in \eqref{bb.1} to gain
\eqn{bb.2}
$$
\nr{Du^{\varepsilon}}_{L^{\infty}(\ti{B})}+[Du^{\varepsilon}]_{0,\beta_{*};\ti{B}}\le c,
$$
with $c\equiv c(\data,\tx{l}(\ti{B},B),\mathcal{F}(u;B))$ and $\beta_{*}\equiv \beta_{*}(\data,\tx{l}(\ti{B},B),\mathcal{F}(u;B))$. We can then send $\varepsilon\to 0$ in \eqref{bb.2} and recall \eqref{convconv} to complete the proof.

\end{document}